  \documentclass{preprint}
  \pdfoutput=1
\usepackage{amsfonts,amsmath, times,pdfpages}
\usepackage {graphicx}
\usepackage{mhequ}
\usepackage{appendix}
\usepackage{theorem}

\newtheorem{theorem}{Theorem}[section]
\newtheorem{lemma}[theorem]{Lemma}
\newtheorem{proposition}[theorem]{Proposition} 
\newtheorem{corollary}[theorem]{Corollary}

{\theorembodyfont{\rmfamily}
\newtheorem{remark}[theorem]{Remark}
\newtheorem{definition}{Definition}[section]
\newtheorem{example}[theorem]{Example}

}

\numberwithin{equation}{section}

\begin{document}

\newcommand{\tmname}[1]{\textsc{#1}}
\newcommand{\tmop}[1]{\operatorname{#1}}
\newcommand{\tmsamp}[1]{\textsf{#1}}
\newenvironment{enumerateroman}{\begin{enumerate}[i.]}{\end{enumerate}}
\newenvironment{enumerateromancap}{\begin{enumerate}[I.]}{\end{enumerate}}

\newcounter{problemnr}
\setcounter{problemnr}{0}
\newenvironment{problem}{\medskip
  \refstepcounter{problemnr}\small{\bf\noindent Problem~\arabic{problemnr}\ }}{\normalsize}
\newenvironment{enumeratealphacap}{\begin{enumerate}[A.]}{\end{enumerate}}
\newcommand{\tmmathbf}[1]{\boldsymbol{#1}}

\def\paral{/\kern-0.55ex/}
\def\parals_#1{/\kern-0.55ex/_{\!#1}}
\def\bparals_#1{\breve{/\kern-0.55ex/_{\!#1}}}
\def\n#1{|\kern-0.24em|\kern-0.24em|#1|\kern-0.24em|\kern-0.24em|}

\newcommand{\A}{{\bf \mathcal A}}
\newcommand{\B}{{\bf \mathcal B}}
\def\C{\mathbb C}
\newcommand{\D}{{\rm I \! D}}
\newcommand{\dom}{{\mathcal D}om}
\newcommand{\pathR}{{\mathcal{\rm I\!R}}}
\newcommand{\Nabla}{{\bf \nabla}}
\newcommand{\E}{{\mathbf E}}
\newcommand{\Epsilon}{{\mathcal E}}
\newcommand{\F}{{\mathcal F}}
\newcommand{\G}{{\mathcal G}}
\def\g{{\mathfrak g}}
\newcommand{\HH}{{\mathbb H}}
\def\h{{\mathfrak h}}
\def\k{{\mathfrak k}}
\newcommand{\I}{{\mathcal I}}
\def\LL{{\mathbb L}}
\def\law{\mathop{\mathrm{ Law}}}
\def\m{{\mathfrak m}}
\newcommand{\K}{{\mathcal K}}
\newcommand{\p}{{\mathfrak p}}
\newcommand{\R}{{\mathbb R}}
\newcommand{\Rc}{{\mathcal R}}
\def\T{{\mathcal T}}
\def\M{{\mathcal M}}
\def\N{{\mathcal N}}
\newcommand{\pnabla}{{\nabla\!\!\!\!\!\!\nabla}}
\def\X{{\mathbb X}}
\def\Y{{\mathbb Y}}
\def\L{{\mathcal L}}
\def\1{{\mathbf 1}}
\def\half{{ \frac{1}{2} }}
\def\f{\frac}
\def\vol{{\mathop {\rm vol}}}

\def\ad{{\mathop {\rm ad}}}
\def\Conj{{\mathop {\rm Ad}}}
\def\Ad{{\mathop {\rm Ad}}}
\newcommand{\const}{\rm {const.}}
\newcommand{\eg}{\textit{e.g. }}
\newcommand{\as}{\textit{a.s. }}
\newcommand{\ie}{\textit{i.e. }}
\def\s.t.{\mathop {\rm s.t.}}
\def\esssup{\mathop{\rm ess\; sup}}
\def\Ric{\mathop{\rm Ric}}
\def\div{\mathop{\rm div}}
\def\grad{\mathop{\rm grad}}
\def\ker{\mathop{\rm ker}}
\def\Hess{\mathop{\rm Hess}}
\def\Image{\mathop{\rm Image}}
\def\Dom{\mathop{\rm Dom}}
\def\id{\mathop {\rm Id}}
\def\Image{\mathop{\rm Image}}
\def\Cyl{\mathop {\rm Cyl}}
\def\Conj{\mathop {\rm Conj}}
\def\Span{\mathop {\rm Span}}
\def\trace{\mathop{\rm trace}}
\def\ev{\mathop {\rm ev}}
\def\Ent{\mathop {\rm Ent}}
\def\tr{\mathop {\rm tr}}
\def\graph{\mathop {\rm graph}}
\def\loc{\mathop{\rm loc}}
\def\so{{\mathfrak {so}}}
\def\su{{\mathfrak {su}}}
\def\u{{\mathfrak {u}}}
\def\o{{\mathfrak {o}}}
\def\pp{{\mathfrak p}}
\def\gl{{\mathfrak gl}}
\def\hol{{\mathfrak hol}}
\def\z{{\mathfrak z}}
\def\t{{\mathfrak t}}
\def\<{\langle}
\def\>{\rangle}
\def\span{{\mathop{\rm span}}}
\def\diam{\mathrm {diam}}
\def\inj{\mathrm {inj}}
\def\Lip{\mathrm {Lip}}
\def\Iso{\mathrm {Iso}}
\def\Osc{\mathop{\rm Osc}}
\renewcommand{\thefootnote}{}

\def\paral{/\kern-0.55ex/}
\def\parals_#1{/\kern-0.55ex/_{\!#1}}
\def\bparals_#1{\breve{/\kern-0.55ex/_{\!#1}}}
\def\n#1{|\kern-0.24em|\kern-0.24em|#1|\kern-0.24em|\kern-0.24em|}
\newenvironment{proof}{
 \noindent\textbf{Proof}\ }{\hspace*{\fill}
  \begin{math}\Box\end{math}\medskip}

\author{Xue-Mei Li}
\date{}
\title{Homogenisation on Homogeneous Spaces}
\institute{
Department of Mathematics, Imperial College London, London SW7, 2AZ, U.K.}
 
\maketitle

\begin{abstract}

Motivated by collapsing of Riemannian manifolds and inhomogeneous scaling of left invariant Riemannian metrics on a real Lie group~$G$ with a sub-group~$H$,
we introduce a family of interpolation equations on $G$ with  a parameter $\epsilon>0$, interpolating  
 hypo-elliptic diffusions on $H$ and translates of exponential maps on $G$ and examine the dynamics as $\epsilon\to 0$.  When  $H$ is compact, we use
the reductive homogeneous structure of Nomizu to extract a converging family of stochastic processes (converging on the time scale~$\f 1 \epsilon$), proving  the convergence of the  stochastic dynamics on the orbit spaces $G/H$ and their parallel translations, providing also an estimate on the rate of the convergence in the Wasserstein distance. 
  Their limits are   not necessarily Brownian motions and are classified algebraically by a Peter Weyl's theorem for real Lie groups and geometrically using a weak notion of the
 naturally reductive property; the classifications allow to conclude the Markov property of the limit process. This can be considered
as ``taking the adiabatic limit'' of the differential operators $\L^\epsilon=\f 1 {\epsilon} \sum_k (A_k)^2+ \f 1{\epsilon} A_0+ Y_0$  where $Y_0, A_k$ are left invariant vector fields  and $\{A_k\}$ generate the Lie-algebra of~$H$.
\end{abstract}
\\[1em]

\\
{Key words: stochastic averaging, diffusion creation, adiabatic  limit, H\"ormander's conditions, classification of  effective dynamics.

\tableofcontents

\section{Introduction} 

 By deforming the fibres of the Hopf fibration, the canonical round metric on the Lie group $S^3$ gives rise to a family of left invariant Berger's metrics which we denote by $m_\epsilon$. As $\epsilon$ approaches zero, the Riemannian manifolds
 $(S^3, m_\epsilon)$ converge to the lower dimensional two sphere $S^2({1\over 2})$ of radius ${1\over 2}$, while keeping the sectional curvatures  bounded.   See  \cite[J. Cheeger, M. Gromov]{Cheeger-Gromov},  \cite[K. Fukaya]{Fukaya87}, and  \cite[A. Kasue, H. Kumura]{Kasue-Kumura} for other types of convergences of Riemannian manifolds.
  Let us denote by 
 $\Delta_{S^3}^\epsilon$ and $\Delta_{S^1}$ the Laplacians on $(S^3, m_\epsilon)$ and on $S^1$ respectively, and also
 denote by $\Delta^h$ the horizontal Laplacian identified with the Laplacian on $S^2(\f 12)=S^3/S^1$.
These operators commute and $\Delta^\epsilon _{S^3} ={1\over \epsilon}\Delta_{S^1} + \Delta^h$. If $\{X_1, X_2, X_3\}$ are the Pauli matrices, identified with left invariant vector fields, then $\Delta_{S^1}=(X_1)^2$, $\Delta^h=(X_2)^2+(X_3)^2$.
 As $\epsilon$ approaches $0$, any eigenvalues of the Laplacian $\Delta^\epsilon _{S^3} $ coming from a non-zero eigenvalue of $\f 1 \epsilon  \Delta_{S^1} $ is pushed to the back of the spectrum and an eigenfunction of $\Delta^\epsilon_{S^3}$, not constant in the fibre, 
flies away. In other words the spectrums of ${S^3}$ converge to that of $S^2$,  see \cite[S. Tanno]{Tanno-79},  \cite[L. B\'erard-Bergery and J. -P. Bourghignon]{Berard-Bergery-Bourguignon} \cite[H. Urakawa]{Urakawa86} for discussions on the spectrum of Laplacians on spheres, on homogeneous Riemannian manifolds and on Riemannian submersions with totally geodesic fibres. 

Another interesting family of operators is  $\{\gamma \Delta_{S^1}+\delta Y_0: \gamma>0, \delta>0\}$
where $Y_0=aX_2+bX_3$ is a non-zero unit length left invariant vector field.  If $\gamma=0$, the time evolutions associated with the operators
are geodesics on $S^2$; if $\delta=0$ the time evolutions with initial values in $H$ are scaled Brownian motions on $H$.
Let us take $\gamma\to \infty$ and keep $\delta=1$. In other words we consider
$\L^\epsilon=\f 1 {\epsilon}\Delta_{S^1}+Y_0$ where $\epsilon>0$. Unlike  $\Delta^\epsilon _{S^3}$, the summands of $L^\epsilon$
do not commute, it is nevertheless easy to see that the orbits of the corresponding random evolutions exhibit interesting asymptotics.
For this observe that the Hopf map on the Hopf fibration is a conservation law,  this allows us to construct a converging family of first order random linear differential operators $\tilde\L^\epsilon$ ( the evaluations of the Hopf map of  the evolutions of $L^\epsilon$ and $\tilde\L^\epsilon$ agree)
such that their time evolutions converge to a Markov process with effective limit  $\lambda \Delta^h$ for an explicit constant $\lambda$. See Example \ref{Hopf-theorem}, and also see \cite[Li]{Li-OM-1} where a related family of operators were studied.   We study such phenomena on  homogeneous manifolds.  

Let $G$ be a smooth connected real Lie group, not necessarily compact, with a non-trivial closed proper subgroup $H$. Denote by $\g$ and $\h$ their respective Lie algebras.  We  assume that $H$ is connected, and compact for certain type of results and in this introduction. 
 If $H$ is compact, there exists an $\Ad_H$-invariant inner product on $\g$ which descends to $T_oM$ and induces a $G$-invariant Riemmanian metric on $M$. 
 
 For a positive real number $\epsilon$, let us scale the inner product on $\h$ by a number $\epsilon>0$. Let $\{X_1, \dots, X_p, X_{p+1}, \dots, X_n\}$ be an orthonormal basis of $\g$ where $\{X_j, j\le p\}$ is a basis of $\h$.
  By declaring the left invariant vector fields $\{\f 1 {\sqrt{\epsilon}} X_1, \dots, \f 1 {\sqrt{\epsilon}} X_p, X_{p+1}, \dots, X_n\}$ an orthonormal frame we obtain a family of Riemannian metrics. Let us consider the sum of the squares of vector fields operators, arising from the non-homogeneous scaling,
 $$ \L^\epsilon=\f 1 {\epsilon}\L_0+ Y_0, \quad \text{where} \quad  \L_0= \f 12\sum (A_k)^2+ A_0,$$
 $\{A_k\in \h\}$ are bracket generating left invariant vector fields, and $Y_0$ is a left invariant vector field to be specified later. Similar to the Hopf fibration, a family of interpolation equations can be associated with $\L_0$ and $Y_0$. We study the asymptotics of  the family of operators $ \L^\epsilon$ on $[0,1]$ and on $[0, \f 1\epsilon]$. Observe that $\L_0$ is not necessarily symmetric nor hypoelliptic in $G$.

 We first consider $\L^\epsilon$ as a perturbation to the hypoelliptic diffusions induced by $\f 1 {\epsilon}\L_0$.  The hypoelliptic diffusion processes stay in the orbits/fibres of their initial values while the perturbation in $Y_0$ direction induces motions transversal to the fibres, describing a conservation law.
Our first task is to use it to understand the nature of the perturbation and to extract from them a family of first order random differential operators, $ \tilde \L^\epsilon$,   called the `slow motions'. If the slow motions converge to a fixed point, we study the dynamics on the longer time scale $[0, \f 1 \epsilon]$. On the Riemannian homogeneous manifold, for $Y_0$ appropriately chosen, the `effective limit' on $[0, 1]$ is either a non-degenerate ODE or a fixed point. In the latter case and on the scale of $\f 1 \epsilon$, we would however treat the term involving $\L_0$ as a perturbation. It is counter intuitive to consider the dominate part as perturbation to the smaller part, but the perturbation, although very large in magnitude, is fast oscillating. The large oscillating motion will be averaged out, leaving an effective stochastic motion corresponding to a second order differential operator on~$G$.

The singular perturbation problem described above can  be considered as a stochastic homoganization problem for the following family of stochastic differential equations (SDEs) on~$G$,
\begin{equs}
\label{1}dg_t^\epsilon ={1\over \sqrt \epsilon} \sum_{k=1}^N   A_k(g_t^\epsilon) \circ db_t^k +{1\over \epsilon} A_0(g_t^\epsilon)dt
+Y_0(g_t^\epsilon) dt, \quad g_0^\epsilon =g_0,
\end{equs}
where $\circ$ denotes Stratonovich integrals.
These SDEs  belong to a family of equations, see \S \ref{section-interpolation},  which interpolate between translates of  one parameter subgroups of $G$ and hypoelliptic diffusion on  $H$. Scaled by $1/\epsilon$, the Markov generator of  $(g_{t/\epsilon}^\epsilon)$ is precisely~$\f 1\epsilon \L^{\epsilon}$.   

One of our tools is a {\it reductive structure}, in the sense of Nomizu \cite{Nomizu54},  i.e. an $\Ad_H$-invariant sub-space $\m$ of $\g$ complementing $\h$, not in the sense of  having a completely irreducible adjoint representation.  Such a complement $\m$ exists if $H$ is compact,  see  \S\ref{section2} for further discussion where a non-Riemannian example is also given. If $Y_0$ belongs to
 a non-trivial  irreducible invariant subspace $\m_k$ of $\m$,  the time scale for taking the limit is determined by $\m_k$, see \S\ref{section-reduction}-\S \ref{sectionConditions}.  Indeed, the effective limits  can  be classified using an orthogonal decomposition of $\m$ into invariant subspaces, see \S \ref{section5}-\S\ref{section8}.

   The reductive structure allows us to use the projection $\pi:G\to G/H$ as a `conservation law', the projected process is a slow variable on the orbit space, and 
  $(g_t^\epsilon)$ is a {\it perturbation to the conservation law} and we expect that the variable in the orbit space converge over a long time scale (equation (\ref{1})  is already on this large time scale.  To separate the slow and the fast variables, we `horizontally lift back' the projected curves and obtain a slow variable on the total space $G$, the fast variable is a diffusion process on the subgroup $H$. The `horizontal lifts' of the projected processes are random smooth curves driven by random ODEs whose generator will be denoted by $\tilde \L^\epsilon$, for almost surely all sample paths they are parallel translations along the projection. It is easy to see that they converge to the solution of an ODE: this is the averaging principle. We then identify those $Y_0$ for which the limiting ODE is trivial and in particular the projections of their solution curves are fixed points in the orbit space $G/H$. For such $Y_0$ we study the dynamics in a longer time interval  and conclude that 
   the family of random differentiable curves, determined by the family of random ODEs, converge to the sample paths of a non-trivial Markov process, i.e. diffusion reaction, the proof for this convergence is  based on a theorem in \cite{Li-limits}; we do however need to take care of the non-compactness of $G$ which is not entirely covered by  \cite{Li-limits}. In terms of the stochastic processes on the orbit manifolds, we prove their convergence on the interval of size $[0, \f \1 \epsilon]$  together with the convergence of their horizontal lifts, i.e.  parallel translations along them converge to the stochastic parallel translation along the limiting diffusion process. The question  whether or not the latter is a Markov process, on its own, will also be studied.
The convergence will be in terms of the probability distributions induced by $\f \1 \epsilon \tilde \L_\epsilon$ in the weak topology
and Wasserstein distance on the space $C([0,1]; G)$ of continuous paths over $G$ and over $G/H$, see \S\ref{section-convergence}, where the rate of convergence in the Wasserstein distance for their evaluations at specific times is  given.

A heuristic argument based on multi-scale analysis, which we give shortly, appears to suggest the following  centring condition:
 $\int_H Y_0F_0 dh=0$ where $F_0$ is a solution to the effective parabolic equation and where $dh$ is the right invariant Haar measure on $H$. This is not quite the right assumption, we will assume instead that 
 $Y_0\not =0$,  $\bar Y_0\equiv \int_H \Ad(h)(Y_0)dh=0$. 
 Indeed, we show that  $\f \1 \epsilon \tilde \L_\epsilon=\Ad(h_{\f t\epsilon}(\omega)) (Y_0)$, where $(h_t)$ is
a stochastic process in $H$; and that $\bar Y_0=0$ precisely when  $Y_0$ is orthogonal to $\h\oplus \m_0$, where $\m_0$ is the set of $\Ad_H$-invariant vectors in $\m$ (\S\ref{sectionConditions}).

{\it The heuristic Argument.}  Let us give the heuristics argument, made rigorous in \S\ref{section-reduction}-\S\ref{section-convergence}.  
For simplicity take $A_0=0$ and denote by $dh$ the normalised Haar measure on a compact $H$.
 Let $\L_0 ={1\over 2}\sum_{k=1}^p (A_k)^2$ and let $F: \R_+\times G\to \R$ be a solution to the evolution equation $\f {\partial F}{\partial t}  =\L^\epsilon F$. 
Expand $F$ in $\epsilon$, $F(t)=F_0(t)+ \epsilon F_1(t)+\epsilon^2 F_2(t)+ o(\epsilon^2)$,
and plug this into the parabolic equation. Equating coefficients of $\epsilon^0$ and $\epsilon^{-1}$ 
we obtain the following equations:
  $$  \L_0F_1+{Y_0}F_0=0, \quad \f {\partial F_0}{\partial t}  ={Y_0} F_1+\L_0F_2.$$
   If $ {Y_0} F_0$ averages to zero, the formal solution to the first equation is  $F_1=-\L_0^{-1}({Y_0}F_0)$.
We should interpret this as an equation on $H$, and will solve this equation explicitly.
The second equation reduces to
$\f {\partial F_0}{\partial t} =-{Y_0}(\L_0^{-1}({Y_0}F_0))+\L_0F_2$ which we integrate with respect to $dh$. Since $\L_0$ is symmetric 
by Lemma \ref{le-symmetric} and $\L_0F_2$ averages to zero.  We obtain the equation for the effective motion:
$${d\over dt} \int_H F_0 dh=-\int_H {Y_0} \L_0^{-1} ({Y_0}F_0) dh.$$ 
The computations are formal.  Firstly we neglected higher powers of $\f 1 \epsilon$, which are very large for $\epsilon$ small.  Secondly we assumed that the Poisson equation $\L_0 ={Y_0}F_0$ is solvable. Finally we assumed that
$\int_H \L_0 F_2 dh=0$. Observing that $\L_0$ is not hypoelliptic on $G$, we must reduce the system to a space where it is, to justify these assertions.

{\it Main result.}   We assume that $\m$ is a reductive structure in the sense of Nomizu, a list of useful reductive decompositions are given in \S \ref{section2}. Denote by $\Ad_H: H \to \LL(\g;\g)$ the restriction of the adjoint of $G$ to $H$ and also the restricted representation $\Ad_H: H \to \LL(\m;\m)$. Denote by $\m_0\subset \m$  the space of invariant vectors of $\Ad_H$. Let $\m=\m_0\oplus \m_1\oplus \dots \oplus \m_l$ be an invariant decomposition for $\Ad_H$ where, for each $l\not =0$,  $\m_l$ is an irreducible invariant space. If $H$ is compact we may and will assume that the decomposition is orthogonal.
  The main results are in three parts. 
\begin{enumerate}
\item [(1)]   {\it Separation of slow and fast variables and Reduction.} Use the Ehresman connection on the principal bundle  $\pi: G\to G/H$, determined by $\m$,  and horizontal lifts of curves from $G/H$ to $G$ we deduce stochastic  equations
for the horizontal lifts $\tilde x_t^\epsilon$ of $x_t^\epsilon=\pi(g_t^\epsilon)$, the slow motions, 
 and for the fast motions $h_t^\epsilon$ on $H$. 

\item  [(2)]  {\it Convergence of the slow components.} 
If $Y_0$ is orthogonal to $\m_0$ and $\{A_0, A_1, \dots, A_N\}$ generates $\h$, the stochastic processes
$ (\tilde x_{t\over \epsilon}^\epsilon, t\in [0,T])$ and $(x_{t\over \epsilon}^\epsilon, t\in [0,T])$ converge weakly, as
 $\epsilon\to 0$, to $\bar u_t $ and $\bar x_t$ respectively, in the weak topology on the path space over $G$ and in Wasserstein distance. A rate of convergence is given. 

\item [(3)] {\it Effective Process.} Assume $A_0=0$ and $Y_0\in \m_l$. The effective process on $G$ is a Markov process
with generator  $c\Delta_{m_l}$. If furthermore $M$ is isotropy irreducible and if $M$ is naturally reductive 
(and more generally  ) then $\bar x_t$ is a scaled Brownian motion whose scale is computable.
\end{enumerate}

We indicate the problems pertinent to part (3).
The reduced first oder random differential operators give rise to second order differential operators by the action of the Lie brackets.
 If the Lie bracket between $A_k$ and $Y_0$ is not trivial, randomness is generated in the $[A_k,Y_0]$ direction and transmitted from the vertical  
to the horizontal directions. We ask the question whether the limiting operators are scaled horizontal Laplacians.
 This is so for the Hopf fibration.
However it only takes a moment to figure out this cannot be always true. 
The noise cannot be transmitted to directions in an irreducible invariant subspace of $\m$ not containing $Y_0$.
If $\m_l$ is the irreducible invariant subspace containing $Y_0$,
the action of $\Ad_H$  generates directions in $\m_l$,  not any direction in a component complementary to $\m_l$ and so the rank of $\bar \L$ is at most $\dim(\m_l)$. 
Within an irreducible $\Ad_H$-invariant subspace $\m_l$, the transmission of the noise should be `homogeneous'
and we might expect that there is essentially only one, up to scalar multiplication,  candidate effective second order differential operator: the generalised Laplacian' operator' $\Delta_{\m_l}:=\sum L_{Y_i}L_{Y_i}$ where $\{Y_i\}$ is an orthonormal basis of $\m_l$.
The generalised Laplacian $\Delta_{\m_l}$ is independent of the choice of the basis.

 If $\{A_1, \dots, A_p\}$ is a basis of $\h$, we solve a Poisson equation and prove that $\bar \L=\sum_{i,j}a_{i,j}(Y_0)L_{Y_i}L_{Y_j}$, where $\{Y_i\}$ is a basis of the irreducible $\Ad_H$ invariant space $\m_l$ containing $Y_0$ and $a_{i,j}(Y_0)$ trigonometric functions of the adjoint sub-representation in this basis. It  can be written in terms of eigenvalues of $\L_0$, computable from the `Casimir', and has an $\Ad_H$-invariant form. If $Y_0\in \m_l$ then  by the real Peter-Weyl theorem, a proof of which by Dmitriy Rumynin is appended at the end of the paper, we see that $\bar \L={1\over \dim(\m_l)\lambda_l}|Y_0|^2\Delta_{\m_l}$.  In case $G$ is uni-modular and $\m$ is irreducible this is the `horizontal Laplacian'. 
    If $\{A_k\}$ is only a set of generators of the Lie algebra $\h$,  a similar formula holds with the constant $\lambda_l$ replaced by a constant $\lambda(Y_0)$, depending on $Y_0$ in general.  

 The above descriptions are algebraic, in \S\ref{section8} we discuss their differential geometric counterpart. Let $G$ be endowed with
a left invariant and $\Ad_H$-invariant  Riemannian metric and let $M=G/H$ be given the induced $G$-invariant Riemannian metric. The translates of the one parameter subgroups of $G$ are not necessarily geodesics for the Levi-Civita connection on $G$. Their projections to the Riemannian homogeneous manifold are not necessarily geodesics either. In general $\bar \L$ needs not be the Laplace-Beltrami operator even if it is elliptic for it may have a nontrivial drift. The more symmetries there are, the closer is the effective diffusion to a Brownian motion. A maximally symmetric and non-degenerate effective process ought to be a scaled Brownian motion.
   In other words we like to see the convergence of the random smooth curves in the orbit manifold to a scaled Brownian motion, which will be studied under the condition on the trace of $\ad(Z)$, under which
   $\trace_{\m_l} \nabla^L d=\trace_{\m_l} \nabla d$ and $\trace_{\m_l} \nabla^c d=\trace_{\m_l} \nabla d$ on $M$, where $\nabla$ denotes the Levi-Civita connections on $G$ and on $M$ and $\nabla^c$ denotes the canonical connection on $M$.

\medskip

{\it The context.} The study of parabolic differential equation of the type 
${\partial \over \partial t} =Y_0+\f 1 {\epsilon} \L_0$ where $\L_0$ is an elliptic operator and $Y_0$ a vector field, in the non-geometric settings goes back to S. Smoluchowski (1916) and to H. Kramers (1940) \cite{Kramers} and are known as  Smoluchowski-Kramers limits. This was taken up in \cite[E. Nelson]{Nelson} for unifying Einstein's theory and Ornstein-Uhlenbeck's theory of Brownian motions.  The further scaling by $1/\epsilon$, leading to the asymptotic problems for $\f 1 {\epsilon}Y_0+ \f 1 {\epsilon^2} \L_0$, are known by the a number of terminologies in a great many  subjects:  averaging principle, stochastic homogenization, multi-scale analysis, singular perturbation problem, or taking the adiabatic limit. They also appear in the study of interacting particle systems. We treat our problem as a perturbation to a non conventional conserved system and use techniques from multi-scale analysis and stochastic homogenization. 
 For perturbation to Hamiltonian systems see M. Freidlin \cite{Freidlin}, see also \cite[M. Freidlin and M. Weber ]{Freidlin-Weber}.  
The use of  `Hamiltonian'  is quite liberal, by which we merely mean that they are conserved quantities, here they are manifold valued functions or orbits. 
  Some ideas in this paper were developed from  \cite[X.-M. Li ]{Li-OM-1}. Convergences of Riemannian metrics from the probabilistic point of view, have been studied by N. Ikeda and Y. Ogura \cite{Ikeda-Ogura} and Y. Ogura and S. Taniguchi \cite{Ogura-Taniguchi96}. 

Our reduced equations are ODEs on manifolds with random coefficients, the study of random ODEs on $\R^n$ goes back to
\cite[R.L. Stratonovich]{Stratonovich-rhs}, \cite[R.Z. Khasminski]{Khasminskii66},  \cite[W. Kohler and G. C. Papanicolaou]{Kohler-Papanicolaou74} and   \cite[G. C. Papanicolaou and S. R. S. Varadhan]{Papanicolaou-Varadhan73},
 \cite[A. N. Borodin and M. Freidlin]{Borodin-Freidlin}.  The convergence from operator semigroup point of view is studied in \cite[T. Kurtz]{Kurtz70}.
A collection of limit theorems, some of which are used here, together with a set of more complete references can be found in  \cite[X.-M. Li ]{Li-limits}.  
For stochastic averaging with geometric structures, we refer to  \cite[X.-M. Li]{Li-averaging}  for a stochastic integrable Hamiltonian systems in symplectic manifolds, and to \cite[P. Gonzales-Gargate and  P. Ruffino]{Gonzales-Gargate-Ruffino} and \cite[M. Hogele and P. Ruffino]{Hogele-Ruffino} for studies of stochastic flows on foliated manifolds. 

 We also use  multi-scale analysis to separate the slow and fast variables and draw from the work in \cite[K. D. Elworthy, Y. LeJan and X.-M. Li]{Elworthy-LeJan-Li-book-2}. Our reduction in complexity bears resemblance to aspects of geometric mechanics and Hamiltonian reduction, see \cite[L-C. Joan-Andreu and O. Juan-Pablo]{Joan-Andreu-Juan-Pablo}, also \cite[D. Holm, J. Marsden, T. Ratiu, A. Weinstein]{Holm-Marsden-Ratiu-Weinstein}, and  \cite[J. Marsden, G. Misiolek,  J.-P. Ortega, M. Perlmutter and T. Ratiu]{M-M-O-P-R}.
The problem of identifying limit operators is reminiscent of that in M. Freidlin and A. D. Wentzell \cite{Freidlin-Wentzell-93} where  an explicit Markov process on a graph was obtained from the level sets of a Hamiltonian function in $\R^2$, but they are obviously  very different in~nature.

  The asymptotic of linear operators of the form $\f 1 \epsilon \L^\epsilon:=\f 1 {\epsilon^2}\L_0+\f 1 \epsilon \L_1$   are often referred as taking the adiabatic limits in some contexts, see for example see \cite[R.R. Mazzeo-R. B. Melrose]{Mazzeo-Melrose} and
  \cite[N. Berline, E. Getzler, and  M. Vergne]{Berline-Getzler-Vergne}.    These operators act not just on scalar functions and their objectives are motivated by problems in topology.  
     Motivated by analysis of loop spaces J.-M. Bismut  \cite{Bismut-Lie-group, Bismut-hypoelliptic-Laplacian}   studied `limits' of a family of hypoelliptic differential operators 
 on the tangent bundle  or the cotangent bundle of a compact Riemannian manifold $M$,
 see also  \cite[J.-M. Bismut and G. Lebeau]{Bismut-Lebeau} for an earlier work.  They are  devoted to study  the convergence of probability distributions of the projections of  hypoelliptic diffusions to Brownian motion. 
 Unlike in these studies  our state space is not necessarily compact,  the operators $\L^\epsilon$ are not necessarily hypoelliptic, 
 the projections of the effective limiting diffusions are not always  Brownian motions.
Classify the limits are one of our objectives.
In another direction, E. Nelson \cite{Nelson}, proved that the physical particles whose velocity are one  dimensional Ornstein-Uhlenbeck processes approximates, in high temperature,  one dimensional Brownian motions. This was generalised to manifolds in \cite[R. M. Dowell]{Dowell}.  
This study bears much resemblance to the above mentioned articles; however the model  and the objectives in \cite[Li]{Li-geodesic} are much closer: SDEs on the orthonormal frame bundle are used to show that  Brownian motions are generated by `geodesics' with rapidly changing directions.

\section{Notation,  preliminaries, and examples}
\label{section2}

If $g\in G$ we denote by $L_g$ and $R_g$ the left and right multiplications on $G$. Denote by $dL_g$ and $dR_g$  their differentials. The  differentials are also denoted by $TL_g$ and $TR_g$ where these are traditionally used. 
Denote by
 $\Ad$ the adjoint representation of $G$ with $\Ad_H$ its restriction to $H$ and $\ad$ the adjoint representation on $\g$ with restriction $\ad_\h$. Both the notation $\ad_XY$ and $[X,Y]$ are used.
If $\m$ is a subspace of $\g$, $g\m$ denotes 
the set of left translates of elements of $\m$ to $T_gG$. We identify the
set of left invariant vector fields with elements of the Lie algebra, denote them by the same letters or with 
the superscript~$*$. If $V$ is a vector field and $f$ function, both  $L_V f$  and $Vf$ are use to denote Lie differentiation in the direction $V$, the latter for simplicity, the first to avoid confusing with operators denoted by $\L$. 
Also, $\Delta_H$ denotes the Laplacian for a bi-invariant metric on the Lie group $H$.

A manifold $M$ is homogeneous if there is a transitive action by a Lie group $G$. It can be represented as the coset space $G/H$ where $H$ is the isotropy group at a point $o$. 
The coset space is given the 
unique manifold structure such that  $g H\in G/H\mapsto \bar L_g o \in M$ is a diffeomorphism where $\bar L_g$ denotes the action of $g$. If $\pi: G\to M$ denotes the projection taking $g$ to the coset $gH$ and $1\in G$ the identity, the sub-Lie algebra $\h$ is the kernel of $(d\pi)_1$ and $T_oM$ is isomorphic to a complement of $\h$ in $\g$.

 The homogeneous space $G/H$ is reductive if in the Lie algebra $\g$ there exists a subspace $\m$ such that
 $\Ad(H)(\m)\subset \m$ and $\g=\h\oplus \m$ is a vector space direct sum. By $\Ad_H(\m)$ we mean the image of $H$ under the adjoint map,  treated as linear maps restricted to $\m$. 
We say that $\g=\h\oplus \m$ is a {\it reductive decomposition} and $(\g, \h)$ a reductive pair. This implies that $[\h,\m]\subset \m$ and vice versa if $H$ is connected. 
The homogeneous space is also called reductive, a reductive property in the sense of Nomizu, a concept different from a Lie group being reductive.
In particular, the Lie group $G$ is not necessarily reductive, by which we mean its adjoint representation is completely reducible. 

We discuss briefly the connectedness of the Lie group. Firstly
 if $H$ and $G/H$ are connected, so is $G$. This follows from the fact that a topological space $X$  is connected if and only if every continuous function from $X$ to $\{0,1\}$ is constant. 
 The identity component  $G^0$ of $G$ is a normal subgroup of $G$, and any other
component of $G$ is a coset of $G^0$.  The component group $G/G^0$ is discrete.
If a Lie group $G$ acts transitively on a connected smooth manifold $M$, so does $G^0$. See \cite{Gorbatsevich-Onishchik-Vinberg}.
Our stochastic processes are continuous in time, and hence we may and will assume that both $G$ and $H$ are connected.

The existence of an $\Ad_H$invariant inner product is much easier than requesting an $\Ad$-invariant inner product
on $\g$ which is equivalent to $G$ is of compact type.  
If $H$ is compact, by the unitary trick,  there exists an $\Ad_H$-invariant inner product on $\g$ and a reductive structure by setting $\m=\h^\perp$. 
The compactness of $H$ is not a restriction for a Riemannian homogeneous manifold.  If $G/H$ is a connected Riemannian homogeneous space and $G$ is connected then by a
  theorem of van Danzig and van der Waerden (1928), the isotropy isometry groups at every point is compact,  \cite[Kobayashi]{Kobayashi}. If $H$ is connected and its Lie algebra is reductive in $\g$, in the sense that $\ad(\h)$ in $\g$ is completely reducible, then $G/H$ is reductive.
The Euclidean space example below will cover the averaging model used in \cite[M. Liao and L. Wang ]{Liao-Wang-07}.
Let $G=E(n)$ be the space of rigid motions on $\R^n$,
$$E(n)=\left\{\left( \begin{matrix} R &v\\ 0&1\end{matrix}  \right): R\in SO(n), v\in \R^n\right\}$$
 and $H$ its subgroup of rotations.  Elements of $H$ fix the point $o=(0,1)^T$,  $E(n)/H=\{(x,1)^T, x\in \R^n\}$ and a matrix in $E(n)$ projects to its last column. We may take
$$ \h=\left\{\left( \begin{matrix} A &0\\ 0&0\end{matrix}  \right), A\in \so(n)\right\}, \;
 \m=\left\{\left( \begin{matrix} 0 &v\\ 0&0\end{matrix}  \right),  v\in \R^n \right\}.$$

A reductive structure may not be unique. For example, let $G$ be  a connected compact Lie group, $H=\{(g,g): g\in G\}$ and $\h=\{(X, X): X\in \g\}$. 
   Then $G=(G\times G)/H$ is a reductive homogeneous space in three ways:  $\m^0=\{(X,-X), X\in \g\}$,
$$ \m^+=\{(0,X), X\in \g\}, \quad     \quad \m^-=\{(X,0), X\in \g\}.$$
The first one,  $\h\oplus \m^0$,  is the symmetric decomposition. 

A definite metric is not necessary either.  If $\g$ admits an $\Ad_H$ invariant non-degenerate bilinear form such that its restriction to $\h$ is non-degenerate,  let
$\m=\{ Y\in \g: B(X, Y)=0, \forall X\in \h\}$.
In   \cite[Chap.~10]{Kobayashi-NomizuI}, Kobayashi and Nomizu considered the case where $B$ is $\Ad(G)$ invariant.  
Their proof can be modified to work here.

  \subsection{The Motivating Example}
\label{Bergers}

As a motivating example, we take $G=SU(2)$, $H=U(1)$ and the bi-invariant 
Riemannian metric such that
$\<A,B\>={1\over 2}\trace AB^*$ where $A, B\in \g$.  Let
\begin{equation}\label{pauli}
X_1=\left(\begin{matrix} i &0\\0&-i
\end{matrix}\right), \quad X_2=\left(\begin{matrix} 0 &1\\-1&0
\end{matrix}\right),
 \quad X_3=\left(\begin{matrix} 0 &i\\i&0
\end{matrix}\right)
\end{equation}
be the Pauli matrices. They form an orthonormal basis of the Lie algebra $\g$ with respect to the bi-invariant metric and $G$ is the canonical three sphere in $\R^4$.
The Lie algebra $\h$ is generated by $X_1$ and we take $\m$ to be the vector space generated by the remaining two Pauli matrices and obtain the family of Berger's metric, which we denote as before by $\{m^\epsilon, \epsilon>0\}$. 

 Let us first take the Brownian motions on Bergers' spheres. 
 A Brownian motion on $(S^3, m^\epsilon)$ is determined by the H\"ormander type operator
 $\Delta^\epsilon={1\over \epsilon} (X_1)^2+ \sum_{i=2}^3 (X_i)^2$.  Although no longer associated with the round metric,   there are many symmetries in the following SDEs, 
 \begin{equation}
 \label{sphere-1}
 dg_t={1\over \sqrt \epsilon} X_1(g_t)\circ db_t^1+X_2(g_t)\circ db_t^2
+X_3(g_t)\circ db_t^3.
\end{equation}
 In particular the probability distributions of its slow components at time $t$
 are independent of $\epsilon$, see Example \ref{Berger's-sphere} in Section \ref{section-reduction}. Breaking up the symmetry we may consider the equation,
  $$dg_t={1\over \sqrt \epsilon} X_1(g_t)\circ db_t^1+X_2(g_t)\circ db_t^2,$$
  in which the incoming noise is 2-dimensional and its Markov generator satisfies the strong H\"ormander's conditions.  We will go one step further and use a one dimensional noise.
Let $(b_t)$ be a real valued Brownian motion. 
 If $Y_0=c_2X_2+c_3X_3$, where $c_2, c_3$ are numbers not simultaneously zero, the infinitesimal generator of the equation
 \begin{equation}
\label{sde0}
dg_t={1\over \sqrt \epsilon} X_1(g_t)\circ db_t+(Y_0)(g_t) dt,
\end{equation}
satisfies weak H\"ormander's conditions. Indeed by the structural equations, $[Y_0, X_1]=2c_2X_3-2c_3X_2$, and the matrix   
   $$\left(
\begin{array}{ccc}
{1\over \sqrt \epsilon}&0&0\\
0  &c_2 &-2c_3\\ 0&c_3&2c_2
\end{array}	\right).$$
 is not degenerate. In general, equation (\ref{1}) need not  satisfy H\"ormander's condition. 
This example is concluded in Example \ref{Hopf-theorem}, using Corollary \ref{corollary-symmetric}.

    \section{The interpolation equations}
    \label{section-interpolation}

Given a left invariant Riemannian metric on $G$, we consider 
 a family of non-homogeneously scaled Riemannian metrics.  To define these let $\{X_1, \dots, X_n\}$ be an orthonormal basis of $\g$ extending an orthonormal basis $\{X_1, \dots, X_p\}$ of $\h$ and let $$E^\epsilon=\left\{{1\over \sqrt \epsilon} X_1^*, \dots, {1\over \sqrt \epsilon} X_p^*, X_{p+1}^*, \dots X_n^*\right\}.$$ 
 The superscript $*$ above a letter denotes the corresponding left invariant vector field which we will omit, from time to time,  in favour of simplicity. The dual frame of $E^\epsilon$ defines a  family of left invariant Riemannian metrics on $G$ which are denoted by $m^\epsilon$, then $E^\epsilon$ is an orthonormal frame.  
  In this article we are not concerned with
 the problem of keeping the sectional curvatures bounded, and $G$ needs not be compact.

  Let $(\Omega, \F, \F_t, P)$ be a filtered probability space satisfying the usual assumptions 
 and  $(b_t^k, k=1,\dots, N,)$ a family of independent real valued Brownian motions. Let  $\gamma,   \delta$ be positive real numbers, $X_k \in \h$ as above, and  $Y_0\in \g$. 
 We study the stochastic dynamics associated with the above inhomogeneous scalings of Riemannian metrics and propose the following interpolation equation that,  in the limit of $\epsilon\to 0$, describes the `effective motion'  across the `orbits'.
 \begin{equation}\label{1-2}
dg_t =\sum_{k=1}^p  \gamma X_k(g_t) \circ db_t^k +
\delta Y_0(g_t) dt,
\end{equation}
with a given initial value $g_0$. Here $\circ$ denotes Stratonovich integration.  Their solutions are Markov processes whose probability laws are determined by the fundamental solutions to the following parabolic equation 
$\f \partial {\partial t}=\f 1 2\gamma^2\sum_{k=1}^p(X_k)^2+\delta Y_0$. 

For $\gamma=1$ and ${\delta }={1\over |Y_0|}$,
  these equations are driven by unit length vector fields on the Riemannian manifold $(G,m^\epsilon)$. 
If $\delta=0$ the solution with its initial value the identity of the group is a scaled Brownian motion on the subgroup $H$. 
If $\gamma=0$, the solutions are translates of the one parameter family subgroup generated by $Y_0$. We denote the vector space generated by $\{ X_{p+1}, \dots, X_n\}$ by $\m$. We assume in addition that $\m$ is $\Ad_H$-invariant. If $Y_0\in \m$, these one parameter families, $\exp(\delta tY_0)$, are horizontal curves for the horizontal distribution determined by $\m$. 
The solutions of  (\ref{1-2}) interpolate between translates of the one parameter group, generated by $\delta Y_0$, on $G$ and Brownian motion on $H$. The question is:  if we take $\gamma\to \infty$ while keeping $\delta$ fixed,  what can we say about the solutions of these equations?

We will work on more general operators, allowing $\{A_k\}$ to be a Lie algebra generating subset of $\h$ instead of assuming ellipticity. Take $Y_0\in\m$. We study the following family of~SDEs, where $\epsilon>0$ is a small parameter,
 \begin{equation*}
dg_t^\epsilon ={1\over \sqrt \epsilon} \sum_{k=1}^N   A_k(g_t^\epsilon) \circ db_t^k +{1\over \epsilon} A_0(g_0^\epsilon)dt
+Y_0(g_t^\epsilon) dt, \quad g_0^\epsilon =g_0.
\end{equation*}
The condition $g_0^\epsilon=g_0$, independent of $\epsilon$, is assumed only for the simplicity of the statements.
Let $\L_0=\f 12 \sum (A_k)^2+ A_0$.
The corresponding parabolic problems are:
$$\f \partial {\partial t}=\f 1 \epsilon  \L_0+Y_0, \quad \quad \f \partial {\partial t}=\f 1 {\epsilon^2} \L_0+\f 1 {\epsilon}Y_0.$$

For intuition, let us  review the theory for randomly perturbed Hamiltonian dynamics.
If $H$ is a function on $R^{2n}$, it is a first integral for the following Hamiltonian system: $\dot q=-{\partial H \over \partial p},  \dot p={\partial H \over \partial q}$. Let $x^\epsilon(t):=(q^\epsilon(t), p^\epsilon(t))$ denote solutions to a perturbed Hamiltonian system, which we do not specify, then $H^\epsilon(t):=H(p^\epsilon(t), q^\epsilon(t))$ varies slowly with $t$ on $[0,1]$. Under suitable mixing conditions on the perturbation the stochastic processes $H^\epsilon(t/\epsilon)$ converge, see e.g. 
 \cite[M. Freidlin and A. D. Wentzel]{Freidlin-Wentzell-93} and    \cite[V. I. Arnold]{Arnold89}.
 Although we do not have a Hamiltonian system, the projection $\pi:G\to G/H$ is a {\it conservation law}
  for an `unperturbed' dynamical system, which is the key for the reduction to the slow varying dynamics.  If
$(y_t^\epsilon)$ are solutions to the equations $dy_t^\epsilon ={1\over \sqrt \epsilon} \sum_{k=1}^N   A_k(y_t^\epsilon) \circ db_t^k+{1\over \epsilon} A_0(y_t^\epsilon) $, where $A_k\in \h$, then $\pi(y_t^\epsilon)=\pi(y_0^\epsilon)$ for all $t$.
 The orbits of $(g^\epsilon_t)$ are perturbations to the orbits of the vertical motions $(y_t^\epsilon)$. The `constant of motion' for the latter takes its value in the orbit manifold $G/H$.  There will be of course extra difficulties, which is mainly due to the fact that our slow motions are not necessarily local functions of $(g_t^\epsilon)$, they depend not simply on $(g_t^\epsilon)$ but also on its whole trajectory.  

For $\epsilon$ small, we would expect $\pi(g_t^\epsilon)$ to measure its deviation from the `orbit' containing $g_0^\epsilon$. 
The inherited non-linearity from $\pi$ causes some technical problems, for example the homogeneous manifold is in general not prarallelisable, so it is not easy to work directly with $x_t^\epsilon=\pi(g_t^\epsilon)$. 
The operator ${1\over \epsilon }\L_0+ Y_0$
does not satisfy H\"ormander's conditions, so we do not 
wish to work directly with $(g_t^\epsilon)$ either.
To overcome these difficulties we use an $Ad_H$-invariant decomposition of $\g$, with which we construct a stochastic process $(u_t^\epsilon)$ on $G$, having the same projection as $(g_t^\epsilon)$. 
Since $\m$ is $\Ad_H$ invariant,  $G$ is a principal bundle over $M$ with structure group 
$H$. We lift $(x_t^\epsilon)$ to $G$ to obtain a `horizontal' stochastic process $(\tilde x_t^\epsilon)$ covering $(x_t^\epsilon)$. 
The horizontality is with respect to the Ehresmann connection determined by $\m$. Then $(u_t^\epsilon)$ is the horizontal lift process. 
 As  `perturbations' to the random motions on the fibres, the horizontal lifts $(\tilde x_t^\epsilon)$ describe transverse motions across the fibres.

 This consideration has a bonus:
in addition to asymptotic analysis of the $x$ processes, we also obtain information on the asymptotic properties of their horizontal lifts. This is more striking if we get out of the picture of homogenisation for a moment, and consider instead the three dimensional Heisenberg group
 as a fibre bundle over $\R^2$. The horizontal lift of an $\R^2$-valued Brownian motion to the Heisenberg group is their stochastic L\'evy area, c.f. \cite{Li-OM-1}. 
 Horizontal lifts of stochastic processes are standard tools in the study of Malliavin calculus in association with the study of the space of continuous paths over a Riemannian manifold. In \S\ref{section-reduction} below we deduce an explicit equation for the horizontal lift in terms of the vertical component of $(g_t^\epsilon)$, making further analysis possible.

    \section{Reduction, separation of slow and fast variables}
   \label{section-reduction}
   The sub-group $H$ acts on $G$ on the right by group multiplication and we have a principal bundle structure $P(G,H,\pi)$ with base space $M$ and structure group $H$. Each fibres $\pi^{-1}(x)$ is diffeomorphic to $H$ and the kernel of $d\pi_g$, the differential at $g$, is $g\h=\{ TL_g( X): X\in \h\}$
For $a\in G$, denote by $\bar L_a$ the left action on $M$:  $\bar L_a( gH)= agH$.  Then    $\pi\circ L_a=\bar L_a \circ\pi$
and  $(d\pi)_adL_a|_{T_1G}=(d\bar L_a)(d\pi)_1$ where $1$ denotes the unit element of $G$. 

In this section we assume that $H$ contains no normal subgroup of $G$, which is equivalent to $G$ acts on $G/H$ effectively, i.e. an element of $G$ acting as the identity transformation on $G/H$ is the identity element of $g$. This is not a restriction on the homogeneous space. If $G$ does not act transitively on $M=G/H$, then there exists a normal subgroup $H^0$ of $H$ such that
$G/H^0$ acts transitively on $G/H=(G/H^0)/(H/H^0)$. 
We identify $\m$ with $T_oM$ as below:
 $$X\mapsto  (d\pi)_1 (X)= {d\over dt}\big|_{t=0} \, d\bar L_{\exp(tX)}o.$$
 An Ehresmann connection is a choice of  a set of complements of $g\h$ that is right invariant by the action of $H$.   Our basic assumption is that $\g=\h\oplus \m$ is a {\it  reductive decomposition}, in which case $(gh)\m=TR_h (g\m)$ and
$T_gG=g\h\oplus g\m$
is an Ehresmann connection. See  \cite[S. Kobayashi and K. Nomizu]{Kobayashi-NomizuI}. 

 An Ehresmann connection determines and is uniquely determined by
horizontal lifting maps $\h_{u}: T_{gH} M\to T_uG$ where $u\in \pi^{-1}(gH)$. Furthermore, to every
  piecewise $C^1$ curve $c$ on $M$ and every  initial value $\tilde c(0)\in \pi^{-1}(c(0))$, 
  there is a unique curve $\tilde c$ covering $c$ with the property that
${d\over dt} {\tilde c(t)}\in \tilde c(t)\m$.  See Besse \cite{Besse}.
 We can also horizontally lift a sample continuous  semi-martingale.
The case of the linear frame bundle is specially well known, see
J. Eells and D. Elworthy \cite{Eells-Elworthy76}, P. Malliavin \cite{Malliavin78}. See also  M. Arnaudon \cite{Arnaudon-homogeneous}, M. Emery \cite{Emery} and D. Elworthy \cite{Elworthy-book}, and N. Ikeda, S. Watanabe \cite{Ikeda-Watanabe}. This study has been taken further in D. Elworthy, Y. LeJan, X.-M. Li \cite{Elworthy-LeJan-Li-book-2}, in connection
with horizontal lifts of intertwined diffusions. 
A continuous time Markov process, whose infinitesimal generator is in the form of the sum of squares
of vector fields, is said to be horizontal if the vector fields are horizontal vector fields.

 Let $\{b_t^l, w_t^k,   k=1, \dots,  N_1, l= 1, \dots, N_2\}$ be 
real valued, not necessarily independent, Brownian motions.  Let $\g=\h\oplus \m$ be a reductive structure .
 Let $\{A_i, 1\le i\le p\}$ be a basis of $\h$, $\{X_j, p+1\le j\le n\}$  a basis of $\m$, and $\{c_k^i, c_l^j\}$
is a family of real valued smooth functions on $G$.  
  Let $\varpi$ be the canonical {\it connection 1-form} on the principal bundle $P(G, H, \pi)$, determined by
$\varpi (A_k^*)=A_k$ whenever $A_k\in \h$ and $\varpi(X_j^*)=0$ whenever $X_j\in \m$.
A left invariant vector field corresponding to a Lie algebra element is denoted by an upper script $*$ for emphasizing.

 \begin{definition}
 A  semi-martingale  $(\tilde x_t)$  in $G$ is horizontal if  $\varpi(\circ d\tilde x_t)=0$. If $(x_t)$ is a semi-martingale on $M$, we denote by $(\tilde x_t)$ a  horizontal  lift.
 \end{definition}
 
Let $Y_k^h(g)=\sum_{i=p+1}^n c_k^i(g)X_i^*(g)$ and $
 Y_l^v(g)=\sum_{j=1}^{p}c_l^j(g)A_j^*(g)$.
Denote by $(g_t, t<\zeta)$ the maximal solution to the following system of equations,
\begin{equation}\label{3-1}
dg_t=Y_0^h(g_t)dt+ \sum_{k=1}^{N_1}Y_k^h(g_t) \circ dw_t^k+Y_0^v (g_t)dt+
  \sum_{l=1}^{N_2} Y_l^v(g_t)\circ db_t^l,
\end{equation} 
with initial value $g_0$ and  $x_t=\pi(g_t)$. For simplicity let $b_t^0=t$ and $w_t^0=t$.

In the lemma below, we split $(g_t)$ into its `horizontal' and `vertical part' and
describe the horizontal lift of the projection of $(g_t)$ by an explicit 
stochastic differential equation where the role played by the vertical part is transparent.
\begin{lemma}
\label{lemma5.2}
 Let $x_0=\pi(g_0)$. Take $u_0\in \pi^{-1}(x_0)$ and define $a_0=u_0^{-1}g_0$.  Let
  $(u_t,  a_t, t<\eta)$ be the maximal solution to the  following system of equations  
\begin{eqnarray}
\label{horizontal-1}
d u_t  &=&\sum_{k=0}^{N_1} \sum_{i=p+1}^n c_k^i(u_t a_t)
\left( \Ad(a_t) X_i\right)^*( u_t) \circ   dw_t^k\\
d a_t&=& \sum_{l=0}^{N_2} \sum_{j=1}^{p} c_l^j(u_t a_t)A_j^*(a_t)\circ db_t^l.
\label{horizontal-2}
\end{eqnarray}
Then the following statements hold.
\begin{enumerate}
\item [(1)]  $(u_ta_t, t<\eta)$ solves (\ref{3-1}). Furthermore $\eta\le \zeta$ where $\zeta$ is the life time of $(g_t)$.
\item[(2)] 
 $(u_t, t< \eta )$ is a horizontal lift of $(x_t, t< \zeta ) $. Consequently $\zeta=\eta$ a.s.

\item [(3)]  If $(\tilde x_t, t<\zeta)$ is an horizontal lift of $(x_t, t<\zeta)$, it is a solution of (\ref{horizontal-1}) with $u_0=\tilde x_0$.

\end{enumerate}
\end{lemma}

\begin{proof} 
(1) Define $\tilde g_t:=u_ta_t$.   On $\{t< \eta\}$, we have
$$d\tilde g_t=dR_{a_t}  \circ du_t+( a_t^{-1}\circ da_t)^*(\tilde g_t).$$
Here $dR_g$ denotes the differential of the right translation $R_g$, $a_t^{-1}$ in the last term denotes the action of the differential, $dL_{(a_t)^{-1}}$, of the left multiplication. See page 66 of S. Kobayashi, K. Nomizu \cite{Kobayashi-NomizuI}. The stochastic differential $d$ on both the left and right hand side denotes Stratonovich integration.
Then $(\tilde g_t, t < \eta)$ is a solution of (\ref{3-1}),  which follows from the computations below.
\begin{equation*}
{\begin{aligned}
d\tilde g_t=&\sum_{k=0}^{N_1} \sum_{i=p+1}^n c_k^i(u_t a_t)
dR_{a_t}\left( \Ad(a_t) X_i\right)^*( u_t) \circ   dw_t^k\\
&+\sum_{l=0}^{N_2} \sum_{j=1}^{p} c_l^j(u_t a_t)A_j^*(\tilde g_t)\circ db_t^l.
\end{aligned}}
\end{equation*}
Since $ dR_{a_t}\left(\Ad(a_t) (X_j)\right)^*(u_t)=X_j^*(\tilde g_t)$, 
$$d\tilde g_t=\sum_{k=0}^{N_1} \sum_{j=p+1}^n c_k^j(\tilde g_t)
\left(X_k\right)^*( u_t) \circ   dw_t^k
+\sum_{l=0}^{N_2} Y_l^v(\tilde g_t)\circ db_t^l,
$$ 
which is equation (\ref{3-1}). Since the coefficients of (\ref{3-1}) are smooth, pathwise uniqueness holds.  In particular
  $g_t=u_t a_t$ and the life time $\zeta$ of (\ref{3-1}) must be greater or equal to $\eta$.

(2)   It is clear that $a_0\in H$ and $\varpi(\circ du_t)=0$.
 Let $y_t=\pi(u_t)$. Then
$$d y_t  =\sum_{k=0}^{N_1} \sum_{i=p+1}^n c_k^i(u_ta_t)
 d\bar L_{u_ta_t} \left(d\pi(X_i)\right) \circ   dw_t^k,$$
following from the identity $d\pi\left(\left( \Ad(a) X_i\right)^*(u)\right)=d\bar L_ud\bar L_a \left(d\pi(X_i)\right)$. 
By the same reasoning $(x_t)$ satisfies the equation
$$dx_t=\sum_{k=0}^{N_1} d\pi(Y_k^h(g_t)) \circ dw_t^k.$$
By the definition, $d\pi(Y_k^h(g_t))=\sum_{i=p+1}^n c_k^i(g_t)d\pi(X_i^*(g_t))$ and $d\pi(X_i^*(g_t))=T\bar L_{g_t} d\pi(X_i)$.
Using part (1), $g_t=u_ta_t$, we conclude that the two equations above are the same and $\pi(u_t)=x_t$. This concludes that  $(u_t)$ is a horizontal lift of $(x_t)$
up to time $\eta$.

It is well known that through each $u_0$ there is a unique horizontal lift $(\tilde x_t)$
and the life time of $(\tilde x_t)$ is the same as the life time of $(x_t)$.
See I. Shigekawa \cite{Shigekawa82} and R. Darling \cite{Darling-thesis}.
The life time of $(x_t)$ is  $\zeta$. 
Let $\tilde a_t$ be the process such that $g_t=u_t\tilde a_t$ for $t<\eta$. On $\{t<\eta\}$, $\tilde a_t=a_t$ and $u_t=\tilde x_t$. If $\eta<\zeta$, as $t\to \eta$,
$\lim_{t\to \eta} u_t$ leaves every compact set. This is impossible as it agrees with $\tilde x_t$. Similarly $(a_t)$ cannot explode before $\zeta$.

(3) Let $(g_t, t<\zeta)$ be a solution of (\ref{3-1}) and set $x_t=\pi(g_t)$.
For each $t$,  $g_t$ and $\tilde x_t$ belong to the same fibre.
Define $k_t=\tilde x_t^{-1}g_t$,  which takes values in $H$ and is defined for all $t<\zeta$. 
Then,
 $$d \tilde x_t= dR_{(k_t)^{-1}} \circ  dg_t + \left( k_t\circ d(k_t)^{-1} \right)^*(\tilde x_t).$$
From this and equation (\ref{3-1})  we obtain the following,

\begin{equation}\label{3-3}
{\begin{aligned}d \tilde x_t&=  dR_{(k_t)^{-1}} \left(
\sum_{k=0}^{N_1} \sum_{i=p+1}^n c_k^i(g_t)
X_i^*( g_t) \circ   dw_t^k  \right)\\
&+dR_{(k_t)^{-1}} \left(\sum_{l=0}^{N_2} \sum_{j=1}^p c_l^j(g_t) A_j^*(g_t)\circ db_t^l
\right) +( (k_t) \circ d(k_t)^{-1})^*(\tilde x_t). 
\end{aligned}}
\end{equation}
 We apply the connection 1-form $\varpi$ to equation (\ref{3-3}), observing  $\omega(\circ d\tilde x_t)=0$ and $\Ad(k_t)(X_i)\in \m$,
$${\begin{aligned}0&=\sum_{l=0}^{N_2} \sum_{j=1}^p c_l^j(g_t)
 \varpi_{\tilde x_t}\left(dR_{(k_t)^{-1}} A_j^*(g_t)\right)\circ db_t^l+(k_t) \circ d(k_t)^{-1}\\
&=\sum_{l=0}^{N_2} \sum_{j=1}^p c_l^j(g_t) \Ad(k_t) A_j\circ db_t^l
+(k_t) \circ d(k_t)^{-1}.\end{aligned}}$$
We have used the fact that $X_j^* $ are horizontal, $dR_{(k_t)^{-1}} \left(  (X_j)^*(g_t)\right)=\left(\Ad(k_t) (X_j)\right)^*(\tilde x_t)$,
and  $\varpi_{ga^{-1}}(R_{(a^{-1})_*}w)=\Ad(a)\varpi_g (w)$ for any  $w\in T_gG$.
It follows that $d(k_t)^{-1}  =-\sum_{l=0}^{N_2} \sum_{j=1}^p c_l^j(g_t) R_{(k_t)^{-1}} A_j\circ db_t^l$. By the product rule,
$$dk_t=  \sum_{l=0}^{N_2} \sum_{j=1}^pc_l^j(g_t) A_j^*(k_t)\circ db_t^l.$$
Thus $(k_t)$ solves equation (\ref{horizontal-2}) and we take this back to (\ref{3-3}). 
Since the vertical vector field associated to $k_t\circ dk_t^{-1}$ evaluated at $\tilde x_t$ is given by the formula
$$( k_t \circ d(k_t)^{-1})^*(\tilde x_t)
=-\sum_{l=0}^{N_2} \sum_{j=1}^p c_l^j(g_t)( \Ad(k_t)A_j)^*( \tilde x_t) \circ db_t^l,$$ 
 the second term and the third term on the right hand side of
(\ref{3-3}) cancel. Using the same computation given earlier, we see that 
$${\begin{aligned}d \tilde x_t&= dR_{(k_t)^{-1}} \left(
\sum_{k=0}^{N_1} \sum_{i=p+1}^n c_k^i(g_t)
X_i^*( g_t) \circ   dw_t^k  \right)  \\
&=\sum_{k=0}^{N_1} \sum_{i=p+1}^n c_k^i(\tilde x_tk_t)
\left(\Ad(k_t) X_i\right)^*( \tilde x_t) \circ   dw_t^k,
\end{aligned}}$$
proving that  $(\tilde x_t, t<\zeta)$ is a solution of (\ref{horizontal-1}) and concludes the proof. In particular
$\zeta\ge \tau$.
\end{proof}

We observe  that the Ehresmann connection induced by the reductive decomposition is independent of the scaling of the Riemannian metric.
\begin{corollary}
\label{Corollary3.3}
Let $\epsilon_l >0$ and $Y_0\in \m$. Let $(g_t^\epsilon)$ be a solution to the equation
\begin{equation}\label{3-1-2}
{\begin{aligned}
dg_t^\epsilon&=Y_0^h(g_t^\epsilon)dt+ \sum_{k=1}^{N_1}Y_k^h(g_t^\epsilon) \circ dw_t^k
+{1\over \epsilon_0}Y_0^v (g_t^\epsilon)dt+
   \sum_{l=1}^{N_2} {1\over \epsilon_l}Y_k^v(g_t^\epsilon)\circ db_t^l, \quad
  g_0^\epsilon=g_0.
\end{aligned}}
\end{equation} 
Then the horizontal lift of $\pi(g_t^\epsilon)$ satisfies the  following system of equations
\begin{equation}
{\begin{aligned}
\label{horizontal}
d\tilde x_t^\epsilon  =&\sum_{k=1}^{N_1} \sum_{i=p+1}^n c_k^i(\tilde x_t^\epsilon a_t^\epsilon)
\left( \Ad(a_t^\epsilon) X_i\right)^*( \tilde x_t^\epsilon) \circ   dw_t^k\\
&+\sum_{i=p+1}^n 
c_0^i(\tilde x_t ^\epsilon a_t^\epsilon) \left( \Ad(a_t^\epsilon) X_i\right)^*( \tilde x_t^\epsilon) \,dt, \\
d a_t^\epsilon=& \sum_{l=1}^{N_2} {1\over \epsilon_l}\sum_{j=1}^{p} c_l^j(\tilde x_t^\epsilon a_t^\epsilon)A_j^*(a_t^\epsilon)\circ db_t^l
+\sum_{j=1}^{p} {1\over \epsilon_0} c_0^j(\tilde x_t^\epsilon a_t^\epsilon)A_j^*(a_t^\epsilon) \, dt,
\end{aligned}}
\end{equation}
up to an explosion time. Here $\tilde x_0^\epsilon =g_0$ and $a_0^\epsilon$ is the identity.
\end{corollary}

\begin{example}\label{Berger's-sphere}
Let us take the Hopf fibration $\pi: SU(2)\to S^2({1\over 2})$, given the bi-invariant metric. 
If we represent $SU(2)$ by the unit sphere in $\C^2$ and $S^2({1\over 2})$ as a subset in  $\R\oplus \C$,  the Hopf map
is given by the formula  $\pi(z,w)=({1\over 2}(|w|^2-|z|^2), z\bar w)$. It is a Riemannian submersion.
Let $\{X_1, X_2, X_3 \} $ be Pauli matrices defined by (\ref{pauli}) and let
 $\m=\<X_2, X_3\>$. Then $[\m, \h]\subset \m$. This is easily seen from the structure of the Lie bracket: $[X_1, X_2]=-2X_3,  \quad
[X_2, X_3]=-2X_1, \quad [X_3, X_1]=-2X_2$. 
Let $(g_t^\epsilon)$ be a solution to the equation (\ref{sphere-1})
 and $x_t^\epsilon=\pi(g_t^\epsilon)$. Denote by $(u_t^\epsilon)$ the horizontal lift of $(x_t^\epsilon)$.  By 
 Corollary \ref{Corollary3.3}, $(u_t^\epsilon, a_t^\epsilon)$ satisfies
$${\begin{aligned}du_t^\epsilon&=(\Ad(a_t^\epsilon)X_2)(u_t^\epsilon) \circ db_t^2+(\Ad(a_t^\epsilon)X_3)(u_t^\epsilon) \circ db_t^3,\\
da_t^\epsilon&={1\over \sqrt \epsilon} X_1(a_t^\epsilon) \circ db_t^1.
\end{aligned}}$$
 Since the metric is invariant by $\Ad_H$, $(u_t^\epsilon, t\ge 0)$  is a horizontal Brownian motion, its Markov generator is the horizontal Laplacian $\Delta^h=\trace \nabla_{\m}  d$.
Furthermore \linebreak $\{\pi_*( (\Ad(a_t^\epsilon)X_2)^*), \pi_*(  (\Ad(a_t^\epsilon)X_3)^*)\}$ is an orthonormal frame in $S^2({1\over 2})$ and then for each $\epsilon>0$,
 $(x_t^\epsilon, t\ge 0)$ is a Brownian motion on $S^2({1\over 2})$.   
 \end{example}

\begin{corollary}
\label{corollary3.5}
 Let $Y_0\in \m$, $A_0\in \h$ and $\{A_k, 1\le k\le N\}\subset \h$.  If there is an $\Ad_H$ invariant inner product on $\g$,  the following SDEs are conservative  for every $\epsilon$.
\begin{eqnarray}\label{3-6}\dot u_t^\epsilon&=&\left(\Ad(h_{t\over \epsilon}) Y_0\right)^*(u_t^\epsilon ), \\
\quad dh_t&=& \sum_{k=1}^NA_k^*(h_t)\circ db_t^k+A_0^*( h_t) dt
\label{3-6-2}\\
dg_t^\epsilon&=&Y_0 ^*(g_t^\epsilon)  dt+ {1\over \sqrt\epsilon} \sum_{k=1}^N A_k^*(g_t^\epsilon)\circ db_t^k+{1\over \epsilon} A_0^*(g_t^\epsilon) dt. \label{3-7}\end{eqnarray}
Furthermore $\pi(g_t^\epsilon)= \pi(u_t^\epsilon)$, and $(g_t^\epsilon)$ and $(u_t^\epsilon h_{t\over \epsilon})$ are equal in law.
 \end{corollary}
\begin{proof}  
With respect to the left invariant Riemannian metric on $G$ generated by the $\Ad_H$-invariant inner product, the random vector fields $(\Ad(h_{t\over \epsilon}^\epsilon) Y_0)^*$ are bounded, so  (\ref{3-6}) is conservative for almost surely all $\omega$.
That (\ref{3-6-2}) does not explode is clear, c.f. Lemma \ref{le-symmetric} below.
Since $\pi(g_t^\epsilon)=\pi(u_t^\epsilon)$, by Lemma \ref{lemma5.2}, $(\ref{3-7})$ is also conservative. 
 \end{proof}
 
{\bf Remark.} By the averaging principle we expect that the processes $\{u_\cdot^\epsilon, \epsilon>0\}$ converge to the solution of the ODE $$\dot u_t=\int_H (\Ad(h)(Y_0) )^*(u_t) \;dh.$$
where $dh$ is the Haar measure on $H$. This can be seen, assuming $G$ compact for simplicity, by sub-dividing $[0,\f t \epsilon]$ into sub-intervals of size $\epsilon^{-\delta}$ for a suitable positive number $\delta$ and by the law of large numbers for Brownian motions on a compact manifold, 
 stochastic averaging of stochastic slow-fast systems  is treated in \cite{Li-conservation}.  
Our main aim is to study the effective diffusion on the next time scale in case $\bar Y_0:=\int_H \Ad(h)(Y)  \;dh$ vanishes.
We do not need an explicit statement on the averaging principle, and will therefore refer the interested reader to \cite{Li-OM-1} and also to~\cite{Li-averaging}.


  \section{Elementary lemmas}
  \label{section-elementary}

   Let $\{X_k, k=1,\dots, m\}$ be smooth vector fields on a smooth manifold $N$. Denote by $\L$ the H\"ormander type operator ${1\over 2}\sum_{k=1}^m (X_k)^2+X_0$. 
  If at each point, $X_1$, $\dots$, $X_m$ and their Lie brackets generate the tangent space, we say that $\L$ satisfies strong H\"ormander's condition.  
 It satisfies H\"ormander's condition if  $X_0$ is allowed.    A H\"ormander type operator on a compact manifold satisfying strong H\"ormander's condition has a unique invariant probability measure $\pi$;
furthermore for $f\in C^\infty(N;\R)$, $\L F=f$ is solvable if and only if $\int f d\pi=0$. We denote by
 $\L^{-1} f$ a solution to the Poisson equation $\L F=f$ whenever it exists. 
 If a Markov operator $\L$ has a unique invariant probability measure $\pi$  and $f\in L^1(N; \pi)$ we write $\bar f=\int_N f d\pi$.

 \begin{lemma}
\label{le-symmetric}
Let $G$ be a Lie group with left invariant Riemannian metric.  Then an SDE driven by left invariant vector fields
is conservative. 
If $\mu$ is a right invariant measure on $G$ then left invariant vector fields are divergence free and the linear operator $\B={1\over 2}\sum_{i=1}^m (X_i)^2$, where $X_i\in \g$, is symmetric on $L^2(G; d\mu)$.\end{lemma}
\begin{proof}
The SDE is conservative follows from the fact that a Lie group with left invariant metric is geodesically complete 
and has positive injectivity radius. The left invariant vector fields and their covariant derivatives are bounded,
so by localisation or the uniform cover criterion in D. Elworthy \cite[ Chapt vii]{Elworthy-book}, solutions of equation (\ref{3-6-2}) from any initial point
exist for all time.
 If $f\in BC^1(G; \R)$, using the right invariance of the measure $\mu$,
$\int_G (Xf) (g)  \mu(dg) =\int_G {d\over dt} f(g \exp(tX))\Big|_{t=0} \mu(dg) =0$.
Consequently $X $ having vanishing divergence with respect to $\mu$ and for
  $f_1, f_2\in BC^1(G.; \R)$,
$\int_G  f_2 (X{f_1} ) d\mu=-\int_G f_1 (Xf_2)  d\mu$.
In particular $\B$ is symmetric  on $L^2(G, \mu)$.  
\end{proof}

We say a family of vectors $\{A_1, \dots, A_N\}$ in $\h\subset \g$ is  {\it Lie algebra generating}
if $\{A_1, \dots, A_N\}$ and their iterated brackets generate $\h$.
Define $\L_0={1\over 2}\sum_{k=1}^N (A_k )^2+A_0 $. 
We restrict $A_k $ to the compact manifold $H$ and treat $\L_0$  as an operator on $H$. 
If $\L_0$ is symmetric and satisfies H\"ormander's condition, the maximal principle states that $\L_0^* u=0$ has only constant solutions.
 \begin{lemma}
 \label{lemma4.6}
If $H$ is compact and $\{A_0, A_1, \dots, A_N\}\subset \h$ is Lie algebra generating,
 the following statements hold.
 \begin{enumerate}
 \item  The normalised Haar measure $dh$ is the unique invariant probability measure for $\L_0$,
and $\L_0$ is a Fredholm operator with Fredholm index $0$. 
\item If $\int_H \<\Ad(h)(Y_0), Y\>dh=0$, where $Y\in \g$,
there is a unique function $F\in C^\infty( H; \R)$ solving the Poisson equation $\L_0 F=\<\Ad(\cdot)(Y_0), Y\>$.
 \end{enumerate}
\end{lemma}
\begin{proof}
Denote by $\L_0^*$ the dual of $\L_0$ on $L^2(H; \R)$ which, by Lemma \ref{le-symmetric}, is $\L_0^*=\sum_{k} (A_k )^2-A_0$.
Both $\L_0$ and $\L_0^*$ satisfies H\"ormander's condition.
We have seen that $\int_H \L_0f dh$ vanishes for all $f\in C^\infty$ and $dh$ is an invariant measure, with full topoligical support. Distinct ergodic invariant measures have disjoint supports, and since every invariant measure is a convex combination of ergodic invariant measures,
the Haar measure is therefore the only invariant measure, up to a scaling. 
 Also $\L_0$ satisfies a sub-elliptic estimate:
$\|u\|_s\le \|\L_0 u\|_{L^2}+c\|u\|_{L^2}$, \cite[L. H\"ormander]{Hormander-acta}, which implies that $\L_0$ has compact resolvent,
and $\L_0$ is a Fredholm operator. In particular $\L_0$ has closed range, \cite[L. H\"ormander ]{Hormander-book} 
and $\L_0u =\<\Ad(h)(Y_0), Y\>$ is solvable if and only if $\<\Ad(h)(Y_0), Y\>$ annihilates the kernel of $\L_0^*$, i.e. $\int_H \<\Ad(h)(Y_0), Y\>dh$ vanishes. 
By the earlier argument the dimension of the kernels of $\L_0$ and $\L_0^*$
agree and $\L_0$ has Fredholm index $0$.
\end{proof}

\section{Diffusion creation and  rate of convergence}
    \label{section-convergence}
    
    Let  $(h_t)$ be a Markov process on a compact manifold $H$ with generator $\L_0=\sum_k (A_k)^2+A_0$ where $A_k$ are smooth vector fields satisfying H\"ormander's condition, and a invariant probability measure $\mu$.
  Let  $\Phi^\epsilon_t(y)$ be the solution to a family of conservative random differential equations
$\dot y_t^\epsilon =\sum_{k=1}^m \alpha_k(h_{\f t \epsilon})Y_k(y_t^\epsilon)$ with  $y_0^\epsilon=y_0$ and $Y_k$ smooth vector fields.

\begin{lemma}
\label{thm-weak}
Suppose that $N$ is compact ; or satisfies the following conditions.
\begin{itemize}
\item The injectivity radius of $N$ is greater than a positive number $2a$. 
\item For a Riemannian distance function $\rho$ on $N$,
$$C_1(p):=\sup_{s,t\le 1}\sup_{x\in N} \E \left(|Y_j(y^\epsilon_{\f t \epsilon} )|^p
\1_{\rho(y^\epsilon_{\f t \epsilon}, x)\le 2a} \right) <\infty,$$ for all $p$; also
 $C_2(p):=\sup_{s,t\le 1}\sup_{x\in N} \E \left(|\nabla d \rho^2(y^\epsilon_{\f t \epsilon}, x )|^p
\1_{\rho(y^\epsilon_{\f t \epsilon}, x)\le 2a} \right)$ is finite.
\end{itemize} 
Then  $(y_{t\over\epsilon}^\epsilon)$ converge weakly to a Markov process with generator $-\sum_{i,j}\overline{\alpha_i \L_0^{-1}\alpha_j}L_{Y_i}L_{Y_j}$. 
  \end{lemma}
\begin{proof}
Let $\beta_j=\L_0^{-1}\alpha_j$
and we first prove the tightness of the family of stochastic processes$\{y^\epsilon_{\f t \epsilon}, \epsilon>0\}$. Note that
$$\rho(y_{\f s \epsilon}^\epsilon,  y_{\f t \epsilon}^\epsilon)=\int_{\f s {\epsilon}}^{\f t \epsilon}
\nabla \rho ( y_{\f s \epsilon}^\epsilon, y_{r}^\epsilon)
 \left(\sum_{k=1}^m \alpha_k(h_{\f r \epsilon}) Y_k(y_r^\epsilon)\right)dr, $$
where the gradient is on the second variable. At the first glance we expect that
$\rho(y_{\f t \epsilon}^\epsilon, y_{\f s \epsilon}^\epsilon)$ to be of oder $\f {t-s} \epsilon$, we will use a trick, the It\^o trick, to overcome this problem, see
 \cite{Li-averaging} where it was also used.  We may assume that $(h_t)$ solves (\ref{3-6-2}). Since  $\L_0 \beta_j=\alpha_j$, 
$$\beta_j(h_{\f t \epsilon})=\beta_j(h_0)+\int_0^{\f t \epsilon} \alpha_j (h_r) dr
+M^\epsilon_{\f t \epsilon} 
$$
where $M^\epsilon_{\f t \epsilon} $ denotes a local martingale, and $ \int_0^{\f t \epsilon} \alpha_j (h_r) dr$ is of order $1$, modulus the fast oscillating local martingale. To use this in our setting let us take  a $C^2$ function $f: G\to \R$, by the product rule:
\begin{equation*} 
\label{product}
{\begin{split} 
df(Y_j(y^\epsilon_{t\over \epsilon} ) )\beta_j(h_{t\over \epsilon^2})
&= df(Y_j(y^\epsilon_{s\over \epsilon} ))\beta_j( h_{s\over \epsilon^2})
+\sum_{j=1}^m \int_{s\over \epsilon}^{t\over \epsilon} L_{Y_i}L_{Y_j} f(y^\epsilon_r)
\,\alpha_i(h_{\f r \epsilon})\;\beta_j( h_{\f r \epsilon})\; \;dr \\
&+ {1\over \sqrt \epsilon} \sum_{k=1}^{m'} \int_{s\over \epsilon}^{t\over \epsilon}  
L_{Y_j} f( y^\epsilon_r) \, d\beta_j\left( A_k (h_{\f r \epsilon})\right)db_r^k
+{1\over \epsilon} \int_{s\over \epsilon}^{t\over \epsilon}  L_{Y_j} f(y^\epsilon_r) \,\L_0 \beta_j(h_{\f r \epsilon})dr .
\end{split} }
\end{equation*}

Since the Riemannian distance function fails to be $C^2$ if the two points are on  the cut locus of each other, we  take a smooth function $\phi: \R_+\to \R_+$ such that
$\phi(r)=r$ if $r<a$ and $\phi(r)=1$ for all $r>2a$ and define $\tilde \rho=\phi\circ \rho$. We inverse engineer with the last term in the equation above, replacing $f$ by $\tilde \rho^2$ to see that, c.f. \cite[Lemma 3.1]{Li-limits},  
 for $s,t\le 1$,
$${ \begin{split}
&\tilde \rho^2\left(y^\epsilon_{s\over \epsilon}, y_{t\over \epsilon}^\epsilon\right)
=\int_{s\over \epsilon}^{t\over \epsilon} \sum_{k=1}^m L_{Y_k(y_r^\epsilon)}(\tilde \rho^2)\left(y^\epsilon_{s\over \epsilon}, y_{r}^\epsilon\right)\left(  \alpha_k(h_{\f r \epsilon})\right)dr
 \\
=& \epsilon \sum_{j=1}^m 
 \left(L_{Y_j}\tilde \rho^2(y^\epsilon_{s\over \epsilon}, y_{t\over \epsilon}^\epsilon)\right) 
\;\beta_j(h_{\f t {\epsilon^2}})
-\epsilon \sum_{i,j=1}^m\int_{s\over \epsilon}^{t\over \epsilon} 
  \left(L_{Y_i}L_{Y_j} \tilde \rho^2( y^\epsilon_{s\over \epsilon}, y^\epsilon_r)\right)\;
\alpha_i(h_{\f r {\epsilon}})\;\beta_j(h_{\f r {\epsilon}})\; \;dr\\
&-  \sqrt \epsilon  \sum_{j=1}^m\sum_{k=1}^{m'}
\int_{s\over \epsilon}^{t\over \epsilon}  (L_{Y_j} \tilde \rho^2)(y^\epsilon_{\f s \epsilon}, y^\epsilon_r) \;
(L_{A_k} \beta_j)(h_{\f r {\epsilon}}) \;db_r^k.
\end{split} } $$
We raise both sides to the power $p$ where $p> 2$ to see for a constant $c_p$ depending on $|\beta_j|_\infty$, $|\alpha_j|_\infty$, $|A_k|_\infty$, $m$, and $p$,  $C_1$ and $C_2$, which may represent a different number in a different line,
$${ \begin{split}
&\E\left[ \tilde \rho^{2p}\left(y^\epsilon_{s\over \epsilon}, y_{t\over \epsilon}^\epsilon\right)\right]
 \\
\le & c_p\epsilon^p \sum_{j=1}^m 
 \E \left|L_{Y_j}\tilde \rho^2(y^\epsilon_{s\over \epsilon}, y_{t\over \epsilon}^\epsilon)\right|^p
+c_p\epsilon^p \sum_{i,j=1}^m \E \left(\int_{s\over \epsilon}^{t\over \epsilon} 
  \left|L_{Y_i}L_{Y_j} \tilde \rho^2( y^\epsilon_{s\over \epsilon}, y^\epsilon_r)\right|
 \;dr\right)^p\\
&+ c_p \epsilon^{\f p  2} \sum_{j=1}^m\sum_{k=1}^{m'}
\E \left( \int_{s\over \epsilon}^{t\over \epsilon}  |L_{Y_j} \tilde \rho^2(y^\epsilon_{\f s \epsilon}, y^\epsilon_r))|^2dr\right)^{\f p2}\\
&\le c_p C_1(p) \epsilon^p+c_p  (t-s)^{p}\sqrt{C_1(4p)} \sqrt{C_2(2p)}
+c_p(t-s)^{\f p 2}\sqrt{C_1(2p)}\sqrt{C_2(2p)}.
\end{split} } $$
Applying $\tilde \rho$ directly to $y_{t\over \epsilon}^\epsilon$ giving another estmate:
$$\begin{aligned}
\E \left|\tilde \rho^2\left(y^\epsilon_{s\over \epsilon}, y_{t\over \epsilon}^\epsilon\right)\right|^p
&\le c_p \E\sum_j\left| \int_{s\over \epsilon}^{t\over \epsilon}
L_{Y_j} \tilde \rho^2 (y^\epsilon_{\f s \epsilon}, y_{r\over \epsilon}^\epsilon) \alpha_j(h_{\f r \epsilon})dr\right|^p\\
&\le c_p \left(\f {t-s} \epsilon\right)^p \sqrt {C_1(2p)} \sqrt {C_2(2p)}.
\end{aligned}$$
Interpolate the estimates for  $\epsilon^p \le (t-s)^{\f p 2}$ and for $\epsilon^p \ge (t-s)^{\f p 2}$, to see
$$ \E \left|\tilde \rho\left(y^\epsilon_{s\over \epsilon}, y_{t\over \epsilon}^\epsilon\right)\right|^{2p}
\le c_p |t-s|^{\f p 2}.$$
Taking $p>4$ and applying Kolmogrov's theorem we obtain the required tightness. The weak convergence follows just as for the proof of
 \cite[Theorem 5.4]{Li-limits},  using a law of large numbers with rate of convergence the square root of time 
 \cite[Lemma 5.2]{Li-limits}.  For a Fredholm operator $\L_0$ of index zero, the limit is identified as following.
  Let $\{u_i, i=1, \dots, n_0\}$ be a basis of $\ker ( \L_0)$ 
and $\{\pi_i\, i=1, \dots, n_0\}$ the dual basis for the null space of $\L_0^*$.    Then
$\bar \L=-\sum_{i,j}\sum_{b=1}^{n_0} u_b \< \alpha_i \beta_j ,\pi_b\>L_{Y_i}L_{Y_j}$,
where the
bracket denotes the dual pairing between $L^2$ and $(L^2)^*$.
In our case there is only one invariant probability measure for $\L_0$, from which we conclude that $\bar \L=-\sum_{i,j}\overline{ \alpha_i \beta_j }L_{Y_i}L_{Y_j}$.
 \end{proof}

Let use return to our equations on the product space $G\times H$,
 \begin{equs}
\label{5-1-1}\dot u_t^\epsilon&=\left(\Ad(h_{t\over \epsilon}) Y_0\right)^*(u_t^\epsilon ),  \quad u_0^\epsilon=u_0 \\
\quad dh_t&= \sum_{k=1}^{N}A_k^*(h_t)\circ db_t^k+ A_0^*( h_t) dt, \quad \h_0=1.
\label{5-1-2}\end{equs} 
    Let $x_{\cdot\over \epsilon}^\epsilon=\pi(u_{\cdot\over \epsilon}^\epsilon)$ and $x_0=\pi(u_0)$. 
Let $\<,\>$ denote a left invariant and $\Ad_H$ invariant scalar product on $\g$,  $\{Y_j\}$ an orthonormal basis of $\m$ and
define
\begin{equation}
 \alpha(Y_0, Y_j)(h)=\left\< \Ad(h)(Y_0), Y_j\right \>.
\end{equation} 
 Also denote by $dh$ the normalised Haar measure on $H$, we define $$\bar Y_0\equiv \int_H \Ad(h)(Y_0)dh.$$

  \begin{proposition}
\label{thm-limit}
Suppose that the subgroup $H$ is compact, $\{A_0, A_1, \dots, A_N\}$ is a Lie algebra-generating subset of $\h$, and $Y_0\in \m$ is such that $\bar Y_0=0$. Let $T$ be a positive number. Then  as $\epsilon \to 0$,  $(u_{s\over \epsilon}^\epsilon, s\le T) $ converge
 weakly to a Markov process $(\bar u_s, s\le T)$, whose  Markov generator is 
$$\bar \L =-\sum_{i,j=1}^m \overline{ \alpha(Y_0, Y_i) \;(\L_0^{-1} \alpha(Y_0, Y_j)) }\;
L_{Y_i^*}L_{Y_j^*}.$$
Also, $(x_{s\over \epsilon}^\epsilon, s\le T)$ converges weakly 
to a stochastic process $(\bar x_s, s\le T)$.  
\end{proposition}

\begin{proof}
By the left  invariance of the Riemannian metric, Equation (\ref{5-1-1}) is equivalent to
 $u_t^\epsilon= \sum_{j=1}^N \alpha(Y_0, Y_j)(h_{t\over \epsilon})Y_j^*(u_t^\epsilon )$ and
  \begin{equation*}
(\Ad(h)(Y_0))^*(g)=\sum_{j=1}^N \left\<( \Ad(h)(Y_0))^*, Y_j^*\right\> Y_j^*(g)
=\sum_{j=1}^N\alpha(Y_0, Y_j)(h) Y_j^*(g).
\end{equation*}
We may rewrite (\ref{5-1-1}) as
$ \dot u_t^\epsilon =\sum_{j=1}^N Y_j (u_t^\epsilon)\alpha(Y_0, Y_j) ( h_{t\over \epsilon})$.
Since $\bar Y_0$ vanishes  apply Lemma \ref{lemma4.6}, so
 $\L_0^{-1}\alpha_j$ exists and is smooth for each $j$.
Furthermore by Lemma \ref{le-symmetric}, the $\bar \L$ diffusions exist for all time.
Since the Riemannian metric on $G$ is left invariant, its Riemannian distance function $\rho$ is also left invariant, $\rho(gg_1, gg_2)=\rho(g_1,g_2)$ for any $g, g_1, g_2 \in G$. Furthermore,
 $\sup_{p:\rho(p,y)<\delta} \nabla^2 \rho(y, p)$ is finite, for sufficiently small  $\delta>0$, and is independent of $y$, where $\nabla$ is the Levi-Civita connection. Thus $C_1(p)$ and $C_2(p)$, from Lemma \ref{thm-weak},
  are both finite. We observe that $G$ has positive injectivity radius and bounded geometry, and we  then apply Lemma \ref{thm-weak} to conclude that  $(u^\epsilon_{s\over \epsilon}, s\le t)$ converges weakly as $\epsilon$ approaches $0$.
 As any continuous real valued function $f$ on $M$ lifts to a continuous function on $G$, the weak convergence passes, trivially,  to the processes $(x_{t\over \epsilon}^\epsilon)$. 
 \end{proof}
 
 For $p\ge 1$, denote by $W_p$ or $W_p(N)$ the Wasserstein distance on probability measures over a  metric space $N$. For $T$  a positive number,
let $C_x([0,T];N)$ denote the space of continuous curves defined on the interval $[0,T]$ starting from $x$.
Given two probability measures $\mu_1, \mu_2$ and $C_x([0,T];N)$,
$$W_p(\mu_1, \mu_2):=\left(\inf \int\sup _{0\le s\le T} \rho^p(\sigma_1(s),\sigma_2(s)) d\nu(\sigma_1, \sigma_2)\right)^{1\over p}.$$
Here $\sigma_1$ and $ \sigma_2$ take values in  $C_{x}([0,T];N)$ and 
the infimum is taken over all probability measures $\nu$  whose marginals are $\mu_1, \mu_2$.

\begin{remark}
The weak convergence of a sequence of stochastic processes $\{y^\epsilon_{\f s \epsilon}\}$, in the case where the limit belongs to $W_p$ and that  $ \E\sup_{s\le t} \rho(0, y^\epsilon_{\f s \epsilon})^p$ is uniformly bounded in $\epsilon$,
 implies their convergence in the Wasserstein $p$ distance. Thus, if $G$ is compact or has negative curvature and s.t. every point is a pole, 
our processes converge in the Wasserstein distance. The compact case is trivial. For the latter just note that in the proof of the lemma, $\tilde \rho^2$ can be replaced by $\rho^2$ and the left invariant vector fields are bounded. Since the sectional curvature is bounded,  $|\nabla d \rho^2|$ is bounded by curvature comparison theorem. Thus
$\rho(0, y^\epsilon_{\f s \epsilon})^2$ has finite moments, the assertion follows.
\end{remark}

\subsection{Rate of convergence.}

The definition of $BC^r$ functions on a manifold depend on the linear connection on $TM$, in general. 
We use the flat connection $\nabla^L$ for $BC^r$ functions on the Lie group $G$. 
The canonical connection $\nabla^c$ is a convenient connection for defining $BC^r$ functions on $M=G/H$, as  the parallel transports for $\nabla^c$ are differentials of the left actions of $G$ on $M$. See \S\ref{section8} below for more discussions on the canonical connection.
 In the theorem below,   $M$ is given the induced $G$-invariant Riemannian metric. 
Set \begin{equation*}
{\begin{aligned}
BC^r(M;\R)&=\{ f\in  C^r(M;\R):  |f|_\infty+ |\nabla f|_\infty+\sum_{k=1}^r  |(\nabla^c)^{(k)} df|_\infty<\infty  \}\\
BC^r(G;\R)&=\{ f\in  C^r(G;\R):  |f|_\infty+ |\nabla f|_\infty+\sum_{k=1}^r  |(\nabla^L)^{(k)} df|_\infty<\infty  \}.
\end{aligned}}
\end{equation*}
Denote  $|f|_{r, \infty}=\sum_{k=0}^r |\nabla^{(k)} df|_\infty$ where $\nabla$ is one of the connections above.
 \begin{remark}
  A function $f $ belongs to $BC^r(G)$ if and only if
  for an orthonormal basis $\{Y_i\} $ of $\g$, the following functions are bounded
  for any set of indices: $ f$, ${Y^*_{i_1}}f$, \dots, 
  ${Y^*_{i_1}}\dots {Y^*_{i_r}}f$.  In fact $\nabla f=\sum_i d f(Y_i^*)Y_i^*$ and  $\nabla^L_\cdot  \nabla f=\sum_i  Y_i^* L_\cdot  ({Y_i^*} f) $,
 and so on.  Here $L_\cdot$ denotes Lie differentiation so $L_\cdot f=df(\cdot)$.
 \end{remark}

 Denote by $\bar \mu$  the probability measure of $(\bar u_\cdot )$ and $\bar P_t$ the associated probability semigroup.  We give the rate of convergence, which essentially follows from
 \cite[Theorem 7.2]{Li-limits}, to follow which we would assume one of the following conditions. (1) $G$ is compact; 
 (2)  For some point $o\in G$, $\rho^2_o$ is smooth and $|\nabla^L d\rho^2_o|\le C+K\rho^2_o$.
(3) 
 There exist  $V\in C^2(M;\R_+)$, $c>0$, $K>0$ and $q\ge 1$ such that
$$|\nabla V|\le C+KV, \quad |\nabla dV|\le c+KV, \quad |\nabla d\rho^2_o|\le c+KV^q.$$
However it is better to prove the rate of convergence directly.

\begin{theorem}
\label{thm-limit-2}
Suppose that $H$ is compact and $\{A_1, \dots, A_N\}\subset$ is  Lie algebra generating and $\bar Y_0=0$.  
Then,
(1) Both $(u_{t\over \epsilon}^\epsilon, t \le T)$ and $(x_{t\over \epsilon}^\epsilon, t\le T)$ converge in $W_p$ 
for any $p>1$.\\
(2) There are numbers $c$ such that for all $f\in BC^4(G; \R)$,
  $$\sup_{0\le t\le T}\left|\E f\left( u_{t\over \epsilon}^\epsilon\right)-\bar P_tf (u_0)\right|
\le C \,\epsilon\, \sqrt{|\log \epsilon|} \;  (1+|f|_{4, \infty}), \quad u_0\in G.$$
(3) For  any $r\in (0,{1\over 4})$, 
$\sup_{0\le t\le T}W_1 (\law(u_{t\over \epsilon}^\epsilon), \bar\mu_t) \le C\epsilon^r$.
An analogous statement  holds for $x_{t\over \epsilon}^\epsilon$.
\end{theorem}
\begin{proof}
 If $\rho^2$ is bounded or smooth, $L_{Y_i^*}L_{Y_j^*} \rho=\nabla^L d\rho (Y_i^*, Y_j^*)+ \nabla^L _{Y_i^*}Y_j^*$ is bounded. Denote $\T^L$ the torsion of $\nabla^L$ and $\nabla$  the Levi-Civita connection. The derivative flow for  (\ref{5-1-1}) satisfying the following equation $$\nabla v_t^\epsilon= -\f 12\sum_j  \alpha(Y_0, Y_j) (h_t^\epsilon) \,\T^L(v_t^\epsilon, Y_j(u_t^\epsilon)).$$
Indeed this follows from linearising the equation $ \dot u_t^\epsilon =\sum_{j=1}^N Y_j ^*(u_t^\epsilon)\alpha(Y_0, Y_j) ( h_{t\over \epsilon})$ and the relation $\nabla^L=\nabla+\f 1 2\T^L$.
 Since $\T^L(u,v)=-[u,v]$, the torsion tensor and their covariant derivatives are bounded. 
 Thus all covariant derivatives of $Y_j^*$ with respect to the Levi-Civita connection are bounded.
 We use Lemma \ref{thm-limit} and follow the proof of \cite[Theorem 7.2]{Li-limits},
 to conclude the first assertion for the $u$ process. These are proved by discretising time and 
 writing the differences as telescopic sums. The main ingredients are:
 (1)  $(h_t)$ has an exponential mixing rate which follows from H\"ormander's conditions,
 (2) estimates for $|\bar P_tf- \bar P_s f-(t-s) \bar P_s \bar  \L f | \le C (t-s)^2$ which follows from 
 the fact that the vector fields $Y_j$ have bounded derivatives of all order.

 For the process on the homogeneous manifold, take $f\in BC^4(M;\R)$ and let $\tilde f=f\circ \pi$.
 Then ${Y_i}{Y_j} \tilde f= \nabla^c d f(\pi_*Y_i^*, \pi_*Y_j^*)+df\left(\nabla^c_{Y_i^*} d\pi (Y_j^*)\right)$.
 The last term vanishes, as $  (d\pi)_u(Y_i^*) =T\bar L_u (  (d\pi)_1Y_i))$, c.f. Lemma \ref{nabla} below.
 Since $\pi_*$ is a Riemannian isometry, ${Y_i^*}{Y_j^*} \tilde f$ is bounded if $\nabla^cdf$ is. The same assertion holds for higher order derivatives. Thus $f\circ \pi \in BC^4(G; \R)$ and
$$\sup_{0\le t\le T}\left|\E f\left( x_{t\over \epsilon}^\epsilon\right)-\pi_*\bar P_tf (x_0)\right|
\le  C\,\epsilon\, \sqrt{|\log \epsilon|}\; \gamma(x_0) \;  (1+|f|_{4, \infty}),$$
where $\gamma$ is a function in $B_{\rho,o}$.

For the convergence in the Wasserstein distance, denote by $\tilde \rho$ and $\rho$ respectively 
 the Riemannian distance function on $G$ and on $M$.
For $i=1,2$, let $x_i\in M$ and $u_i\in \pi^{-1}(x_i)$. If $u_1$ and $u_2$ are the end points of a horizontal lift of the unit speed geodesic connecting $x_1$ and $ x_2$, then $\tilde \rho(u_1, u_2) = \rho(x_1, x_2)$. Otherwise  $\tilde \rho(u_1, u_2) \ge \rho(x_1, x_2)$.
If $c_1$ and $c_2$ are $C^1$ curves  on $M$ with $c^1$ curves $\tilde c_1$ and $ \tilde c_2$ on $G$ covering $c_1$
and $c_2$ respectively, then $ \rho(c_1, c_2)=\sup_{t\in [0,T]} \rho(c_1(t), c_2(t))\le   \tilde \rho(\tilde c_1, \tilde c_2)$.

Since $C_{u_0}([0,T];G)$ is a Polish space, there is an optimal  coupling  of 
 the probability law of $u_{\cdot \over\epsilon}^\epsilon$ and $\bar \mu$ which we denote by $\mu$.
 Then $\pi_*\bar\mu$ is a coupling of $\law(x_{\cdot\over \epsilon}^\epsilon)$ 
 and $\pi_*\bar\mu$ and
\begin{equation*}
{\begin{aligned}
W_p(\law(u^\epsilon_{\cdot \over \epsilon}), \bar \mu)
&=\left(\int_{C_{u_0}([0,T];G)}\tilde \rho^p (\gamma_1, \gamma_2) d\mu(\gamma_1, \gamma_2)\right)^{1\over p}
\\
&\left(
 \int_{C_{x_0}([0,T];M)}\rho^p ( \pi(\gamma_1), \pi(\gamma_2) ) d\mu(\gamma_1, \gamma_2)\right)^{1\over p}
\\
&= \left(
 \int_{C_{x_0}([0,T];M)} \sup_{0\le t\le T} \rho^p(\sigma_1(t), \sigma_2(t)) d\pi_*\bar \mu (\sigma_1, \sigma_2) \right)^{1\over p}
 \\
&\ge  W_p(\law (x^\epsilon_{\cdot \over \epsilon}), \pi_*\bar \mu).
\end{aligned}}
\end{equation*}
Consequently $x_{\cdot \over\epsilon}^\epsilon$ converges in $W_p$. Similarly
the rate of convergence passes from the $u$ process to the $x$ process. For $r<{1\over 4}$ there is a number $C$ such that,
$W_1(\law(x_{t\over \epsilon}^\epsilon), \pi_*\bar \mu_t)\le W_p(\law(u^\epsilon_{t \over \epsilon}), \bar \mu)\le C \gamma(x_0)\epsilon^r$.
\end{proof}

 \subsection{The Centre Condition}
 \label{sectionConditions}

We identify those vectors $Y_0$ satisfying $\bar Y_0=0$. Given a reductive structure,  $\Ad_H$ is the direct sum of sub representations $\Ad_H=\Ad_H|_{\h}\oplus \Ad_H|_\m$. 
The condition $\Ad(H)(\m)\subset\m$ implies that $[\h, \m]\subset \m$ and $\Ad(H)(\m)=\m$.
Let
$$\m_0=\{X\in \m: \Ad(h)(X)=X \text{ for all }  h\in H\},$$ be the subspace on which $\Ad_H$ acts trivially. 
We consider an $\Ad_H$-invariant subspace space $\tilde \m$ of $\m$, transversal to $\m_0$, i.e. $\tilde\m\cap \m_0=\{0\}$.  It is  more intuitive to study the action of $\Ad_H$
 through its isotropy representation $\tau$ of $H$ on $T_oM$ which is defined by the formula $\tau_h(\pi_* X):=(\bar L_h)_* (\pi_*X)$.  
  Since left translation on $T_oM$ corresponds to adjoint action on $\m$, then
 $\pi_*\Ad(h)(X)$ agree with $(\bar L_h)_*(\pi_*X)$,  the linear representation $\Ad_H$ is equivalent to the isotropy representation. 
 A representation $\rho$ of $H$ is said to acts transitively on the unit sphere, of the representation space $V$,
if  for any two unit vectors in $V$ there is a $\rho(h)$ taking one to the other. If $\Ad_H$ acts transitively, its representation space is irreducible and $\m_0=\{0\}$.

Recall that $H$ is unimodular is equivalent to that the Haar measure $dh$ is bi-invariant.  

\begin{lemma}
\label{le:transitive}
Suppose that $H$ is uni-modular.
\begin{itemize}
\item 
[(1)] If $Y\in \m_0$ is non-trivial,  then $\bar Y$ does not vanish.
\item[(2)]  $\bar Y=0$ for every  $Y\in \tilde \m$.
In particular,  if $\Ad_H|_\m$ has no non-trivial invariant vectors,  then $\bar Y=0$ for all $Y\in \m$. 
\end{itemize}
\end{lemma}
\begin{proof}
Part (1) is clear by the definition. For any $Y\in\tilde  \m$, the  integral $\bar Y:=\int_H \Ad(h)(Y)dh$ is an invariant vector of $\Ad_H$, using the bi-invariance of the measure $dh$. Since $\bar Y\in  \tilde \m$ and $\tilde  \m\cap \{\m_0\}=\{0\}$ the conclusion follows.
 \end{proof}

  \begin{example}
 Let $M$ be the Stiefel manifold $S(k,n)$ of oriented $k$ frames in $\R^n$.
 The orthogonal group takes a $k$-frame to a $k$ frame and acts transitively.
 The isotropy group of the $k$-frame $o=(e_1, \dots, e_k)$, the first $k$ vectors from
 the standard basis of $\R^n$, contains rotation matrices that keep the
 first $k$ frames fixed and rotates the rest. Hence,
 $S(k,n)=SO(n)/SO(n-k)$. Then
 $\h=\left\{\left(
 \begin{matrix} 0&0\\0&A\end{matrix} \right)  \right\}$ where $A\in \so(n-k)$.
 Denote by $ M_{n-k, k}$ the set of $(n-k)\times k$ matrices and let
 $ \m=\left\{\left(
 \begin{matrix} S&-C^T\\C&0\end{matrix} \right)  \right\}$ where $S\in \so(k)$
  and $C \in M_{n-k, k}$.
Since $\Ad\left(   \left(
 \begin{matrix} 1&0\\0&R\end{matrix} \right)\right)  \left(\left(
 \begin{matrix} S&-C^T\\C&0\end{matrix} \right)\right)=\left(
 \begin{matrix} S&-C^TR^T\\ RC&0\end{matrix} \right)$ for $R\in SO(n-k)$, 
we see  $ \m_0=\left\{\left(
 \begin{matrix} S&0\\0&0\end{matrix} \right) \right\}$. Define
  $  \tilde \m=\left\{\left(
 \begin{matrix} 0&-C^T\\C&0\end{matrix} \right)  \right\}$.
Let us identify $\tilde \m$ with $(n-k)\times k$ matrices. 
For $C\in M_{n-k, k}$ denote by $Y_0$ the corresponding skew symmetric matrix in $\tilde  m$.
It is clear that $\int_H RC dR$ and $\int_H \Ad(R)(Y_0) dR$ vanish.
 \end{example}
 
The symmetry group of a  Riemmannian homogeneous space of dimension $d$ has at most dimension $d(d+1)/2$.
If $G/H$ is a connected $d$-dimensional manifold and if $G$ admits an $\Ad_H$-invariant inner product with $\dim(G)=\f 12 d(d+1)$, then $\bar Y$ vanishes for all $Y\in \m$. Such Riemannian homogenous manifolds are of constant curvature, isometric to one of the following spaces: an Euclidean space, a sphere, a real projective space, and a simply connected hyperbolic space.

\begin{example}
Suppose that $M=G/H$ is a symmetric space,
 it is in particular a reductive homogeneous space.
 Then $L_h \pi=\pi \Ad(h)$
and  $(dL_h)_o d\pi =d\pi \Ad_h$, where $o$ is the identity coset. Since $d\pi$ restricts to an isomorphism
on $\m$, left actions correspond to the $\Ad$ action on $G$. For every $Y_0$ there 
is an $h$ with the property that $\Ad(h)(Y_0)=-Y_0$.
A compact Lie group with the bi-invariant metric is a symmetric space.
\end{example}

\begin{example}
A more general class of  manifolds, for which  $\bar Y_0=0$ for every  $Y_0\in \m$,
are  weakly symmetric spaces among the Riemannian homogeneous spaces.
A Riemannian manifold is weakly symmetric 
if there exists some closed subgroup $K$ of the Isometry group $I(M)$ and some $\mu\in I(M)$ such that $\mu^2\in K$ and for all $p,q\in M$ there exists $I\in K$ such that $I(p)=\mu(q)$ and  $I(q)=\mu(p)$. This concept was introduced in \cite[A. Selberg]{Selberg}
and studied by Szab\'o who characterises them as ray symmetric space.
Also, a Riemannian manifold is weakly symmetric if and only if any two points  in it can be interchanged by some isometry, see  \cite[J. Berndt, L. Vanhecke]{Berndt-Vanhecke}.  The following property characterizes weakly symmetric spaces among Riemannian homogeneous spaces: let $H$ be the isotropy subgroup of $I(M)$ 
and $\tau: H\to GL(T_oM)$ its isotropy representation. Then for each $X\in \h$ there exists an element of $h\in H$ such that 
$\rho(h)(X)=-X$.    Some of the weakly symmetric spaces are given in the form of
$G/H$ where $G$ is a connected semi simple Lie group, $K$ a closed subgroup of $G$, and 
$H$ a closed subgroup $K$. Then there  is a principal fibration $G/H\to G/K$ with fibre $K/H$.
\end{example}

The three spheres with the left invariant Berger's metrics, $\epsilon\not =1$, are not locally Riemannian symmetric spaces, they are reductive homogeneous spaces. 
We remark also that irreducible symmetric spaces are in fact strongly
isotropy irreducible. A connected homogeneous manifold $G/H$ is strongly
isotropy irreducible if $H$ is compact and its identity component acting irreducibly on the tangent space. Such spaces admit left invariant Einstein metric and are completely classified, for more study see \cite{Besse} and  see also \cite{Burstall-Rawnsley}. If the symmetric space has rank $1$, i.e. it is the quotient space of a semi-simple Lie group $G$ whose maximal torus group $H$ is of dimension $1$, they are two point homogeneous spaces, see \cite[Prop. 5.1]{Helgason}.

\begin{theorem}
\label{convergence-and-limit}
Suppose that $H$ is compact and $\{A_0, A_1, \dots, A_N\}$ is Lie algebra generating. Let $Y_0 \in \tilde\m$, $x_t^\epsilon=\pi(g_t^\epsilon)$, where $(g_t^\epsilon)$ is the solution to  (\ref{1}), and $ (u_t^\epsilon)$ its horizontal lift through $g_0$.   
Then $(u_{t\over \epsilon}^\epsilon) $ and $(x_{t\over \epsilon}^\epsilon) $ converge
weakly with respective limit diffusion process $(\bar u_t)$ and $(\bar x_t)$.
Let $\{Y_j, j=1,\dots, N\}$  be an orthonormal basis of $\tilde \m$. Then
the Markov generator of $(\bar u_t)$ is 
$$\bar \L =-\sum_{i,j=1}^{N}a_{i,j}(Y_0)\;L_{Y_i^*}L_{Y_j^*}, \quad a_{i,j}(Y_0)=\int_H  \alpha(Y_0,Y_i) \,\L_0^{-1} ( \alpha(Y_0, Y_j) ) dh.$$
  \end{theorem}
\begin{proof}
By Corollary \ref{corollary3.5}  the horizontal lift  process of $(x_t^\epsilon)$
solves equations  (\ref{5-1-1}-\ref{5-1-2}). By Lemma \ref{le:transitive}, $\bar Y_0=0$. Since $\tilde \m$ is an invariant
space of $\Ad_H$ and $Y_0\in \tilde \m$,   $\Ad(h)(Y_0)\in \tilde \m$. 
In Proposition \ref{thm-limit} it is  sufficient to take $\{Y_j\}$ to be an orthonormal basis of the invariant subspace $\tilde \m$.
\end{proof}

 In the rest of the paper we study the effective limits $\bar \L$ given in Theorem \ref{convergence-and-limit}. 
 Although $\bar u_t$ is always a Markov process,  that its projection $(\bar x_t)$ (describing the motion of the 
effective orbits) is a Markov process on its own is not automatic.  In the next two sections we study this and identify the stochastic processes
$(\bar u_t)$ and  $(\bar x_t)$ by computing the Markov generator $\bar \L$ and its projection.

 \section{Effective limits, casimir, and  Markov property}
 \label{section5}
In this section $H$ is a compact connected proper subgroup of $G$, $\L_0=\f 12 \sum_{k=1}^N (A_k)^2$, $\<, \>$ an $\Ad_H$-invariant inner product on $\g$ and 
   $\m=\h^\perp$. Let $\m=\m_0\oplus \m_1\oplus \dots \oplus \m_r$ be the orthogonal sum of $\Ad_H$-invariant subspaces,  where each  $\m_j$ is irreducible for $j\not =0$. This is also  an invariant decomposition for $\ad_\h$.

Let $\B_{\ad_\h, \m_l}(X,Y)$ be the symmetric associative bilinear form of  the adjoint sub-representation of 
$\h$ on $\m_l$. 
The symmetric bilinear form for a representation $\rho$ of a Lie algebra $\g$ on a 
finite dimensional vector space is given by the formula
$\B_\rho(X,Y)=\trace \rho(X)\rho(Y)$ for $ X, Y\in \h$. 
A bilinear form $\B$ on a Lie algebra is associative if $\B([X,Y],Z])=\B(X,[Y,Z])$.   Let $\id_{\m_l}$ be the identity map on~$\m_l$.
The following lemma allows us to compute the coefficients of the generator.

 \begin{lemma}\label{lambda}
 Let $\{A_j\}$ be an orthonormal basis of $\h$.
The following statements hold. \begin{enumerate}
\item There exists $\lambda_l$ such that
 ${1\over 2}\sum_{k=1}^p \ad^2(A_k)|_{\m_l}=- \lambda_l\id_{\m_l} $. Furthermore $\lambda_l=0$ if and only if $l=0$.

\item  Suppose that $\ad_\h:\h\to {\mathbb L}(\m_l;\m_l)$ is faithful. 
Then $\B_{\ad_\h, \m_l}$ is non-degenerate, negative definite, and $\Ad_H$-invariant.
If $\l\not =0$, 
$$\lambda_l=-\B_{\ad_h,\m_l}(X,X)  {\dim(\h) \over 2\dim(\m_l)},$$
where $X$ is any unit vector in $\h$.
If the inner product on $\h$ agrees with $-\B_{\ad_h,\m_l}$ then $\lambda_l= {\dim(\h) \over 2\dim(\m_l)}$.
\end{enumerate}
\end{lemma}

\begin{proof}  
Part (1). For $Z\in \h$ and $X, Y \in \g$,  we differentiate the identity
$$\<\Ad(\exp(tZ))(X), \Ad(\exp(tZ))(Y)\>=\<X,Y\>$$
 at $t=0$ to see that 
 \begin{equation}
 \label{ad-skew-symmetric}
 \<\ad(Z)(X), Y\>+\<X, \ad(Z)(Y)\>=0.
 \end{equation}
Let $\{A_1, \dots, A_p\}$ be an orthonormal basis of $\h$, for the $\Ad_H$ invariant metric, then $\sum_{k=1}^p (\ad(A_k))^2$ commutes with every element of $\h$. Indeed if $X\in \h$ let $[X,A_k]=\sum_{l'} a_{kl'}A_{l'}$, then
$${\begin{aligned} 
&a_{kl}=\sum_{l'} a_{k,l'}\<A_{l'},A_l\>= \<[X,A_k], A_l\>=-\<([X,A_l], A_k\>=-a_{lk},\\
&[\ad(X),\sum_{k=1}^p\ad^2(A_k)]
=\sum_k[ \ad(X), \ad(A_k)]\ad(A_k)+\sum_k\ad(A_k)[\ad(X), \ad(A_k)]\\
&=\sum_k\ad([X,A_k])\ad(A_k)+\sum_k\ad(A_k)\ad ([X,A_k])=0.
\end{aligned}}$$
 If $Y_0\in \m$, $\ad(A_k)(Y_0)\in \m$ and by the skew symmetry, (\ref{ad-skew-symmetric}),
 $$\<\sum_{k=1}^p( \ad(A_k))^2(Y_0), Y_0\>=-\sum_{k=1}^p \<\ad(A_k)(Y_0), \ad(A_k)(Y_0)\>.$$
Thus $\sum_{k=1}^p( \ad(A_k))^2(Y_0)=0$, where $p=\dim(\h)$,  implies that $\ad(A_k)(Y_0)=~0$ for all $k$ which in turn implies  that $Y_0\in \m_0$.
 Conversely if $Y_0\in \m_0$ it is clear that $\sum_{k=1}^p\ad^2(A_k)=~0$.

 (2)  Firstly,  $\B_{\ad_\h, \m_l}$ is non-degenerate. If $\{Y_j\}$ is
a basis of $\m_l$ with respect to the $\ad_H$ invariant inner product and $X\in \h$, then
$$\B_{\ad_\h, \m_l}(X,X)=-\sum_j \left|[X, Y_j] \right|^2,$$
which vanishes only if $[X,Y_j]=0$ for all $j$. 
By  the skew symmetry of $\ad_\h$,
for any $Y\in \m_l$, $0=\<[X, Y_j], Y\>=-\<Y_j, [X, Y]\>$ for all $j$. Since $\ad(X)(\m_l)\subset \m_l$, $\B_{\ad_\h, \m_l}(X,X)=0$ implies that $[X,Y]=0$ for all $Y$ which implies $X$ vanishes from the assumption that $\ad_\h$ is faithful.
It is clear that $\B_{\ad_h, \m_l}$ is $\Ad_H$-invariant. Let $X\in \h$ and $h\in H$, then
\begin{equation*}
{\begin{aligned}
&\B_{\ad_\h, \m_l}(\Ad(h)(X), \Ad(h)(X))
=\sum_j \< \ad(\Ad(h)(X))\ad(\Ad(h)(X))Y_j, Y_j\>\\
&=-\sum_j \left\<[\Ad(h)(X),Y_j], [\Ad(h)(X),Y_j]\right \>\\
&=-\sum_j\left\< \Ad(h) [X,  \Ad(h^{-1})Y_j],  \Ad(h) [X,  \Ad(h^{-1})Y_j]\right\>\\
&=-\sum_j \left\< [X,  \Ad(h^{-1})Y_j], [X,  \Ad(h^{-1})Y_j]\right \>
=\B_{\ad_\h, \m_l}(X,X).
\end{aligned}}
\end{equation*}
Since there is a unique, up to a scalar multiple, $\Ad_H$-invariant inner product on a compact manifold,
$\B_{\ad_\h, \m_l}$ is essentially the inner product on $H$.
There is a positive number $a_l$ such that
$\B_{\ad_\h, \m_l}=-a_l\<,\>$.  It is clear that $a_l>0$. For the orthonormal basis $\{A_k\}$, 
$\B_{\ad_\h, \m_l}(A_i, A_j)=-a_l\delta_{i,j}$. We remark that $\{{1\over a_l} A_k\}$ is a dual basis of $\{A_k\}$ with respect to $\B_{\ad_\h, \m_l}$ and ${1\over a_l}\sum_{k} \ad^2(A_k)$ is the Casimir element. 

By part (1), there is a number $\lambda_l$ such that  ${1\over 2} \sum_{k=1}^{\dim(\h)}\ad^2(A_k)=-\lambda_l \id_{\m_l}$. 
The ration between the symmetric form $\B_{\ad_\h, \m_l}$ and the inner products on $\m_l$ can be determined by any unit length vector in $ \h$.
It follows that $$ {\trace}_{\m_l} \sum_{k=1}^{\dim(\h)} \ad^2(A_k)={\trace}_{\m_l}(-2\lambda_l {\id}_{\m_l})= -2\lambda_l \dim(\m_l).$$
On the other hand, 
\begin{equation*}
{\begin{aligned}
{\trace}_{\m_l} ( \sum_{k=1}^{\dim(\h)}\ad^2(A_k))&= \sum_{k=1}^{\dim(\h)}{\trace}_{\m_l} (\ad^2(A_k))= \sum_{k=1}^{\dim(\h)}\B_{\ad_h,\m_l}(A_k, A_k).
\end{aligned}}
\end{equation*}
Consequently $\lambda_l=-\B_{\ad_h,\m_l}(A_1,A_1)  {\dim(\h) \over 2\dim(\m_l)}$.
We completed  part (2).
\end{proof}
If $\h=\h_0\oplus \h_1$ s.t. $\h_0$ acts trivially on $\m_l$ and $\h_1$ a sub Lie-algebra acts faithfully, we take $\h_1$ in place of $\h$,   $\lambda_l$ can be computed using the formula in (2) with $\{A_k\}$ taken to be an orthonormal basis of $\h_1$.

 If  $\A={1\over 2} \sum_{k=1}^N (X_k)^2+X_0$ where $X_k\in \g$, we denote by $c(\A)={1\over 2}\sum_{k=1}^N\ad^2(X_k)+\ad(X_0)$ to be the linear map on $\g$.
  
\begin{lemma}
 \label{compute-L0}
 For any $Y_0\in \m$, $\L_0(\Ad(\cdot)Y_0)=\Ad(\cdot) (c(\L_0)(Y_0))$.
 If $Y_0$ is an eigenvector of $c(\L_0)$ corresponding to an eigenvalue $-\lambda(Y_0) $ then for any $Y\in\m$,
 $ \<Y, \Ad(h)(Y_0)\>$ is an eigenfunction of  $\L_0$ corresponding to $-\lambda(Y_0) $. The converse also holds.
  \end{lemma}
\begin{proof}
Just note that 
$$L_{A_k}(\Ad(\cdot)(Y_0))(h)={d\over dt}|_{t=0} \Ad(h \exp(tA_k))(Y_0)=\Ad(h)([A_k,Y_0]),$$ which by iteration leads to
 $L_{A_k}L_{A_k}\Ad(\cdot)(Y_0)=\Ad(h)([A_k, [A_k,Y_0]])$ and to the required identity
  $\L_0\left(\Ad(\cdot) Y_0\right)=\Ad(\cdot)c(\L_0)(Y_0)$.  Furthermore, for every $Y\in \m$,
 $$\L_0\left(\<\Ad(\cdot) Y_0, Y\>\right)=\<\Ad(\cdot) (c(\L_0)Y_0), Y\>,$$
 from which follows the statement about the eigenfunctions.
\end{proof}

We take an orthonormal basis $\{Y_j\}$ of $\m_l$ and for $Y_0$ fixed set
 $\alpha_j(Y_0)=\<\Ad(\cdot)(Y_0), Y_j\>$. If $f$ is a real valued function,
 set  $\overline f=\int_H f dh$ and by  $\L_0^{-1} f$ we denote a solution to the Poisson equation $\L_0u=f$.
  Also set $\bar \L =-\sum_{i,j=1}^{m_l} \overline{ \alpha_j(Y_0)\,\L_0^{-1} (\alpha_i (Y_0))}\; L_{Y_i}L_{Y_j}$.
 Below we give an invariant formula for the limiting operator on $G$.

  \begin{theorem}
\label{Homogeneous}  
 Suppose that $\{A_1, \dots, A_p\}$ is a basis of $\h$ and
 $Y_0\in \m_l$ where $l\not =0$.  Then, for any orthonormal basis  $\{Y_1, \dots, Y_d\}$  of $\m_l$,
$$\bar \L f=\sum_{i,j=1}^d a_{i,j}(Y_0) \nabla^L df(Y_i,Y_j), \quad f\in C^2(G; \R),$$
where $ a_{i,j}(Y_0)={1\over \lambda_l}\int_{H} \<Y_i,  \Ad(h)Y_0\>  \<Y_j, \Ad(h)Y_0\>\, dh$, 
 $\lambda_l$ is an eigenvalue of $\L_0$, and $-{1\over 2}\sum_{k=1}^p \ad^2(A_k)= \lambda_l\id_{\m_l} $.
Equivalently,
\begin{equation}\label{invariant-form}
\bar \L f
 ={1\over \lambda_l} \int_H  \nabla^L df  \left( (\Ad(h) (Y_0))^*, (\Ad(h)(Y_0))\right)dh.
\end{equation}
Furthermore, $\bar x_t$ is a Markov process
with Markov generator
$$\bar \L (F\circ \pi)(u) ={1\over \lambda_l} \int_H  (\nabla^c dF)_{\pi(u)} 
\left( d\bar L_{uh} Y_0, d\bar L_{uh} Y_0\right)dh, \quad F\in C^2(M;\R),$$
where $\nabla^c$ be the canonical connection on the Riemannian homogeneous manifold.
   \end{theorem}
 
 \begin{proof}
For the Left invariant connection, $\nabla^L _{Y_i}Y_j^*=0$ for any $i,j$, and so
$$ -\sum_{i,j=1}^m  \overline{ \alpha_i (Y_0)\,(\L_0^{-1}  \alpha_j (Y_0)) } {Y_i^*}{Y_j^*}f=-\sum_{i,j=1}^m \overline{ \alpha_i (Y_0)\,(\L_0^{-1}  \alpha_j (Y_0)) } \nabla^L df(Y_i^*,Y_j^*).$$
We can always take the $\Ad_H$ invariant inner product on $\h$ w.r.t. which
 $\{A_1, \dots, A_p\}$ is an orthonormal basis of $\h$, and so Lemma \ref{lambda} applies.
  By Lemma \ref{lambda},
 ${1\over 2}c(\L_0)=-\lambda_l \id_\m$. For any $Y_0, Y\in \m_l$,  $\beta(Y, Y_0)(\cdot):=\<Y,  \Ad(\cdot) Y_0\>$ is an eigenfunction of   $\L_0$ with eigenvalue $-2\lambda_l$, using Lemma\ref{compute-L0}.
Note also that  $\lambda\not =0$.
  Consequently,
\begin{equation}
\overline{ \alpha_i (Y_0)\,(\L_0^{-1}  \alpha_j (Y_0)) }= -{1\over \lambda_l}\int_{H} \<Y_i,  \Ad(h)Y_0)\>  \<Y_j, \Ad(h)(Y_0)\>\, dh.
\end{equation}
The equivalent formula is obtained from summing over the basis of $\m_l$: 
\begin{equation*}
{\begin{aligned}& \sum_{i,j} \alpha_{i,j}(Y_0) \nabla^L df(Y_i^*(u),Y_j^*(u))\\
 &={1\over \lambda_l} \int_H  (\nabla^L df)_u \left( (\Ad(h) (Y_0))^*(u), (\Ad(h)(Y_0) )^*(u)\right)dh.
 \end{aligned}}
\end{equation*}

Finally we prove the Markov property of the projection of the effective process. For $g\in G$,
\begin{equation*}
{\begin{aligned}
&\nabla^L d(F\circ \pi) \left( (\Ad(h) (Y_0))^*(g), (\Ad(h)(Y_0) )^*(g)\right)\\
&=L_{ \Ad(h) (Y_0)^*}  \left (dF \Big(   d\pi \big( (\Ad(h) Y_0)^*\big)\Big)\right)(g)
=\nabla^c dF \left(d\pi( \Ad(h) (Y_0)^*), d\pi(\Ad(h) (Y_0)^*) \right).
\end{aligned}}
\end{equation*}
 The last step is due to part (4) of Lemma \ref{nabla} in \S\ref{section8}, from which we conclude that $dF \Big(  \nabla^c_{\pi_* \Ad(h) (Y_0)^*}  d\pi \big( (\Ad(h) Y_0)^*\big)\Big)$ vanishes. It is clear that,
\begin{equation*}
{\begin{aligned}
&\int_H (\nabla^c dF)_{\pi(g)} \left(  d\pi( \Ad(h) (Y_0)^*(g)), d\pi(\Ad(h) (Y_0)(g))\right)dh\\
&=\int_H (\nabla^c dF)_{\pi(g)} \left( d\bar L_{gh} d\pi(Y_0),d\bar L_{gh} d\pi(Y_0) \right)dh.
\end{aligned}}
\end{equation*} 
Since $dh$ is right invariant, the above formulation is independent of the choice of $g$ in $\pi^{-1}(x)$.
That the stochastic process $\pi(u_t)$ is a Markov process 
 follows from  Dynkin's criterion which states that if $(u_t)$ is a Markov process
 with semigroup $P_t$ and if $\pi: G\to M$ is a map such that $P_t(f\circ \pi)(y)$ depends only on $\pi(y)$ then 
  $\pi(u_t)$ is a Markov process. See e.g. \cite[E. Dynkin]{Dynkin} and  \cite[M. Rosenblatt]{Rosenblatt}. 
\end{proof}

{\it Remarks.} If we take the $\Ad_H$ invariant product on $\h$ to be the one for which $\{A_k\}$ is an o.n.b.
then $\L_0=\Delta_H=\sum_{k=1}^p (A_k)^2$ which follows from $H$ being compact, c.f. Lemma \ref{unimodular}. 
The operator $\bar \L$ in the above Lemma provides an example of a cohesive operator, as defined in \cite[D. Elworthy, Y. LeJan, X.-M. Li ]{Elworthy-LeJan-Li-book}.

%

\section{Using symmetries}

We fix an $\Ad_H$ invariant inner product on $\g$, and  as usual,  set $\m=\h^\perp$ w.r.t an $\Ad_H$ invariant inner product on $\g$. Let  $\m= \m_0\oplus\m_1\oplus \dots \oplus\m_k$ be an orthogonal decomposition. By $\tilde \m$ we denote an $\Ad_H$ invariant
subspace of $\m$ not containing any non-trivial $\Ad_H$-invariant vectors.
In this section we explore the symmetries of the manifold $M$ to study the functions of the form $\<\Ad(\cdot)Y_0, Y_j\>$ where $\{Y_i\}$ is an orthonormal basis  of~$\tilde \m$.

 \begin{definition}
We say that $\Ad_H$ acts {\it quasi doubly transitively} (on the unit sphere) of $\tilde \m$ if for any orthonormal basis 
$\{Y_i\}$ of $\tilde \m$ and for any pair of numbers $i\not = j$ there is $h^{i,j}\in H$ such that 
 $\Ad(h^{i,j})(Y_i)=Y_i$ and $\Ad(h^{i,j})(Y_j)=-Y_j$. \end{definition}

The family of Riemannian manifold, with a lot of symmetry,
are two-point homogeneous spaces by which mean for any two points $x_1, x_2$ and $y_1, y_2$  with $d(x_1, x_2)=d(y_1, y_2)$ sufficiently small,  there exist
an isometry taking $(x_1, x_2)$ to $(y_1, y_2)$. 
Such spaces were classified by H.-C. Wang (1952, Annals)
to be isometric to a symmetric Riemannian space $\Iso^o(M)/K$ where $K$ is compact.
They have constant sectional curvatures in odd dimensions and the non-compact spaces are all simply connected and homeomorphic to an Euclidean space.  On a locally two point homogeneous space $\Ad_H$
acts transitively on the unit sphere, and acts quasi doubly transitively. If $(e,f)$ or orthogonal 
unit vectors at $T_oM$. Let $x=\exp_o(\delta e)$, $x'=\exp_o(-\delta e)$, and
$y=\exp_o(\delta f)$ where $\delta$ is a number sufficiently small for them to be defined.
There is an isometry $\phi$ taking $x$ to $y$ leaving $o$ fixed. This isometry taking the geodesic $0x$ to the geodesic $oy$. 
It would take $y$ to $x'$ which implies that $\rho(x,y)=\rho(y, x')$.

\begin{lemma}\label{transitive-lemma2}
Let $\{Y_1, \dots, Y_d\}$ be an orthonormal basis of $\tilde\m$,  an $\Ad_H$ invariant
subspace of $\m$ with $\tilde \m\cap \m_0=\{0\}$. Let $Y_0\in \tilde \m$.
\begin{enumerate}
\item If $\Ad_H$ acts transitively on the unit sphere of $\tilde \m$. 
$$\int_{H} \<Y_i,  \Ad(h)Y_0)\> ^2\, dh={|Y_0|^2\over \dim(\tilde\m)}, \quad i=1, \dots, \dim(\tilde \m).$$
\item  If $\Ad_H$ acts quasi doubly transitively on $\tilde\m$
then for $i\not =j$, $$\int_{H} \<Y_i,  \Ad(h)Y_0)\>  \<Y_j, \Ad(h)(Y_0)\>\, dh=0.$$
\item $\dim(\tilde \m)>1$.
\item  If $\dim(\tilde \m)=2$ then $\int_{H} \<Y_i,  \Ad(h)Y_0)\>  \<Y_j, \Ad(h)(Y_0)\>\, dh={|Y_0|^2\over 2}\delta_{i,j}$
for all $i,j$.
\end{enumerate}
\end{lemma}
\begin{proof}
Suppose that $\Ad(h_0)$ take $Y_j$ to $Y_1$. Then   
$${\begin{aligned}
 \int_H \<Y_j, \Ad(h)Y_0\> ^2 dh
  &=\int_H \<Y_j, \Ad( h_0^{-1}) \Ad(h)(Y_0)\> ^2dh=\int_H \<Y_1,  \Ad(h)(Y_0)\> ^2dh,\end{aligned}}$$   
using the $\Ad_H$-invariance of the Haar measure and the inner product. Set $d=\dim(\tilde \m)$.
Since $\tilde\m$ is an invariant space of $\Ad_H$ we sum over the basis vectors to obtain
 $$\sum_{j=1}^{d} \int_H\<Y_j, \Ad(h)Y_0\>^2 dh=\int_H \sum_{j=1}^{d} \<\Ad(h^{-1})Y_j, Y_0\>^2 dh=|Y_0|^2.$$
It follows that  $\int_H\<Y_j, \Ad(h) Y_0\>^2 dh={|Y_0|^2\over \dim(\tilde\m)}$.
For $i\not =j$, there exists $h^{i,j}\in H$  such that $\Ad(h^{i,j})(Y_i)=Y_i$ and $\Ad(h^{i,j})(Y_j)=-Y_j$. 
By the bi-invariance of the Haar measure and the $\Ad_H$ -invariance of the inner product again we obtain,
  $${\begin{aligned}
 &\int_H \<Y_i, \Ad(h)Y_0\> \<Y_j, \Ad(h) Y_0\> dh\\
  &=\int_H \<Y_i, \Ad( h^{i,j}) \Ad(h)(Y_0)\> \<Y_j, \Ad(h^{i,j})\Ad(h) Y_0\> dh\\
  &=-\int_H \<Y_i,  \Ad(h)(Y_0)\> \<Y_j, \Ad(h) Y_0\> dh.\end{aligned}}$$   
and part (2) follows.

We observe that no  $\Ad_H$-invariant subspace of $\m$, containing no trivial $\Ad_H$-invariant vectors, can be one dimensional, for otherwise,
 every orthogonal transformation $\Ad(h)$ takes a $Y_0\in \tilde\m$ to itself or to $-Y_0$, and one of which takes $Y_0$ to $-Y_0$, violating connectedness. 
 The connected component of the compact matrix group $\Ad(H)$ is its normal subgroup of the same dimension to which we apply the following facts to conclude the lemma.

Suppose that  $d=2$, then $\tilde \m$ is irreducible. If we identify $\tilde\m$ with $\R^2$ then $H$ can be identified with $SO(2)$.
 Let $\{Y_1, Y_2\}$ be an orthonormal basis of $\m$. 
 Then $$\int_H \<Y_1, \Ad(h)Y_0\> \<Y_2, \Ad(h) Y_0\> dh=\int_{SO(2)} \<Y_1, gY_0\> \<Y_2, g Y_0\> dg$$
where $dg$ is the pushed forward measure  $\Ad^*(dh)$. Since $dg$ is bi-invariant it is the standard measure on $SO(2)$, normalised to have volume $1$. Thus the above integral vanishes. This can be computed explicitly. The same proof as in Proposition \ref{totally-symmetric} shows 
 that $\int_H\<Y, \Ad(h) Y_0\>^2 dh={|Y_0|^2\over 2}$.
\end{proof}

\begin{proposition}
\label{totally-symmetric}
Let $\tilde\m$ be an  $\Ad_H$ invariant subspace of $\m$ not containing any non-trivial $\Ad_H$-invariant vectors.
Suppose  $Y_0\in \tilde \m$ and let $$\bar \L f
 ={1\over \lambda(Y_0)} \int_H  \nabla^L df \left( (\Ad(h) (Y_0))^*, (\Ad(h)(Y_0) )^*\right)dh.$$
  Then
  $\bar \L={|Y_0|^2\over \lambda(Y_0)  \dim(\tilde\m)} \Delta_{\tilde \m}$ under one of the following conditions:
\begin{itemize}
\item[(a)]  $\dim(\tilde \m)=2$;
\item[(b)] $\Ad_H$ acts transitively and quasi doubly transitively on 
the unit sphere of $\tilde\m$.
\end{itemize}

 \end{proposition}

\begin{proof}
Expanding $(\Ad(h)(Y_0))^* $ in $\{Y_j^*\}$ we see that
$\bar \L f=\sum_{i,j=1}^d a_{i,j} (Y_0)\nabla^L df(Y_i^*,Y_j^*)$ where
 $a_{i,j}(Y_0)={1\over \lambda(Y_0)}\int_{H} \<Y_i,  \Ad(h)Y_0)\>  \<Y_j, \Ad(h)(Y_0)\>\, dh$.
If $\dim(\tilde \m)=2$ then $\Ad(H)$ can be identified with $SO(2)$. In both cases, by Lemma \ref{transitive-lemma2}, $a_{11}(Y_0)=\dots =a_{dd}(Y_0)$ and the cross terms disappear.
\end{proof}

 Since $\ad(\h)$ consists of skew symmetric matrices, 
 with respect to the invariant Riemannian metric on the homogeneous manifold $G/H$, $H$ acts as a group of isometries. So do the left actions by elements of $H$ on $\pi_*(\tilde\m)$. 
 Since $\pi$ takes the identity to the coset $H$, we may rule out translations and every element of $\Ad(h)$ is rotation
 on $\tilde\m$.  We identify $\tilde\m$ with $\R^{d}$, where $d$ is the dimension of $\tilde \m$. Since $H$ is connected, they are orientation preserving rotations.  Since $H$ is compact, so is $\Ad(H)$ and $\Ad(H)$ can be identified with a compact subgroup of $SO(d)$.  
 
  It would be nice to classify subgroups of $SO(d)$ that acts transitively and quasi doubly transitively on the spheres.
Sub-groups of $SO(d)$ acting transitively on the spheres are completely classified and coincide with the list of possible holonomy groups of simply connected, see  \cite{Berger-55},  non-symmetric irreducible complete Riemannian manifolds \cite[J. Simons]{Simons}. See also \cite[E. Heintze, W. Ziller]{Heintze-Ziller}.
 The large subgroups of $SO(d)$ are reasonably well understood, which is due to a theorem of Montgomery and Samelson \cite{Montgomery-Samelson} and also a theorem of M. Obata \cite{Obata}. 
There are however two exceptions: the non-simple group $SO(4)$ and also $SO(8) $.
 The latter has two interesting subgroups: the 21 dimensional  $spin(7)$ and the exceptional $14$ dimensional 
  compact simple Lie group $G_2$, automorphism of the Octonians.  We end this discussion with the following remark.
   

  \begin{remark}\label{lemma-transitive}
  Let  $d=\dim(\tilde\m)$ where $ \tilde\m$ is $\Ad_H$-invariant.
 Suppose that $\dim(\Ad(H)|_{\tilde \m})\ge \dim(SO(d-1))$; or suppose that $\tilde \m$ has no two dimensional invariant subspace with $d\ge 13$ and suppose that $\dim(\Ad(H)|_{\tilde \m})\ge  \dim( O(d-3))+\dim( O(3) )+1$.
Then   the identity component of $\Ad(H)_{\tilde \m}$ is \begin{enumerate}
  \item  $SO(d)$ if $d\not =4 ,8$; \item $SO(4)$ or  $S^3$ if $d=4$;
\item $SO(8)$ or  $spin(7)$, if $d=8$.
\end{enumerate}
Furthermore, $\Ad_H: H\to \LL(\tilde \m; \tilde \m)$ acts transitively on the unit sphere of $\tilde \m$, and acts quasi doubly transitively on $\tilde\m$ if $d >2$. 
\end{remark}

Our aim is to prove that $\Ad_H$ acts transitively and quasi doubly transitively, using properties of connected closed subgroups of
 the rotation groups, we are not concerned with whether the image of $H$ under the representation, considered as subgroups of $SO(d)$, $d=\dim(\m_l)$, is connected. It is sufficient to prove its identity component has the required property.  The connected component of the compact matrix group $\Ad(H)$ is its normal subgroup of the same dimension to which we apply the following facts to conclude the lemma. For our purpose, the image of $\Ad_H$ acts faithfully on $\m_l$
  in the sense that if $H_0$ is the group of identity transformations, acting on $\m_l$, then $H/H_0$
  has the same orbit as $H_0$ and both act transitively.

 We have already seen that  no  $\Ad_H$-invariant subspace of $\m$ can be one dimensional.
  By a theorem of  D. Montgomery and H. Samelson\cite[Lemma 3, 4]{Montgomery-Samelson},  there is no proper closed subgroup $H'$ of $SO(d)$ of dimension greater than the dimension of $SO(d-1)$. If $H'$ is a connected closed sub-group of $SO(d)$
 of the dimension of $SO(d-1)$, then $H$ is continuously isomorphic to $SO(d-1)$ or to the double cover of $SO(d-1)$.
  If $d\not =4,8$,  then $H$ is conjugate with $Q(d-1)$, the sub-matrix of $SO(d)$ leaving invariant the first axis. Note that 
  the identity component of $\Ad(H)$ has the dimension of $\Ad(H)$. Since $\m$ has no one dimensional invariant subspace we conclude that $\Ad(H)=SO(d)$ for $d=4,8$.

 In dimension $4$, the subgroup $S^3$ is not conjugate with $Q(d-1)$.
 The subgroup $S^3$ acts on itself transiently, freely, and leaves no one dimensional sub-space invariant.
It is also doubly transitive.
  
  If $d=8$, the $21$-dimensional sub-group $spin(7)$ is embedded into $SO(8)$ by the spin representation, it acts transitively on   $S^7$ and its isotropy subgroup at a point is $G_2$. We learnt from Dmitriy Rumynin that
  $G_2\sim SU(4) \subset spin(7)$ and $G_2$ act transitively on unit spheres of $S^7$ and transitive on any two pairs of orthogonal unit vectors. Finally we quote a theorem from \cite[M. Obata]{Obata} : if  $K$ is a Lie group of orthogonal $d\times d$ matrices
where $d\ge 13$ and if $\dim (O(d)) > \dim(K)\ge \dim( O(d-3))+\dim( d(3) )+1$,  then $K$ is reducible in the real vector space. Since $ \m_l$ is irreducible, the group $\Ad(H)$ must be $SO(d)$ where $d=\dim(m_l)$, and  so the analysis on $SO(d)$ can be applied.

\section{Laplacian like operators as effective limits}
\label{section6}

Let $\{Y_j\}$ be an orthonormal basis of $\tilde \m$. 
Denote $\Delta_{\tilde\m}= \sum_{j=1}^mL_{Y_j^*}L_{Y_j^*}$, this is the `round' operator on $\tilde \m$. For the left invariant connection $\nabla^L$, $\Delta_{\tilde\m}= \trace_{\tilde \m} \nabla^L d$ is a `generalised' horizontal Laplacian and is independent of the choice of  the basis. 
In the special case where $G$ is isomorphic to the Cartesian product of a compact group and an additive vector group, there is a bi-invariant metric,  $\nabla$ is the Levi-Civita connection for a bi-invariant metric, then $\nabla_{X}X^*=0$ and $\Delta_{\tilde \m}=\trace_{\tilde\m} \nabla d$.
In the irreducible case, this operator $\Delta_{\m}$ is the {\it horizontal Laplacian} and  we denote the operator by $\Delta^{hor}$.  Its corresponding diffusion is a {\it  horizontal Brownian motion}.
 In the reducible case, we abuse the notation and define a similar concept. 
  Since the limit operator in Theorem \ref{convergence-and-limit} is given by averaging the action of $\Ad_H$, we expect that the size of the isotropy group
 $H$ is correlated with the `homogeneity' of the diffusion operator, which we explore in the remaining of the section.

\begin{definition}
 A sample continuous Markov process is  a (generalised) horizontal  Brownian motion if its Markov
generator is ${1\over 2} \Delta_{\tilde \m}$; it is a (generalised) scaled horizontal  Brownian motion with scale $c$ if
its Markov generator is ${1\over 2} c \Delta_{\tilde\m}$ for some constant $c\not =0$. 
\end{definition}

We use the notation in Theorem \ref{convergence-and-limit}. Let ,  $\alpha(Y_j, Y_0)=\<Y_j, \Ad(h)(Y_0)\>$ and $$\bar \L =\sum_{i,j=1}^{d}a_{i,j}(Y_0)\;L_{Y_i^*}L_{Y_j^*}, \quad a_{i,j}(Y_0)=-\int_H  \alpha(Y_0,Y_i) \,\L_0^{-1} ( \alpha(Y_0, Y_j) ) dh.$$
If the representation space of $\Ad_H$ were complex, then $ \<Y,  \Ad(h)Y_0)\>$ where
$Y, Y_0\in \m_l$ are known as trigonometric functions.

\begin{theorem}\label{Degenerate}
Suppose that $\{A_1, \dots, A_N\}$ generates $\h$, $\L_0=\half \sum_{k=1}^N (A_k)^2$, and $Y_0\in  \m_l$ where $l\not =0$.
 If $Y_0=\sum_{m=1}^d c_m Y_m$, then
$$\bar \L=\sum_{m=1}^d\f {(c_m)^2} {\lambda(Y_m)}\f {1}{ \dim(\m_l)}\Delta_{\m_l}.$$
If furthermore $\{A_1, \dots, A_N\}$ is an o.n.b. of $\h$, then $
\bar \L=\f {|Y_0|^2}{\dim(\m_l)\lambda_l}\Delta_{\m_l}$.
\end{theorem}
\begin{proof} Set $d=\dim(\m_l)$. With respect to the $\Ad_H$-invariant inner product on $\g$, $\ad^2(A_k)$ is a 
self-adjoint linear map on $\m_l$. For an orthonormal basis  $\{Y_1, \dots, Y_d\}$  
of $\m_l$ consisting of eigenvectors of  $-{1\over 2}\sum_{k=1}^N \ad^2(A_k)$
with $\lambda(Y_j)$ the corresponding eigenvalues.
Then $\alpha(Y_j, Y_0)$ is an eigenfunction of $\L_0$ corresponding to the
  eigenvalue $ -\lambda(Y_0)$,  and $-\f{\alpha(Y_j, Y_0)} {\lambda(Y_0)}$ solves
 the Poisson equation with right hand side $\alpha(Y_j, Y_0)$, see Lemma \ref{compute-L0}.
  Evidently $\lambda(Y_j)\not =0$, for otherwise
  $\< \sum_k \ad^2(A_k)(Y_j), Y_j\>=-\sum_k |\ad(A_k)(Y_j)|^2=0$ which means $Y_j$ is in the kernel of $\ad_H$
which is possible only if $l=0$. 

Consequently,
  $${\begin{split}a_{i,j}(Y_0)&=-\sum_{m,m'} {c_mc_{m'}} \int_H\<\Ad(h)(Y_{m'}), Y_i\>\L_0^{-1} \<\Ad(h)(Y_{m}), Y_j\>dh\\
  &=\sum_{m,m'} \f{c_mc_{m'}}{\lambda(Y_m)} \int_H\<\Ad(h)(Y_{m'}), Y_i\> \<\Ad(h)(Y_{m}), Y_j\>dh. \end{split}}$$

By  Peter-Weyl's theorem, which  states in particular that if $V$ is an irreducible unitary representation of a compact Lie group $H$ and $\{Y_i\}$ an o.n.b. of $V$,
then the collection of functions $\{ \<Y_i,  \rho(\cdot)Y_k)\>\} $, where $V$ ranges through all equivalent classes of irreducible unitary representations  $\rho$,  is orthogonal with norm $\sqrt {\dim(V)}$, and these functions determine $L^2(G)$. In particular  
$$\int_{H} \<Y_i,  \Ad((h)Y_k)\>  \<Y_j, \Ad(h)(Y_l)\>\, dh=\f 1 d\delta_{ij}\delta_{kl}.$$  
We are grateful to Dmitriy Rumynin for providing us with a version of this theorem, valid for orthogonal  representations, which is appended at the end of the paper.  See also Iwahori  \cite{Iwahori} on representations of real Lie algebras.

From this we see that $a_{i,j}=0$ for $i\not =j$ 
 and
\begin{equation*}
a_{i,i}(Y_0)
 =\sum_{m=1}^d\f {(c_m)^2} {\dim(\m_l)\lambda(Y_m)}, \qquad \bar \L=\f{ 1} {d}\sum_{m=1}^d\f {(c_m)^2} {\lambda(Y_m)}\sum_{i=1}^m L_{Y_i}L_{Y_i}.
\end{equation*}
Then part (1) follows by seeting
   $${1\over \lambda'(Y_0)}=  \sum_{m=1}^d{(c_m)^2\over \lambda(Y_m)}.$$
Note that  $\sum_{i=1}^m L_{Y_i}L_{Y_i}=\trace \nabla^L d$ is independent of  the choice of the basis vectors $\{Y_j\}$
so is also $\bar \L$, as a Markov generator to (\ref{corollary3.5}), which means $\lambda'(Y_0)$ is independent of the choice of the basis vectors. For the case $\{A_1, \dots, A_N\}$ is an orthonormal basis of $\h$, we apply Lemma \ref{lambda} to conclude.
   \end{proof}
   
   If $Y_0$ belongs to a subspace of $\tilde \m$, say $Y_0\in \m_l\oplus \m_{l'}$, then an analogous claim holds if
  the representations $\Ad_H$ on $\m_l$ and $\m_{l'}$ are not equivalent, especially if $\m_l$ and $\m_{l'}$ have different dimensions.

Let $d=2$. Then $\sum_k\ad^2(A_k)|_{\m_l}=\lambda_l \id|_{\m_l}$, for some number $\lambda_l\not =0$, if and only if $\bar \L= \f {|Y_0|^2}{2 \lambda_l}\Delta_{\m_l}$. Indeed, $H$ is essentially $SO(2)$.
Let $Y_1, Y_2$ be a pair of orthogonal unit length eigenvectors of the linear map ${1\over 2}\sum_k \ad^2(A_k)$, restricted to $\m_l$.
Let $Y_0=c_1Y_1+c_2Y_2$. For $j',j,k',k=1,2$, let
$$b_{j',j}^{k,k'}(Y_0)={c_kc_{k'}\over \lambda (Y_{k'})}\int_{S^1}  \< \Ad(e^{i2\pi\theta})(Y_k), Y_{j'}\>\<\Ad(e^{i2\pi\theta})(Y_{k'}), Y_j\> d\theta .$$
By direct computation, it is easy to see that $b_{j',j}^{k',k} =0$ unless $j' =j$, $k=k'$ or $j'=k, j=k'$, 
In particular, we examine the cross term:

\begin{equation*}
{\begin{aligned}
a_{1,2}(Y_0)&=\sum_{k,l=1}^2 b_{1,2}^{k,l}=b_{1,2}^{2,1}+ b_{1,2}^{1,2}\\
&= {c_1c_2\over \lambda (Y_2)}\int_{S^1}  \< \Ad(e^{i2\pi\theta})(Y_1), Y_1\>\<\Ad(e^{i2\pi\theta})(Y_2), Y_2\> d\theta\\
&+{c_1c_2\over \lambda (Y_1)}\int_{S^1}  \< \Ad(e^{i2\pi\theta})(Y_2), Y_1\>\<\Ad(e^{i2\pi\theta})(Y_1), Y_2\> d\theta.
\end{aligned}}
\end{equation*}

Direct computation shows that \begin{equation*}
{\begin{aligned}
a_{1,2}(Y_0)&={c_1c_2\over \lambda (Y_2)} \int \cos^2(i2\pi \theta)d\theta- {c_1c_2\over \lambda (Y_1)}\int \sin ^2(i2\pi \theta)d\theta\\
&={1\over 2} ({c_1c_2\over \lambda (Y_2)}-{c_1c_2\over \lambda (Y_1)}).
\end{aligned}}
\end{equation*}
Thus $$\bar \L=({c_1c_2\over \lambda (Y_2)}-{c_1c_2\over \lambda (Y_1)})L_{Y_1}L_{Y_2}+\f 1 2\left( \f {c_1^2}{ \lambda(Y_1)}+\f {c_2^2}{ \lambda(Y_2)}\right)\Delta_{\m_l}.$$
Thus $a_{1,2}(Y_0)$ vanishes if and only if $\lambda_l:=\lambda(Y_2)=\lambda(Y_1)$, the latter allows us to conclude that
$a_{1,1}(Y_0)=a_{2,2}(Y_0)=\f 12\f{|Y_0|^2}{\lambda_l}$
and $\f 12\sum_k\ad^2(A_k) =\lambda_l \id$.

We next work with semi-simple Lie groups. A Lie algebra is simple if it is not one dimensional and if $\{0\}$ and $\g$ are its only ideals; it is semi-simple if it is the direct sum of simple algebras. CartanÕs criterion for semi-simplicity states that $\g$ is semi-simple if and only if its killing form is non-degenerate. Another useful criterion is that a lie algebra is semi-simple if and if it has no solvable (i.e. Abelian) ideals. The special unitary group $SU(n)$ is semi-simple if  $n \ge 2$; $SO(n)$ is semi-simple for $n\ge 3$;
$SL(n,\R)$ is a non-compact semi-simple Lie group for $n\ge 2$; $\so(p,q)$ is semi-simple for $p+q\ge 3$.


 \begin{corollary}
 \label{corollary-symmetric}
 Let $H$ be a maximal torus group of a semi-simple group $G$. Let $Y_0\in \m_l$
 and suppose that $A_0=0$ and $\{A_i\}$ generates $\h$. Then
   $\bar \L={|Y_0|^2\over 2\lambda_l}\Delta_{\m_l}$ where $\lambda_l$ is determined by ${1\over 2}\sum_{k=1}^p \ad^2(A_k)=-\lambda_l \id_{\m_l}$. 
   \end{corollary}
 \begin{proof}
If $\h$ is the Cartan sub-algebra of the semi-simple Lie algebra $\g$ the dimension of $\m_l$ is $2$ and 
Proposition \ref{totally-symmetric} applies. 
It is clear that $\dim(\m_l)\ge 2$, for otherwise it consists of invariant vectors. It is well known that $\g$ is semi-simple if and only if $\g_\C=\g\otimes \C$ is semi-simple. Denote the complexification of $\h$ by $\mathfrak t$, it is a Cartan sub-algebra of $\g\otimes \C$.   Let $\alpha\in {\mathfrak t}^*$ be a weight for $\g_\C$. Take
$Y=Y_1+iY_2$ from the root space $\g_\alpha$ corresponding to $\alpha$. Let $X\in \h$,
and write $\alpha=\alpha_1+i\alpha_2$.  Then $$[X, Y]=[X, Y_1]+i[X, Y_2]=(\alpha_1(X)+i\alpha_2(X))  (Y_1+i Y_2),$$
and $Y_1$ and $Y_2$ generate an invariant subspace for $\ad_H$ following from:
$$[X, Y_1]=\alpha_1(X) Y_1-\alpha_2(X) Y_2; \quad [X,Y_2]=\alpha_1(X)Y_1+ \alpha_2(X) Y_2.$$
Restricted on each vector space $\g_\alpha$ only one specific $A_k$ in the sum $\sum_{k=1}^p \ad^2(A_k)$ makes non-trivial contribution to the corresponding linear operation. Thus we can restrict  to the one dimensional torus sub-group of $H$, which acts faithfully on $\g_\alpha$. By the earlier discussion,
 $\bar \L ={|Y_0|^2\over \lambda_l  \dim(\m_l)} \Delta_{ \m_l}={1\over 2 \lambda_l} {|Y_0|^2} \Delta_{ \m_l}$.
\end{proof}

 \begin{example}
\label{Grass}
Let $G_0(k,n)=SO(n)/SO(k)\times SO(n-k)$ be the oriented Grassmannian manifold of $k$ oriented planes in $n$ dimensions.  It is a connected manifold of dimension $k(n-k)$. The Lie group $SO(n)$ act on it transitively. Let $o$ be $k$-planes spanned by the first $k$ vectors of the standard basis in $\R^n$. Then $SO(k)\times O(n-k)$ keeps $\R^k$ fixed and as well as keeps the orientation of the first $k$-frame. Let  $\pi: SO(n)\to G(k,n)$, then $\pi O=\{O_1, \dots, O_k\}$, the first $k$ columns of the matrix $O$.
Let $\sigma(A)=SAS^{-1}$, where $S$ is the diagonal block matrix with $-I_k$ and $I_{n-k}$ as entries, be the symmetry map on $G$. The Lie algebra has the symmetric decomposition  $\so(n)=\h\oplus \m$:
  $$\h=\left\{\left( \begin{matrix} \so(k)&0\\0&\so(n-k)\end{matrix} \right)\right\}, \quad
\m=\left\{Y_{M}:=\left( \begin{matrix} 0&M\\-M^T&0\end{matrix} \right), M\in M_{k,n-k} \right\},$$
The adjoint action on $\m$ is  given by
$$\Ad(h)(Y_M)=
\left( \begin{matrix} 0&RMQ^T\\-(RMQ^T)^T&0\end{matrix} \right),$$
for $h=\left(\begin{matrix} R&0\\0&Q\end{matrix} \right)$,  $R\in SO(k)$ and $Q\in SO(n-k)$. 
  We identify $\m$ with $T_oG_0(k,n)$.  
 The isotropy action  on $T_oG_0(k,n)$ is now identified with the map
 $(\left(\begin{matrix} R&0\\0&Q\end{matrix} \right), M)\to RMQ^T$. 
Let  $M_1=(e_1,0\dots, 0)$. Then $RM_1Q^T=R_1Q_1^T$
where $R_1$ and $Q_1$ are the first columns of $R$ and $Q$ respectively. The orbit of $M_1$ by the Adjoint action generates a basis of $M_{k, n-k}$ and $M$ is isotropy irreducible.  Let $M_2=(0, e_2, 0,\dots, 0)$, $dR$ and $dQ$ the Haar measure on $SO(k)$ and $SO(n-k)$ respectively.  Then if $R=(r_{ij})$ and $Q=(q_{ij})$. Then
$$\int_H \<\Ad(h)M_1, M_1\>\<\Ad(h)(M_1), M_2\>dRdQ=\int_H(r_{11})^2q_{11}q_{12}dQdR=0.$$
Hence $\bar \L$ is proportional to $\Delta^{hor}$.
 It is easy to see that
$\bar \L$ satisfies the one step H\"ormander condition and is hypoelliptic on $G$.
 Given $M, N \in M_{k\times (n-k)}$ whose corresponding elements in $\m$ denoted by $\tilde M, \tilde N$. Then $[\tilde M, \tilde N]= N^TM-MN^T$.  There is a basis of $\h$ in this form. If $\{e_i\}$ is the standard basis of $\R^k$, $E_{ij}=e_ie_j^T$, $\{ E_{ij}-E_{ji}, i<j\}$ is a basis of $\so(\k)$.  

\end{example}

\section{Classification of limits on Riemannian homogeneous manifolds}
\label{section8}
Let $M$ be a smooth manifold with a transitive action by a Lie group $G$. 
 A Riemannian metric on $M$ is {\it $G$-invariant } if $\bar L_a$ for all $a\in G$, are isometries, in which case $M$ is a Riemannian homogeneous space and $G$ is a subgroup of  $\Iso(M)$. We identify $G$ with the group of actions and $M$ with
  $G/H$ where $H=G_o$, the subgroup fixing a point $o$.  By declaring $d\pi$ at the identity an isometry,  an $\Ad(H)$-invariant inner product on $\m$ induces a $G$ invariant inner product on $T_oM$ and vice versa.
This extends to a $G$ invariant Riemannian metric by defining: 
$\<(d\bar L_g)_o\pi_*Y_1, (d\bar L_g)_o \pi_* Y_2\>_{go}=\<\pi_*Y_1, \pi_* Y_2\>_o$. 
Furthermore,
 $G$-invariant metrics on $M$ are in one to one correspondence with 
 $\Ad_H$-invariant metrics on $\m$.   We should mention that,  by a theorem of 
S. B. Myers and N. E. Steenrod \cite{Myers-Steenrod},   the set of all isometries of a Riemannian manifold $M$ is a Lie
group under composition of maps, and  furthermore the isotropy subgroup $\Iso_o(M)$ is compact.  See also S. Kobayashi and K. Nomizu \cite{Kobayashi-NomizuI}.
If a subgroup $G$ of $\Iso(M)$ acts on $M$ transitively, $G/H$ is a {\it Riemannian homogeneous space}, in the sense that $G$ acts effectively on $M$.

  A connected Lie group admits an Ad-invariant metric if and only if $G$ is of compact type, i.e. $G$ is isomorphic to the Cartesian product of a compact group and an additive vector group
  \cite[J. Milnor , Lemma 7.5 ]{Milnor}, in which case we choose to use the bi-invariant metric for simplicity.
   The existence of an $\Ad_H$-invariant metric is less restrictive. If $H$ is compact by averaging we can construct an $\Ad_H$ invariant inner product in each irreducible invariant subspace of $\m$. If $\m=\m_0\oplus \bigoplus _{i=1}^{r} \m_i$ is an irreducible invariant decomposition of $\m$, 
   $\Ad_H$-invariant inner products on $\m$ are precisely of the  form $g_0+\sum_{i=1}^{r} a_i g_i$ where $g_0$ is any inner product on $\m_0$,  $a_i$ are positive numbers and $g_i$ are $\Ad_H$ invariant inner products on $\m_i$.   In particular an irreducible homogeneous space with
   compact $H$ admits a $G$-invariant Riemannian metric, unique up to homotheties. See   \cite[J. A. Wolf]{Wolf} and   \cite[A. L. Besse, Thm. 7.44] {Besse}.

In the remaining of the section we assume that $G$ is given a left invariant Riemannian metric which induces a $G$-invariant Riemannian metric on $M$ as constructed.
We ask the question whether the projection of the limiting stochastic process in Theorem \ref{Homogeneous} is a  Brownian motion like process. Firstly, we note that the projections of `exponentials', $g\exp(tX)$, are not necessarily Riemannian geodesics on $M$. They are not necessarily Riemannian exponential maps on $G$.
Also, given an orthonormal basis $\{Y_i, 1\le i\le  n\}$  of $\g$, $\sum_{i=1}^n L_{Y_i^*}L_{Y_i^*}$ is not necessarily the Laplace-beltrami operator.

A connected Lie group $G$  admits a bi-invariant Riemannian metric if and only if its Lie algebra admits an
$\Ad_G$ invariant inner product, the latter is equivalent to $\ad(X)$ is skew symmetric for every $X\in \g$.
It is also equivalent to that $G$ is the Cartesian product of a compact group and an additive vector space.
For $X, Y,Z\in \g$,  Koszul's formula for the Levi-Civita connection of the left-invariant metric on $G$ gives:
$$\<2\nabla_XY, Z\>=\<[X,Y],Z\>+\<[Z,Y],X\>+\<[Z,X],Y\>.$$
By polarisation,  $\nabla_XX=0$ for all $X$ if and only if 
\begin{equation}
\label{tss}
\<[Z,Y],X\>+\<[Z,X],Y\>=0, \quad \forall X, Y, Z,
\end{equation} 
in which case $\nabla_XY={1\over 2}[X,Y]$. In another word, (\ref{tss}) holds if and only if
$\nabla^L$ and $\nabla$ have the same set of geodesics, they are translates of the one parameter subgroups. 
If the Riemannian metric on $G$ is bi-invariant then translates of the one parameter subgroups are indeed geodesics for the Levi-Civita connection and for the family of connections interpolating left and right invariant connections.    A connection is  torsion skew symmetric 
if its torsion $\T$ satisfies:
$\<\T(u,v),w\>+\<\T(u,w), v\>=0$ for all $u,v,w \in T_xM$ and $x\in M$.
Denote by $T^L(X,Y)$ the torsion for the flat connection $\nabla^L$. Since $T^L(X,Y)=-[X,Y]$,
 (\ref{tss}) is equivalent to $\nabla^L$ being torsion skew symmetric.  By {Milnor}, if a connected Lie group has a left invariant connection whose Ricci curvatures is non-negative then $G$ is unimodular. 
 We expand this in the following lemma, the unimodular case  is essentially Lemma 6.3 in   \cite[J. Milnor]{Milnor}.
  \begin{lemma}
\label{unimodular}
Let $G$ be a connected Lie group with a left invariant metric. Then $\sum_{i=1}^n L_{Y_i^*}L_{Y_i^*}=\Delta$ if and only if $G$ is unimodular. If $\m_l$ is a subspace of $\g$, then $\trace_{\m_l} \ad(X)=0$ for all $X\in \g$ if and only if $\trace_{\m_l} \nabla d=\trace_{\m_l} \nabla^L d$.
\end{lemma}
\begin{proof}
A Lie group $G$ is unimodular if  its left invariant Haar measure is also right invariant. 
By  \cite[ S. Helgson]{Helgason}, a Lie group is unimodular if and only if the absolute value of the determinant of 
$\Ad(g): G\to G$ is $1$ for every $g\in G$. Equivalently
 $\ad(X)$ has zero trace for every $X\in \g$, see  \cite[J. Milnor, Lemma 6.3]{Milnor}.
On the other hand  $\sum_{i=1}^n L_{Y_i^*}L_{Y_i^*}=\Delta$ if and only if  $\sum_{i=1}^n \nabla _{Y_i^*}Y_i^*=0$. 
The latter condition is equivalent to 
$\trace \ad(X)=\sum_i\<[X, Y_i],Y_i\>=0$ for all $X\in \g$.
Similarly if $\{Y_i\}$ is an orthonormal basis of $\m_l$, then
$$\sum_{i=1}^{\dim(\m_l)}  L_{Y_i^*}L_{Y_i^*}=\sum_{i=1}^{\dim(\m_l)}  \nabla d + \sum_{i=1}^{\dim(\m_l)}  \nabla _{Y_i^*}Y_i^*.$$
The second statement follows from the identity:
 $$\< \sum_{i=1}^{\dim(\m_l)}  \nabla _{Y_i^*}Y_i^*, X^*\> =\sum_{i=1}^{\dim(\m_l)} \<[X, Y_i], Y_i\>.$$
\end{proof}

We discuss the relation between the Levi-Civita connection and the canonical connection on the Rieamnnian homogeneous manifold.
A connection is $G$ invariant if  $\{\bar L_a, a \in G\}$ are affine maps for the connection, so they preserve parallel vector fields. 
For $X\in \g$, define the derivation $A_X=L_X - \nabla_X $.
A $G$-invariant connection is determined by a linear map $(A_X)_o\in \LL(T_oM;T_oM)$. Each $A_X$  is in correspondence with
an endomorphism $\Lambda_\m(X)$ on $\R^n$, satisfying the condition that for all $X\in \m$ and $h\in H$,
$\Lambda_\m(\ad(h)X)=\Ad(\lambda(h))\Lambda_\m(X)$
where $\lambda(h)$ is the isotropy representation of $h$, see S. Kobayashi and K. Nomizu \cite{Kobayashi-NomizuI}.
Then
$$\tilde \nabla_XY:=L_XY-u_0\Lambda_\m(X)u_0^{-1}Y$$ defines a connection.  We identified $T_oM$ with $\R^n$ by a frame $u_0$. 
 If $\Lambda_\m=0$ this defines the {\it canonical connection} $\nabla^c$ whose parallel translation along a curve
 is left translation. 
 
 Denote by
${D^c\over dt}$ the corresponding covariant differentiation.
Parallel translations along $\alpha(t)$ are given by $d\bar L_{\gamma(t)}$. 
 This is due to the fact that left translations and $\pi$ commute.
If $X\in \m$, let $\gamma_t=\gamma_0\exp(tX)$ and $\alpha(t)=\gamma_0\exp(tX)o$. Then $\gamma$ is the horizontal lift of $\alpha$ 
 and $\alpha(t)$ is a $\nabla^c$ geodesic, 
 $${D^c\over dt}  \dot \alpha={D^c\over dt} \pi_* TL_ {\gamma(t)}(X)
 ={D^c\over dt} 
 T\bar L_ {\gamma(t)} \pi_*(X)=0.$$

For $x\in \g$, denote by $X_m$  and $X_\h$ the component of $X$ in $\m$ and in $\h$ respectively. 
\begin{definition}\label{def-naturally-reductive}
A reductive Riemannian homogeneous space is {\it naturally reductive}, if
\begin{equation}
\label{naturally-reductive}
\<[X,Y]_\m,Z\>=-\<Y, [X,Z]_\m\>, \quad \hbox{ for all $X, Y, Z\in \m$}.
\end{equation}  
\end{definition}
If $G$ is of compact type, then $G/H$ is reductive with reductive structure $m=\h^\perp$ 
and is naturally reductive, with respect to the bi-invariant metric.

It is clear that (\ref{tss}) implies (\ref{naturally-reductive}) and $M$ is naturally reductive if
$\nabla$ and $\nabla^L$ have the same set of geodesics. In particular, if
$G$ admits an $\ad_G$ invariant inner product and $\m=\h^\perp$, 
then $G/H$ is naturally reductive with respect to the induced Riemannian metric.
A special  reductive homogeneous space is a symmetric space with symmetry $\sigma$
and the canonical decomposition $\g=\h\oplus \m$. For it,
$[\h, \m]\subset \m$ and $[\m,\m]\subset \h$. If the symmetric decomposition is orthogonal,
it is naturally reductive.

We collect in the lemma below useful information for the computations of the Markov generators whose proof is included for the convenience of the reader. 
 \begin{lemma}
 \label{nabla}
\begin{enumerate}
\item $\nabla^c$ is torsion skew symmetric precisely if $M$ is naturally reductive.
\item The projections of  the translates of one parameter family of subgroups of $G$  are geodesics for the Levi-Civita connection
if and only if $M$ is naturally reductive.
\item  Let $U: \m\times \m \to \m$ be defined by 
\begin{equation}\label{U}
2\<U(X,Y),Z\>=\<X,[Z,Y]_\m\>+\<[Z,X]_\m, Y\>.
\end{equation}
Then $\trace_{\m_l}\nabla^c df=\trace_{\m_l}\nabla df$ if and only if $\trace_{\m_l} U=0$. In other words,
for any $Z\in \m$, $\trace_{\m_l} \ad(Z)=0$. In particular $\trace_{\m_l}\nabla^c df=\trace_{\m_l}\nabla df$,
if $M$ is naturally reductive.

\item  $\nabla^c d\pi =0$.

\end{enumerate}\end{lemma}
\begin{proof}
(1) Let $Y^*$ denote the action field generated by $Y\in \m$, identified with $\pi_*(Y)\in T_oM$, and
$Y^*(uo)={d\over dt} |_{t=0}\bar L_{\exp(tY)}uo={d\over dt} \big |_{t=0}\pi (\exp(tY)u)$.
If $p=uo$, $\nabla^c_XY^*(o)=[X^*, Y^*] (o)=-[X, Y]_\m$.
The torsion tensor for $\nabla^c$ is  left invariant, its value at $o$ is
$$\T_o^c(X,Y)=(\nabla^c_XY^*)_o-(\nabla^c_YX^*)_o-[X^*,Y^*]_o=-[X,Y]_\m,$$
 We see that $\T^c$ is skew symmetric if and only if (\ref{naturally-reductive}) holds.

 (2)  A $G$-invariant connection on $M$ has the same set of geodesics as $\nabla^c$ if and only if $\Lambda_\m(X)(X)=0$, c.f. Kobayashi and Nomizu \cite{Kobayashi-NomizuI}. 
It is well known that the function
$$\Lambda_\m(X)(Y)={1\over 2}[X,Y]_\m+U(X,Y).$$
defines the Levi-Civita connection. In fact it is clear that ${1\over 2}[X, Y]-U(X,Y)$
has vanishing torsion and $\<\Lambda(X, Y), Z\>+\<Y, \Lambda(X,Z)\>=0$, which together with the fact $\nabla^c$ is metric implies that it is a Riemannian connection.
The Levi-Civita connection $\nabla$ and $\nabla^c$ have the same set of geodesics if and only if $U(X,X)=0$.  By polarisation, this is equivalent to $M$ being naturally reductive. 
 
(3) For any $X,Y\in \m$,
$$(\nabla_XY^*)_o=\nabla_X^cY^*+{1\over 2} [X,Y]_\m+U(X,Y).$$
If $f:M\to \R$ is a smooth function, 
\begin{equation*}
{\begin{aligned}
\<\nabla_{X}\nabla f, Y\>
&=L_Y \left (df(Y^*)\right)-\< \nabla f,  \nabla_XY^*\>\\
&=\<\nabla_X^c\nabla f, Y\>+ \< \nabla f,  \nabla^c_X Y^*-\nabla_XY^*\>.
\end{aligned}}
\end{equation*}
Summing over the basis of $\m_l$ we see that
 $${\trace}_{\m_l} \nabla df-{\trace}_{\m_l} \nabla^c df=\< \nabla f, \sum_i U(Y_i, Y_i)\>,$$
which vanishes for all smooth $f$ if and only if $ \sum_i U(Y_i, Y_i)$ vanishes. If $G$ is of compact type,
it has an $\Ad_G$-invariant metric for which (\ref{tss}) holds.
This completes the proof.

 (4) We differentiate the map
$d\pi:  G \mapsto \LL(TG; TM)$ with respect to the connection $\nabla^c$.
Let $\gamma(t)$ be a curve in $G$ with $\gamma(0)=u$ and $\dot \gamma(0)=w$.  
Then for $X\in \g$,
$$\left(\nabla^c_w d\pi  \right) (X)=\parals_t^c(\pi\circ \gamma){d\over dt} (\parals_t^c(\pi\circ \gamma) )^{-1} 
d\pi \left(  \parals_t^L (\gamma)X\right), $$
where $\parals_t^c(\pi\circ \gamma)$ and $ \parals_t^L (\gamma)$ denote respectively parallel translations 
along $\pi\circ \gamma$ and $ \gamma$ with respect to to $\nabla^c$ and $\nabla^L$. Since 
$\parals_t^c (\pi\circ \gamma)=d\bar L_{\pi\circ \gamma(t)}$ and $\parals_t^L (\gamma)=dL_{\gamma(t)}$
the covariant derivative vanishes. Indeed $\pi$ and left translation commutes,
$(d\pi)_u d L_u=(d\bar L_u)_o (d\pi)_1$,
$$\nabla^c_w d\pi =\parals_t^c(\pi\circ \gamma){d\over dt} \big|_{t=0}
 (d\pi)_1  \left(dL_{\gamma(-t)}\parals_t^L (\gamma)\right)=0. $$
\end{proof}

In the propositions below we keep the notation  in Theorem \ref{convergence-and-limit}, Theorem \ref{Homogeneous}
and Proposition \ref{totally-symmetric}. We identify $\m_l$ with
its projection to $T_oM$.
Let $\bar x_t:=\lim_{\epsilon\to 0} x_{t\over \epsilon}^\epsilon$.

 \begin{proposition}
 \label{projection}
Suppose that $\bar \L={|Y_0|^2\over \lambda(Y_0)  \dim(\m_l)} \Delta_{\m_l}$.
If $\trace_{\m_l} U=0$, equivalently $\trace_{\m_l}\ad(Z)=0$ for all $Z\in \m$,
  then $(\bar x_t)$ is a Markov process with generator
 $${ |Y_0|^2\over \lambda(Y_0) \dim(\m_l)}{\trace}_{\m_l}\nabla d.$$
 If $M$ is furthermore isotropy irreducible, $(\bar x_t)$ is a scaled Brownian motion.
\end{proposition}

\begin{proof}
Let $\{\tilde Y_i, 1\le i\le m_l\}$ be an orthonormal basis of $\m_l$ and let  ${1\over 2} a^2={ |Y_0|^2\over \dim(\m_l)\lambda_l}$.
Let $Y_i=a\tilde Y_i$. Let $(u_t)$ be a Markov process with generator
${1\over 2} \sum_{k=1}^{\dim(m_l)} L_{Y_k^*}L_{Y_k^*}$, and
 represented by a solution of the left invariant SDE:
$du_t=\sum_{k=1}^{m_l}Y_k^*(u_t)\circ dB_t^k$ where $\{B_t^k\}$ are independent one dimensional Brownian motions. 
 Let $y_t=\pi(u_t)$.
Denote by $\phi_t^k$ the integral flow of $Y_k$,  $\phi_t^i(g)=g\exp(tY_i)$. If $f\in C_K^2(M; \R)$, 
\begin{equation*}
{\begin{aligned}
f\circ \pi(u_t)=&f\circ \pi(u_0)+\sum_{i=1}^{m_l} \int_0^t {d\over dt} f\circ \pi \left( u_r \exp(tY_i) \right)\big|_{t=0}  dB_r^i
\\&+{1\over 2} \sum_{i=1}^{m_l}  \int_0^t   {d^2\over dt^2} f\circ \pi \left( u_r \exp(tY_i) \right) )\big|_{t=0}  dr.
\end{aligned}}
\end{equation*}
We  compute the last term beginning with the first order derivative,
$${d\over dt}f\circ \left( u_r \exp(tX_i) \right)|_{t=0}  =df (\pi_* (L_{u_r \exp(tY_i)} Y_i ))|_{t=0}=df(  (\bar L_{u_r})_*X_i).$$

Let ${D\over dt}$ denote covariant differentiation along the curve $(x_t)$ with respect to the Levi-Civita connection.
\begin{equation*}
{\begin{aligned}
&\sum_{i=1}^{m_l}{d^2\over dt^2} f\circ \pi\left( u_r \exp(tY_i) \right)|_{t=0}\\
&=\sum_{i=1}^{m_l} {d\over dt}df \left( (\bar L_{u_r \exp(tY_i)})_*( \pi_*(Y_i))\right)|_{t=0}\\
&=\sum_{i=1}^{m_l}\nabla df\left((\bar L_{u_r})_*(\pi_*(Y_i)), (\bar L_{u_r})_* (\pi_*(Y_i)\right)
+\sum_{i=1}^{m_l} df\left( {D \over dt}\left(\bar L_{u_r \exp(tY_i)})_* (\pi_*(Y_i)\right)|_{t=0}\right).\end{aligned}}
\end{equation*}
  
 Since left translations are isometries, for  each $u\in G$,
$  \{(\bar L_{u})_*(\pi_*(Y_i))\}$ is an orthonormal basis of $\pi_*(u\m_l)\subset T_{\pi(u)}M$. Thus
$$\sum_{i=1}^{m_l}{d^2\over dt^2} f\circ \pi\left( u_r \exp(tY_i) \right)|_{t=0}
= a^2{\trace}_{\m_l} \nabla^c df+\sum_{i=1}^{\m_l}df\left( {D^c \over dt}(\bar L_{u_r \exp(tY_i)})_* Y_i^*|_{t=0}\right).$$
If $\trace_{\m_l}U=0$, ${\trace}_{m_l} \nabla^c df={\trace}_{m_l} \nabla df$ by Lemma \ref{nabla}. We have seen 
that the term involving ${D^c\over dt}$ vanishes. This gives:
$$\sum_{i=1}^{m_l}{d^2\over dt^2} f\circ \pi\left( u_r \exp(tY_i) \right)|_{t=0}
= a^2{\trace}_{\m_l} \nabla df.$$ 
Put everything together we see that
$$\E f\circ \pi(u_t)=f\circ \pi(u_0)
+{1\over 2}a^2 \int_0^t  \E (\trace)_{\m_l} \nabla df(\pi(u_r)) dr.$$

  For $x=\pi(y)$ we define $Q_t f(x)=P_t(f\circ \pi)(y)$, so $P_t(f\circ \pi)=(Q_t f)\circ \pi$.
We apply Dynkin's criterion for functions of a Markov process to see that $(\bar x_t)$ is Markovian. 
 The infinitesimal generator associated to $Q_t$ is
 ${1\over 2}a^2  (\trace)_{\m_l} \nabla df$.
If $M$ is an irreducible Riemannian symmetric space, $\m$ is irreducible, 
$d\pi(\m)=T_oM$ and the Markov process $(\bar x_t)$ is a scaled Brownian motion. This
concluds the Proposition.
 \end{proof}
 
If $M$ is naturally reductive, the proof is even simpler. In this case, $\sigma_i(t)=\pi\left( u\exp(tY_i)\right)$ is a geodesic 
with initial velocity $Y_i$ and 
$\dot\sigma(t)={d \over dt}\pi_*\left( u\exp(tY_i)\right)$.  
Consequently,  the following also vanishes:
$${D\over dt}(\bar L_{u_r \exp(tY_i)})_* (Y_i^*)|_{t=0}=0.$$
%

\section{Examples}
\label{section-final-examples}
Corollary \ref{corollary-symmetric} applies to the example in \S \ref{Bergers}, where  $G=SU(2)$, $H=U(1)$ and
the $\Ad_H$-invariant space $\m=\<X_2, X_3\>$ is irreducible. For any $Y\in \m$,  $\ad^2(X_1)(Y)=-4Y$. Note that $4$ is the second non-zero eigenvalue of $\Delta_{S^1}$,
also the killing form of $SU(2)$ is $K(X,Y)=4 \trace (XY)$ and $\B_{\ad_\h, \m}(X_1, X_1)=K(X_1, X_1)=-8$.

\begin{example}
\label{Hopf-theorem}
Let $(b_t)$ be a one dimensional Brownian motion, $g_0\in SU(2)$, $Y_0\in \<X_2, X_3\>$ non-zero. Let 
$(g_t^\epsilon, h_t^\epsilon)$ be the solution to the following SDE on $SU(2)\times U(1)$,
\begin{equation}
dg_t^\epsilon=g_t^\epsilon  Y_0 dt+ {1\over \sqrt\epsilon} g_t^\epsilon X_1db_t,
\end{equation}
with  $g_0^\epsilon=g_0$. Let $\pi(z,w)=({1\over 2}(|w|^2-|z|^2), z\bar w)$ and $x_t^\epsilon=\pi(g_t^\epsilon)$.
Let $(\tilde x_t^\epsilon)$ be the horizontal lift of $(x_t^\epsilon)$.
Then  $(\tilde  x_{t\over \epsilon}^\epsilon)$ converges weakly to the hypoelliptic diffusion with generator  $\bar \L ={|Y_0|^2\over 4} \Delta^{hor}$.
Furthermore,
$x_{t\over \epsilon}^\epsilon$ converges in law to the Brownian motion on $S^2({1\over 2})$ scaled by ${1\over 2}|Y_0|^2$.
\end{example}
 
The first part of the theorem follows from Theorem \ref{convergence-and-limit} and Corollary \ref{corollary-symmetric}. 
The  scaling ${1\over 2}$ indicates the  extra time needed for producing the extra direction $[X_1, Y_0]$.
For the second part we use the fact that $G$ is compact, so $\< [X,Y], [Z]\>=-\< Y, [X,Z]\>$ for any  $X, Y, Z\in \m$,
and Proposition \ref{projection} applies. Incidentally,
a similar argument can be made for an analogous equation on $G=SU(n)$ and $H$ the torus group of $G$.

\begin{example}
\label{SO4}
Let $G=SO(4)$, $H=SO(3)$, $E_{i,j}$ the elementary $4\times 4$ matrices, and $A_{i,j}={1\over \sqrt 2} (E_{ij}-E_{ji})$.
Then $\h=\{ A_{1,2}, A_{1,3}, A_{2,3}\}$, its orthogonal complement $\m$ is irreducible w.r.t. 
$\Ad_H$. 
  For $k=1,2,$ and $3$, let $Y_k=A_{k4}$. Let us consider the equations,
$$dg_t^\epsilon={1\over \sqrt\epsilon}  A_{1,2}(g_t^\epsilon)\circ db_t^1+
{1\over \sqrt\epsilon}  A_{1,3}(g_t^\epsilon)\circ db_t^2+Y_k(g_t^\epsilon)dt.$$
 Observe that $\{ A_{1,2}, A_{1,3}\}$ is a set of generators,  in fact $[A_{1,2}, A_{1,3}]=-{1\over \sqrt 2}A_{2,3}$,
and $\L_0={1\over 2} ({A_{1,2}})^2+{1\over 2} (A_{1,3})^2$ satisfies 
strong H\"ormander's conditions.
It is easy to check that,
$${1\over 2}\ad^2(A_{1,2})(Y_k)+{1\over 2}\\ad^2(A_{1,3})(Y_k)=-{1\over 4} ( 2\delta_{1,k} +\delta_{2,k}+\delta_{3,k}) Y_k.$$
For any $Y\in \m$, $\<Y, \Ad(h)Y_k\>$ is an eigenfunction of $\L_0$ corresponding to the eigenvalue $-\lambda(Y_k)$, where
$$\lambda(Y_1)={1\over 2}, \; \lambda(Y_2)={1\over 4} , \; \lambda(Y_3)={1\over 4}.$$
Define $\pi: g\in SO(4)\to S^4$ to be the map projecting $g$ to its last column. 
Set $x_t^\epsilon=\pi(g_t^\epsilon)$ whose horizontal lift through $u_0$ will be denoted by $\tilde x_t^\epsilon$. 
Then 
$$a_{i,j}(Y_k):={1\over \lambda(Y_k)} \int_{H} \<Y_i,  \Ad(h)Y_k)\> \L_0^{-1} (\<Y_j, \Ad(h)(Y_k)\> )dh.$$
For $i\not =j$, $a_{i,j}(Y_k)=0$ and $a_{i,i}(Y_k):={1\over 3 \lambda(Y_k)}$.
 In Theorem \ref{convergence-and-limit} take $Y_0=Y_k$ to~see 
 $$\bar \L={1\over \lambda(Y_k) } \int_{SO(3)} \nabla^L df ( \Ad(h) (Y_k))^*, (\Ad(h)(Y_k))^*) dh
={1\over 3 \lambda(Y_k)} \sum_{i=1}^3\nabla^L df (Y_i, Y_i).$$

On the other hand, the effective diffusion for the equations
$$dg_t^\epsilon={1\over \sqrt\epsilon}  A_{1,2}(g_t^\epsilon)\circ db_t^1+
{1\over \sqrt\epsilon}  A_{1,3}(g_t^\epsilon)\circ db_t^2+{1\over \sqrt\epsilon} A_{2,3}(g_t^\epsilon)\circ db_t^3 +Y_k(g_t^\epsilon)dt$$
is the same for all $Y_k$. It is easy to see that $\bar \L={2\over 3}  \sum_{i=1}^3\nabla^L df (Y_i, Y_i)$.
\end{example}

 Further symmetries can, of course, be explored.

\begin{example}
\label{SO4-2}
Let $Y_0=\sum_{k=1}^3 c_k Y_k$ be a mixed vector. Then,
 $$\overline{\alpha_i \beta_j}=\sum _{k, l=1}^3 c_kc_l \overline{\alpha_i (Y_k) \beta_j (Y_l)}
= \sum _{k, l=1}^3 c_kc_l {1 \over \lambda_l}\overline{ \alpha_i (Y_k) \alpha_j (Y_l)}.$$
 If $c_1=0$ and $c_2=c_3$, $\overline{\alpha_i \beta_j}
=4  (c_2)^2\sum _{k, l=2}^3 \overline{ \alpha_i (Y_k) \alpha_j (Y_l)}$. 
By symmetry, $\overline{ \alpha_i (Y_k) \alpha_j (Y_l)}$ vanishes for $i\not =j$.
\end{example}

\begin{example}\label{sphere}
Let  $n\ge 2$,  $G=SO(n+1)$,  
$H=\left\{\left( \begin{matrix} R &0\\ 0&1\end{matrix}  \right), R \in SO(n)\right\}$, and
 $S^n=SO(n+1)/SO(n)$. Then $H$ fixes the point $o=(0, \dots, 0, 1)^T$.
The homogeneous space $S^n$ has the reductive decomposition:
$$\h=\left\{\left( \begin{matrix}  S &0\\ 0&0\end{matrix}  \right), S \in \so(n)\right\},\;\;
\m=\left\{Y_C=\left( \begin{matrix}  0 &C \\ -C^T& 0\end{matrix}  \right), C\in \R^{n}\right\}.
 $$
 Let $\sigma(A)=S_0AS_0^{-1}$ and
 $S_0=\left( \begin{matrix}  I &0\\ 0&-1\end{matrix}  \right)$.  Then $\g=\h\oplus \m$ is
 the symmetric space decomposition for $\sigma$. We identify
 $Y_C$ with the vector $C$, and compute:
 $$\Ad\left (\left( \begin{matrix} R &0\\ 0&1\end{matrix}  \right) Y_C \right)
=\left( \begin{matrix}  0 &RC \\ -(RC)^T& 0\end{matrix}  \right).$$
This action is transitive on the unit tangent sphere and is irreducible.  
There is a matrix $R\in SO(n)$ that sends $C$ to $-C$. Then $\bar Y_0=0$ for every  $Y_0\in\m$
and the conditions of Proposition \ref{totally-symmetric} are satisfied.

Let  $A_{i,j}={1\over \sqrt 2}(E_{ij}-E_{ji})$, where $i<j$,
 and $E_{ij}$ is the elementary matrix with the nonzero entry at the $(i,j)$-th position.  
 Let $Y_0\in\m$ be a non-trivial vector. Consider the equation,
$$dg_t^\epsilon={1\over \sqrt\epsilon} \sum_{1\le i<j\le n} A_{i,j}^*(g_t^\epsilon)\circ db_t^{i,j} +Y_0^*(g_t^\epsilon)dt.$$
 We define  $\lambda={(n-1)\over 4}$. By a symmetry argument it is easy to see that,
 $$ \sum_{1\le i<j\le n} \ad^2(A_{ij})=-{1\over 2}(n-1)I.$$ 
 For $1\le i, j, k\le n$, $[A_{i,j}, A_{k, n+1}]=-{1\over \sqrt 2} \delta_{k,i}A_{j,n+1}+{1\over \sqrt 2} \delta_{j,k} A_{i,n+1}$, which follows from
 $E_{ij}E_{kl}=\delta_{jk}E_{il}$.
 Hence for $i\not =j$, \begin{equation}\label{bracket-formula}
 2\ad^2(A_{i,j})(A_{k,n+1})=-( \delta_{i,k} +\delta_{j,k}) A_{k,n+1}.
 \end{equation}

Let $x_{t\over \epsilon}^\epsilon=g_{t\over \epsilon}^\epsilon o$ and $(\tilde x_{t\over \epsilon}^\epsilon)$ the horizontal lift of $(x_{t\over \epsilon}^\epsilon)$ through $g_0$.
By Theorem \ref{convergence-and-limit}, converges to a Markov process whose limiting Markov generator  is, by symmetry,
$$\bar\L={|Y_0|^2 \over \lambda  \dim(\m)} \Delta^{hor}={4|Y_0|^2\over n(n-1)}\Delta^{hor}.$$
 The upper bound for the rate of convergence from
Theorem \ref{thm-limit-2} holds.  By Theorem \ref{convergence-and-limit} and  \ref{projection}
the stochastic processes $(x_{t\over \epsilon}^\epsilon)$
converge to the Brownian motion on $S^n$ with scale $ {8|Y_0|^2\over n(n-1)}$.
By (\ref{bracket-formula}), for any $1\le i,j, k\le n$,
$A_{k,n+1}$ are eigenvectors for $\ad^2(A_{i,j})$. Furthermore, for $1\le i< j\le n$,
$$[A_{i,n+1}, A_{j,n+1}]=-{1\over \sqrt 2} A_{i,j},$$
and $\Delta^{hor}$ satisfies the one step H\"ormander condition. 

 \end{example}

\begin{example}
We keep the notation in the example above. Let $Y_0$ be a unit vector from $\m$ and  let $(g_t^\epsilon)$ be the solutions of the following equation:
\begin{equation}
dg_t^\epsilon={1\over \sqrt\epsilon} \sum_{1\le i<j\le n} A_{i,j}^*(g_t^\epsilon)\circ db_t^{i,j} +{n(n-1)\over 8} Y_0^*(g_t^\epsilon)dt.
\end{equation}
Then as $\epsilon\to 0$, $\pi(g_{t\over \epsilon}^\epsilon)$ converges to a Brownian motion on $S^n$.
\end{example}

We finally provide an example in which the group  $G$ is not compact.
\begin{example}
\label{hyperbolic}
 Let $n \ge 3$. Let $F(x,y)=-x_0y_0+\sum_{k=1}^n x_k y_k$ be a bilinear form on $\R^{n+1}$.
  Let $O(1,n)$ be the set of $(n+1)\times (n+1)$ matrices preserving the indefinite form 
  $$O(1,n)=\left\{ A\in GL(n+1):   A^TSA=S, S=\left( \begin{matrix} -1 &0  \\ 0&I_n \end{matrix}  \right) \right\}.$$ 
  Let $G$ denote the identity component of $O(1,n)$, consisting of  $A\in O(1,n)$ with $\det(A)=1$ and $a_{00}\ge 1$. 
  This is a $n(n+1)/2$ dimensional manifold,  $X\in \g$ if  and only $X^TS+SX =0$. So
    $$\g=\left\{\left( \begin{matrix}0 &\xi^T  \\ \xi&S\end{matrix}  \right), \xi\in \R^n, S\in \so(n) \right\}.$$ 
      The map  $\sigma(A)=SAS^{-1}$ defines an evolution on $G$
  and  $$M=\{x=(x^0, x^1, \dots, x^n): F(x,x)=-1, x^0\ge 1\}$$ is a symmetric space.
Its isometry group at $e_0=(1, 0, \dots, 0)^T$ is $H=\left\{ \left( \begin{matrix} 1 &0  \\ 0&B \end{matrix}  \right)\right\}$ where $B\in SO(n)$. The $-1$ eigenspace of the Lie algebra involution is:
$$\m=\left\{\left( \begin{matrix} 0 &\xi^T  \\ \xi&0 \end{matrix}  \right) \right\}.$$
Let us take $Y_0\in \m$ and consider the equation,
$$dg_t^\epsilon={1\over \sqrt\epsilon} \sum_{1\le i<j\le n} A_{i,j}^*(g_t^\epsilon)\circ db_t^{i,j} +Y_0^*(g_t^\epsilon)dt.$$
It is clear that Theorem \ref{convergence-and-limit} and  \ref{Homogeneous} apply.  
The isotropy representation of $H$ is irreducible. The action of $\Ad_H$ on $\m$ is essentially the action of $SO(n)$ on $\R^n$:
$$\Ad( \left( \begin{matrix} 1 &0  \\ 0&B \end{matrix}  \right))\left( \begin{matrix} 0 &\xi^T  \\ \xi&0 \end{matrix}  \right)
=\left( \begin{matrix} 0 & (B\xi)^T  \\B \xi&0 \end{matrix}  \right).$$
The conclusion of Theorems hold. The symmetric space $M$ is naturally reductive and so Proposition \ref{projection} applies.
Note that
$$\left[\left( \begin{matrix} 0 &\xi^T  \\ \xi&0 \end{matrix}  \right), \left( \begin{matrix} 0 &\eta^T  \\ \eta&0 \end{matrix}  \right)\right]
=\left( \begin{matrix} 0 &0 \\ 0&  \xi\eta^T-\eta  \xi^T \end{matrix}  \right).$$
Since $e_ie_j^T=E_{ij}$, $[\m,\m]$ generates a basis of $\h$. 
The effective process arising from the family of stochastic differential equations (\ref{1}) is a scaled horizontal Brownian motion. Observe that the latter satisfies the one step H\"ormander condition. The limit of $\pi(g_{t\over \epsilon}^\epsilon)$ is a scaled hyperbolic Brownian motion. The rate of convergence stated in Theorem \ref{thm-limit-2} is valid here, the scale is ${8\over n(n-1)}$.

\end{example}

\section{Further discussions and  open questions}
\label{section11}
The study of scaling limits should generalise to principal bundles  and  to foliated manifolds. 
Let  $\pi:P\to M$ be a principal bundle with group action $H$ whose Lie algebra $\h$ is given a suitable inner product. For $A\in \h$ denote by $A^*$ the corresponding fundamental vertical vector field: $A^*(u)={d\over dt}|_{t=0} u\exp(tA)$. 
Let us fix an orthonormal basis  $\{A_1, \dots, A_{p}\}$ of $\h$. 
Let $\{\sigma_j, j=p+1, \dots, n\}$ be  smooth horizontal sections of $TP$, $f_0$ a vertical vector field,
$\{ w_t^1, \dots, w_t^{N_1}, b_t^1, \dots, b_t^{N_2}\}$ independent one dimensional Brownian motions.
Let  $\{a_j^k, c_i^l\}$ be  a family of smooth functions on $P$.   For example, a computation analogous to that in Lemma \ref{lemma5.2}, should lead to a system of SDEs of Markovian type, whose solutions are slow and fast motions, and the following proposition.
\begin{proposition}
Let $u_t\equiv \phi_t(u_0)$ be a solution to
\begin{equation}
\label{sde10}
{\begin{aligned}
du_t=&(\sigma_0+f_0)(u_t)dt+\sum_{k=1}^{N_1} \sum_{j=p+1}^n (c_k^j \sigma_j)(u_t) \circ dw_t^k
+ \sum_{l=1}^{N_2} \sum_{j=1}^p \left( a_l^jA_j^*\right)(u_t)\circ db_t^l.
\end{aligned}}
\end{equation}
Set $x_t=\pi(u_t)$ and denote by $(\tilde x_t)$ its horizontal lift. Then $u_t=\tilde x_ta_t$ where
\begin{equation*}
{\begin{aligned}
d\tilde x_t=&(R_{a_t^{-1}})_*\sigma_0(\tilde x_t)dt +\sum_{k=1}^{N_2} \sum_{j=p+1}^n c_k^j (\tilde x_t a_t) (R_{a_t^{-1}})_*(\sigma_j)(\tilde x_t) \circ dw_t^k\\
da_t=&TR_{a_t}\left(\varpi_{\tilde x_t} (f_0)\right)dt +\sum_{l=1}^{N_1} \sum_{j=1}^p a_l^j(\tilde x_ta_t)A_j^*(a_t)\circ db_t^l.
\end{aligned}}
\end{equation*}
 \end{proposition}
If each $c_k^j$ vanishes identically the problem is easier  in which case we are led naturally  to random ODEs of the following type
 $$\begin{aligned}\dot y_t^\epsilon &=\sum_{k=1}^m Y_k(y_t^\epsilon) \alpha_k(z_t^\epsilon(\omega), y_t^\epsilon)\\
 y_0^\epsilon&=y_0.
 \end{aligned}$$
We may assume that the stochastic processes $(z_t^\epsilon)$ are hypoelliptic diffusions or  L\'evy processes satisfying a Birkhoff's ergodic theorem, and $\alpha_k: G\times M\to \R$ are smooth functions. The vector fields $\{Y_k\}$ are, for example, given by $Y_k(\cdot)=Y(\cdot)(e_k)$ where for each $y\in P$, $Y(y): \R^m\to T_yP$ is  isometric and $\{e_k\}$ is an orthonormal basis of $\R^m$. Letting $\epsilon$ approaches zero, we expect an {\it averaging principle} and  an effective ODE in the limit, c.f. \cite{Li-OM-1}. For the
 next step, the challenging problem is to find geometrically meaningful conditions on the coefficients so that the effective ODE is trivial, in which case one may proceed to obtain an effective diffusion on a larger time scale. We observe that the limit theorems in \cite{Li-limits} do not cover this.

\bigskip

{\bf Acknowledgment.}
I would like to acknowledge helpful conversations with 
K. Ardakov, R. L. Bryant, K. D. Elworthy, M. Hairer, J. Rawnsley and D. Rumynin.
 I also like  to thank  the Institute for Advanced Studies for its hospitality during my visit in 2014.

 \bibliographystyle{plain}
 \bibliography{Hopf}

\appendix
\addcontentsline{toc}{section}{Appendix by Dmitriy Rumynin: `Real Peter-Weyl'   }


\section*{Supplementary material}

# APPENDIX: REAL PETER-WEYL

## DMITRIY RUMYNIN

ABSTRACT. We explain how unitary representations of compact groups split into three types. We summarise their properties. As an application we state and prove Peter-Weyl Theorem for real representations of compact groups.

Let $G$ be a compact topological group. We refer the reader to textbooks [1, 2] for standard properties of $G$ and its representations. The group $G$ has a unique biinvariant Haar measure $\mu$ such that $\mu(G) = 1$. We utilise the corresponding integral $\int_G \Phi(x)\mu(dx)$.

We study unitary representations of $G$. It is a representation of $G$ on a real (or complex) Hilbert space $V$ (or $\mathcal{V}$ – we use calligraphic letters for complex Hilbert spaces) such that the action map $G \times V \to V$ is continuous and each $g \in G$ acts by an orthogonal (unitary) operator. A key example is the real regular representation $L^2(G)$, that is, the space of all $\mathbb{R}$-valued $L^2$-functions with the action $[x \cdot \Phi](y) = \Phi(yx)$. By $\mathcal{L}^2(G)$ we denote the complex regular representation, the space of all $\mathbb{C}$-valued $L^2$-functions with the same action.

Each unitary representation is a direct sum of irreducible unitary representations. Each irreducible unitary representation is finite dimensional. Let $\mathrm{Irr}_\mathbb{R} G$ be the set of isomorphism classes of irreducible unitary representations (correspondingly $\mathrm{Irr}_\mathbb{C} G$ for complex ones). The character of $V \in \mathrm{Irr}_\mathbb{R} G$ is the function $\chi_V \in L^2(G)$ given by $\chi_V(x) = \mathrm{Tr}_V(x)$. The Frobenius-Schur indicator of $\mathcal{V} \in \mathrm{Irr}_\mathbb{C} G$ (or $V \in \mathrm{Irr}_\mathbb{R} G$) is

$$\mathrm{FS}(\mathcal{V}) = \int_G \chi_\mathcal{V}(x^2)\mu(dx) \ \ (\text{or } \mathrm{FS}(V) = \int_G \chi_V(x^2)\mu(dx) \ ).$$

A complex unitary representation $(\mathcal{V}, \langle|\rangle)$ can be considered as a real vector space $\mathcal{V}_\mathbb{R}$. It is a real unitary representation with a form $\mathrm{Re}\langle|\rangle$. In the opposite direction, $V \otimes_\mathbb{R} \mathbb{C}$ is a complexification of $(V, \langle|\rangle)$ with a natural unitary form $\langle x \otimes \alpha | y \otimes \beta \rangle = \overline{\alpha}\beta \langle x|y\rangle$. Notice that $(V \otimes_\mathbb{R} \mathbb{C})_\mathbb{R} = V \otimes 1 \oplus V \otimes \mathbf{i} \cong 2V$ as real representations of $G$. This defines a correspondence

$$\mathrm{Irr}_\mathbb{R} G \leftrightarrow \mathrm{Irr}_\mathbb{C} G.$$

The representations $V$ and $\mathcal{V}$ are in correspondence if $\mathcal{V}$ is a direct summand of $V \otimes_\mathbb{R} \mathbb{C}$. By Frobenius reciprocity, this is equivalent to $V$ being a direct summand of $\mathcal{V}_\mathbb{R}$. The properties of this correspondence depend on the endomorphism algebra $\mathbb{F} = \mathrm{End}_G(V)$. Recall that $\mathbb{F}$ consists of all $\mathbb{R}$-linear operators on $V$ that commute with all elements of $G$. The algebra $\mathbb{F}$ is a division algebra (Schur lemma), hence, it must be either real numbers $\mathbb{R}$, or complex numbers $\mathbb{C}$, or quaternions $\mathbb{H}$ (Frobenius Theorem). Depending on what $\mathbb{F}$ is, we call each of the representations in correspondence $V \leftrightarrow \mathcal{V}$ *a representation of real, complex or quaternionic type* correspondingly.

Useful properties of representations of each type are summarised in Table 1. The last two rows show that the type can be determined by computing the Frobenius-Schur indicator: it is particularly useful for a finite group $G$.

The third and fourth rows from the bottom are useful for a simple compact Lie group $G$. An irreducible representation $\mathcal{V}$ of $G$ is parametrised by its highest weight $\varpi$: we write $\mathcal{V} = \mathcal{L}(\varpi)$, if $\mathcal{V}$ is an irreducible representation with a highest weight vector of weight $\varpi$. It is known that $\mathcal{L}(\varpi)^* = \mathcal{L}(-w \cdot \varpi)$ where $w$ is the longest element of the Weyl group. The fourth row immediately tells you whether $V \leftrightarrow \mathcal{V}$ is of complex type or not: it is easy to verify whether $\varpi = -w \cdot \varpi$ or not in each particular case. For instance, the vector representation of $\mathrm{SU}_n(\mathbb{C})$ is $\mathcal{L}(\varpi_1)$ where $\varpi_1$ is the first fundamental weight. Then $\mathcal{L}(\varpi_1)^* = \mathcal{L}(-w \cdot \varpi_1) = \mathcal{L}(\varpi_{n-1})$ so that $\mathcal{L}(\varpi_1)$ is of complex type if $n > 2$.

Look at $\mathrm{SU}_2(\mathbb{C})$. Its irreducible representations are $\mathcal{L}(m\varpi_1)$ for natural $m$ and $\mathcal{L}(m\varpi_1)^* = \mathcal{L}(-w \cdot m\varpi_1) = \mathcal{L}(m\varpi)$. No representation of $\mathrm{SU}_2(\mathbb{C})$ is of complex type. The vector representation $\mathcal{L}(\varpi_1)$ admits a symplectic form preserved by $\mathrm{SU}_2(\mathbb{C})$. The representation $\mathcal{L}(m\varpi_1)$ is isomorphic to $S^m(\mathcal{L}(\varpi_1))$, its $m$-th symmetric power, that gets an induced $\mathrm{SU}_2(\mathbb{C})$-invariant form, symmetric if $m$ is even, skew-symmetric if $m$ is odd. Thus, $\mathcal{L}(m\varpi_1)$ is of real type if $m$ is even and is of quaternionic type if $m$ is odd.

The first row helps to determine the type of the adjoint representation $\mathfrak{g}$ of a compact simple Lie group $G$. Its complexification $\mathfrak{g} \otimes_{\mathbb{R}} \mathbb{C}$ is a simple complex Lie algebra. So it is of real type. Alternatively one can use its Killing form or the fact that $\mathfrak{g} \otimes_{\mathbb{R}} \mathbb{C} = \mathcal{L}(\alpha_0)$ where $\alpha_0$ is the highest root.

The field $\mathbb{F}$ is a "natural" field for $V$. We will write elements of $\mathbb{F}$ acting on vectors from the right (it is essential since $\mathbb{H}$ is non-commutative).

**Lemma 1.** *(Real Artin-Wedderburn Formula)* $L^2(G) \cong \bigoplus_{V \in \mathrm{Irr}_{\mathbb{R}} G} \dim_{\mathbb{F}}(V) V$.

*Proof.* Notice that $L^2(G) \otimes_{\mathbb{R}} \mathbb{C} \cong \mathcal{L}^2(G) \cong \bigoplus_{\mathcal{V} \in \mathrm{Irr}_{\mathbb{C}} G} \dim_{\mathbb{C}}(\mathcal{V})\mathcal{V}$. Thus, we need to prove that $\dim_{\mathbb{F}}(V)V \otimes_{\mathbb{R}} \mathbb{C}$ is isomorphic to $\bigoplus_{V \leftrightarrow \mathcal{V}} \dim_{\mathbb{C}}(\mathcal{V})\mathcal{V}$. This is easily checked case by case for each column in Table 1. $\square$

We fix standard imaginary units $\mathbf{i}, \mathbf{j}, \mathbf{k} \in \mathbb{H}$ and $\mathbf{i} \in \mathbb{C}$. The complex conjugation on $\mathbb{F}$ is given by the formula $\overline{\alpha + \beta\mathbf{i} + \gamma\mathbf{j} + \delta\mathbf{k}} = \alpha - \beta\mathbf{i} - \gamma\mathbf{j} - \delta\mathbf{k}$. Note that it is identity if $\mathbb{F} = \mathbb{R}$. Given a $\mathbb{F}$-valued function $\Phi$, we want to fix its one, two or four (depending on $\mathbb{F}$) components, the $\mathbb{R}$-valued functions $\Phi^{\bullet}$ such that

$$\Phi = \Phi^1 1 + \Phi^{\mathbf{i}} \mathbf{i} + \Phi^{\mathbf{j}} \mathbf{j} + \Phi^{\mathbf{k}} \mathbf{k}.$$

Now we are ready to exhibit an explicit orthonormal basis of $L^2(G)$, e.g. state Peter-Weyl Theorem. Let $(V, \langle | \rangle) \in \mathrm{Irr}_{\mathbb{R}} G$, $d = \dim_{\mathbb{F}}(V)$. Consider $V$ as a right vector space over $\mathbb{F}$. Choose a vector space isomorphism with the space of columns

$$V \to \mathbb{F}^d, a \mapsto \underline{a}.$$

This gives a $G$-invariant positive definite unitary $\mathbb{F}$-form on $V$:

$$[a|b] = \int_G (\underline{x \cdot a})^* \underline{x \cdot b} \, \mu(dx)$$

where $(\underline{x \cdot a})^*$ is the conjugate-transpose row of the column $\underline{x \cdot a}$. Observe that $[a\mathbf{q}|b\mathbf{r}] = \overline{\mathbf{q}}[a|b]\mathbf{r}$ for all $\mathbf{q}, \mathbf{r} \in \mathbb{F}$.

**Lemma 2.** *(Real Schur Lemma) Let $\{|\}$ be a $G$-invariant unitary $\mathbb{F}$-form on $V$. Then there exists $\mathbf{r} \in Z(\mathbb{F})$ (the centre of $\mathbb{F}$) such that*

$$\{a|b\} = \mathbf{r}[a|b] \quad \text{for all} \quad a, b \in V.$$

*Let $\{|\}^\sharp$ be a $G$-invariant bilinear $\mathbb{R}$-form on $V$. Then there exists $\mathbf{q} \in \mathbb{F}$ such that*

$$\{a|b\}^\sharp = \langle a|b\mathbf{q}\rangle \quad \text{for all} \quad a, b \in V.$$

*Finally let $\{|\}^\flat$ be a $G$-invariant symmetric bilinear $\mathbb{R}$-form on $V$. Then there exists $\mathbf{r} \in \mathbb{R}$ such that*

$$\{a|b\}^\flat = \mathbf{r}\langle a|b\rangle \quad \text{for all} \quad a, b \in V.$$

*Proof.* The linear operator $\theta : V \to V$ defined by the equation

$$\{a|b\} = [a|\theta(b)] \text{ for all } a, b \in V$$

is a $G$-module endomorphism since both form are $G$-invariant. Thus, $\theta(a) = a\mathbf{r}$ for some $\mathbf{r} \in \mathbb{F}$ and $\{a|b\} = [a|b]\mathbf{r}$. The unitarity of $\{|\}$ ensures that $\mathbf{r}$ is in the centre.

A proof for $\{|\}^\flat$ is similar.

Using the second statement we can find $\mathbf{q} \in \mathbb{F}$ such that $\{a|b\}^\flat = \langle a|b\mathbf{q}\rangle$ for all $a, b$. Since $\{a|b\}^\flat$ is symmetric, $\mathbf{q}$ is self-adjoint:

$$\langle a|b\mathbf{q}\rangle = \{a|b\}^\flat = \{b|a\}^\flat = \langle b|a\mathbf{q}\rangle = \langle a\mathbf{q}|b\rangle.$$

It remains to notice that the adjoint of $\mathbf{q}$ is $\overline{\mathbf{q}}$. This proves that $\mathbf{q} = \overline{\mathbf{q}}$ and $\mathbf{q} \in \mathbb{R}$. □

Notice that the component $[|]^1$ is a $G$-invariant positive definite symmetric bilinear $\mathbb{R}$-form. Lemma 2 implies that $[|]^1 = \langle|\rangle\mathbf{r}$ for some positive $\mathbf{r} \in \mathbb{R}_{>0}$. Replacing $[|]$ with $[|]\mathbf{r}^{-1}$, we get useful properties connecting the two forms.

**Lemma 3.** *If the relevant elements of $\mathbb{F}$ exist, then the following identities hold for all $a, b \in V$:*

$$\langle a|b\rangle = [a|b]^1 = \langle a\mathbf{i}|b\mathbf{i}\rangle = \langle a\mathbf{j}|b\mathbf{j}\rangle = \langle a\mathbf{k}|b\mathbf{k}\rangle,$$

$$\langle a\mathbf{i}|b\rangle = [a|b]^\mathbf{i} = -\langle a|b\mathbf{i}\rangle = -\langle a\mathbf{j}|b\mathbf{k}\rangle = \langle a\mathbf{k}|b\mathbf{j}\rangle,$$

$$-\langle a|b\mathbf{j}\rangle = [a|b]^\mathbf{j} = \langle a\mathbf{i}|b\mathbf{k}\rangle = -\langle a\mathbf{k}|b\mathbf{i}\rangle = \langle a\mathbf{j}|b\rangle,$$

$$-\langle a|b\mathbf{k}\rangle = [a|b]^\mathbf{k} = -\langle a\mathbf{i}|b\mathbf{j}\rangle = \langle a\mathbf{j}|b\mathbf{i}\rangle = \langle a\mathbf{k}|b\rangle.$$

*Proof.* The first identity in the first row is a result of our assumptions. Now $[a\mathbf{i}|b] = [a\mathbf{i}|b]^1 1 + [a\mathbf{i}|b]^\mathbf{i}\mathbf{i} + [a\mathbf{i}|b]^\mathbf{j}\mathbf{j} + [a\mathbf{i}|b]^\mathbf{k}\mathbf{k}$, while at the same time $[a\mathbf{i}|b] = -\mathbf{i}[a|b] = [a|b]^\mathbf{i}1 - [a|b]^1\mathbf{i} + [a|b]^\mathbf{k}\mathbf{j} - [a|b]^\mathbf{j}\mathbf{k}$. This proves the first identity in the second row. The remaining proofs are similar using properties such as $[a\mathbf{j}|b\mathbf{k}] = -\mathbf{j}[a|b]\mathbf{k}$. □

Let $X_1, X_2 \ldots X_m$ be an orthonormal (with respect to $[|]$) $\mathbb{F}$-basis of $V$. The scaled coordinate functions in this basis are $\mathbb{F}$-valued functions on $G$

$$^V\Phi_{n,l} = {^V\Phi_{n,l}^1}1 + {^V\Phi_{n,l}^\mathbf{i}}\mathbf{i} + {^V\Phi_{n,l}^\mathbf{j}}\mathbf{j} + {^V\Phi_{n,l}^\mathbf{k}}\mathbf{k}, \quad {^V\Phi_{n,l}}(g) = \sqrt{\dim_\mathbb{R}(V)}[X_n|g \cdot X_l].$$

**Theorem 1.** *(Peter-Weyl Theorem) The components of the scaled coordinate functions $\{{^V\Phi_{n,l}^\bullet} \mid V \in \operatorname{Irr}_\mathbb{R} G\}$ form an orthonormal basis of $L^2(G)$.*

*Proof.* We deduce this theorem from the usual (complex) Peter-Weyl Theorem. This states that by equipping each $\mathcal{V} \in \mathrm{Irr}_\mathbb{C} G$ with an orthonormal basis $Y_1, Y_2 \cdots$ we obtain an orthonormal basis of $\mathcal{L}^2(G)$ consisting of scaled coordinate functions

$$\{{}^\mathcal{V}\Psi_{n,l} \mid \mathcal{V} \in \mathrm{Irr}_\mathbb{C} G\} \text{ where } {}^\mathcal{V}\Psi_{n,l}(g) = \sqrt{\dim_\mathbb{C}(\mathcal{V})}\langle Y_n | g \cdot Y_k \rangle.$$

If we fix $\mathcal{V}$, the functions ${}^\mathcal{V}\Psi_{n,l}$ form an orthonormal basis of $\dim_\mathbb{C}(\mathcal{V})\mathcal{V}$, a direct summand of $\mathcal{L}^2(G)$. It follows that ${}^V\Phi_{n,l}^\bullet$ and ${}^W\Phi_{n,l}^\bullet$ for non-isomorphic $V$ and $W$ are orthogonal to each other. It remains to observe that ${}^V\Phi_{n,l}^\bullet$ form an orthonormal basis of $\dim_\mathbb{F}(V)V$. Let $d = \dim_\mathbb{R}(V)$.

If $\mathbb{F} = \mathbb{R}$, we can choose $Y_n = X_n \otimes 1$ as an orthonormal basis of $\mathcal{V} = V \otimes_\mathbb{R} \mathbb{C}$. Since $d = \dim_\mathbb{C}(\mathcal{V})$, ${}^V\Phi_{n,l}^1 = {}^\mathcal{V}\Psi_{n,l}$ form an orthonormal basis.

If $\mathbb{F} = \mathbb{C}$, we have two complex representations $\mathcal{V}_+$ and $\mathcal{V}_- \cong \mathcal{V}_+^*$ that we can explicitly pinpoint, writing the two idempotents of $\mathrm{End}_G(V \otimes_\mathbb{R} \mathbb{C}) = \mathbb{F} \otimes_\mathbb{R} \mathbb{C} \cong \mathbb{C}^2$:

$$V \otimes_\mathbb{R} \mathbb{C} = \mathcal{V}_- \oplus \mathcal{V}_+ \text{ where } \mathcal{V}_\pm = (V \otimes_\mathbb{R} \mathbb{C})\mathbf{e}_\pm, \quad \mathbf{e}_\pm = \frac{1}{2}(1 \otimes 1 \pm \mathbf{i} \otimes \mathbf{i}).$$

Now we choose $Y_n^\pm = \sqrt{2}X_n\mathbf{e}_\pm$ as an orthonormal basis of $\mathcal{V}_\pm$. It is clear that $\langle Y_n^+ | Y_l^- \rangle = 0$, while inside $\mathcal{V}_\pm$

$$\langle Y_n^\pm | Y_k^\pm \rangle = \frac{1}{2}\langle X_n \otimes 1 \pm X_n\mathbf{i} \otimes \mathbf{i} | X_l \otimes 1 \pm X_l\mathbf{i} \otimes \mathbf{i} \rangle = \frac{1}{2}\langle X_n | X_l \rangle \mp \frac{\mathbf{i}}{2}\langle X_n\mathbf{i} | X_l \rangle \pm$$

$$\pm \frac{\mathbf{i}}{2}\langle X_n | X_l\mathbf{i} \rangle + \frac{1}{2}\langle X_n\mathbf{i} | X_l\mathbf{i} \rangle = \frac{1}{2}[X_n | X_l]^1 \mp \frac{\mathbf{i}}{2}[X_n | X_l]^\mathbf{i} \mp \frac{\mathbf{i}}{2}[X_n | X_l]^\mathbf{i} + \frac{1}{2}[X_n | X_l]^1 =$$

$$= [X_n | X_l]^1 \mp \mathbf{i}[X_n | X_l]^\mathbf{i} = [X_n | X_l] \text{ ( or } \overline{[X_n | X_l]} \text{ )} = \delta_{n,l}.$$

In this case $d = 2\dim_\mathbb{C} \mathcal{V}_\pm$. We express the functions now:

$${}^{\mathcal{V}_\pm}\Psi_{n,l}(g) = \sqrt{\dim_\mathbb{C}(\mathcal{V}_\pm)}\langle Y_n^\pm | g \cdot Y_k^\pm \rangle = \frac{\sqrt{d}}{2\sqrt{2}}\langle X_n \otimes 1 \pm X_n\mathbf{i} \otimes \mathbf{i} | g \cdot X_l \otimes 1 \pm g \cdot X_l\mathbf{i} \otimes \mathbf{i} \rangle =$$

$$= \frac{\sqrt{d}}{2\sqrt{2}}\langle X_n | g \cdot X_l \rangle \mp \frac{\mathbf{i}\sqrt{d}}{2\sqrt{2}}\langle X_n\mathbf{i} | g \cdot X_l \rangle \pm \frac{\mathbf{i}\sqrt{d}}{2\sqrt{2}}\langle X_n | g \cdot X_l\mathbf{i} \rangle + \frac{\sqrt{d}}{2\sqrt{2}}\langle X_n\mathbf{i} | g \cdot X_l\mathbf{i} \rangle =$$

$$= \frac{\sqrt{d}}{2\sqrt{2}}[X_n | g \cdot X_l]^1 \mp \frac{\mathbf{i}\sqrt{d}}{2\sqrt{2}}[X_n | g \cdot X_l]^\mathbf{i} \mp \frac{\mathbf{i}\sqrt{d}}{2\sqrt{2}}[X_n | g \cdot X_l]^\mathbf{i} + \frac{\sqrt{d}}{2\sqrt{2}}[X_n | g \cdot X_l]^1$$

$$= \frac{1}{\sqrt{2}}{}^V\Phi_{n,l}^1(g) \mp \frac{\mathbf{i}}{\sqrt{2}}{}^V\Phi_{n,l}^\mathbf{i}(g).$$

Orthonormality of ${}^V\Phi_{n,l}^\bullet$ follows from explicit formulas:

$${}^V\Phi_{n,l}^1 = \frac{1}{\sqrt{2}}{}^{\mathcal{V}_+}\Psi_{n,l} + \frac{1}{\sqrt{2}}{}^{\mathcal{V}_-}\Psi_{n,l}, \quad {}^V\Phi_{n,l}^\mathbf{i} = \frac{\mathbf{i}}{\sqrt{2}}{}^{\mathcal{V}_+}\Psi_{n,l} - \frac{\mathbf{i}}{\sqrt{2}}{}^{\mathcal{V}_-}\Psi_{n,l}.$$

If $\mathbb{F} = \mathbb{H}$, choose a rank 1 idempotent in $\mathrm{End}_G(V \otimes_\mathbb{R} \mathbb{C}) = \mathbb{F} \otimes_\mathbb{R} \mathbb{C} \cong M_2(\mathbb{C})$:

$$V \otimes_\mathbb{R} \mathbb{C} = \mathcal{V} \oplus \mathcal{V}^\perp \text{ where } \mathcal{V}^\perp \cong \mathcal{V}, \quad \mathcal{V} = (V \otimes_\mathbb{R} \mathbb{C})\mathbf{e}, \quad \mathbf{e} = \frac{1}{2}(1 \otimes 1 + \mathbf{i} \otimes \mathbf{i}).$$

Now we choose $Y_n^+ = \sqrt{2}X_n\mathbf{e}$, $Y_n^- = \sqrt{2}X_n\mathbf{j}\mathbf{e}$, as an orthonormal basis of $\mathcal{V}$. The calculation in the case of $\mathbb{F} = \mathbb{C}$ makes it clear that $\langle Y_n^+ | Y_n^+ \rangle = 1$ and $\langle Y_n^\pm | Y_k^\bullet \rangle = 0$

if $n \neq k$. Let us do the remaining calculations:

$$\langle Y_n^-|Y_n^-\rangle = \frac{1}{2}\langle X_n\mathbf{j}\otimes 1 - X_n\mathbf{k}\otimes \mathbf{i}|X_n\mathbf{j}\otimes 1 - X_n\mathbf{k}\otimes \mathbf{i}\rangle = \frac{1}{2}\langle X_n\mathbf{j}|X_n\mathbf{j}\rangle - \frac{\mathbf{i}}{2}\langle X_n\mathbf{j}|X_n\mathbf{k}\rangle +$$
$$+ \frac{\mathbf{i}}{2}\langle X_n\mathbf{k}|X_n\mathbf{j}\rangle + \frac{1}{2}\langle X_n\mathbf{k}|X_n\mathbf{k}\rangle = [X_n|X_n]^1 + \mathbf{i}[X_n|X_n]^{\mathbf{i}} = 1,$$
$$\langle Y_n^+|Y_n^-\rangle = \frac{1}{2}\langle X_n\otimes 1 + X_n\mathbf{i}\otimes \mathbf{i}|X_n\mathbf{j}\otimes 1 - X_n\mathbf{k}\otimes \mathbf{i}\rangle = \frac{1}{2}\langle X_n|X_n\mathbf{j}\rangle - \frac{\mathbf{i}}{2}\langle X_n|X_n\mathbf{k}\rangle -$$
$$-\frac{\mathbf{i}}{2}\langle X_n\mathbf{i}|X_n\mathbf{j}\rangle - \frac{1}{2}\langle X_n\mathbf{i}|X_n\mathbf{k}\rangle = -[X_n|X_n]^{\mathbf{j}} + \mathbf{i}[X_n|X_n]^{\mathbf{k}} = 0.$$

Let us optimise the notation by $\Psi_{n,l}^{\pm?} = {}^{\mathcal{V}}\Psi_{n^\pm,l^?}$. The calculation

$$\Psi_{n,l}^{++} = \frac{1}{\sqrt{2}}{}^V\Phi_{n,l}^1(g) - \frac{\mathbf{i}}{\sqrt{2}}{}^V\Phi_{n,l}^{\mathbf{i}}(g)$$

is as in the complex case. Again $d = 2\dim_{\mathbb{C}}\mathcal{V}$. Let us calculate the remaining functions.

$$\Psi_{n,l}^{--}(g) = \frac{\sqrt{d}}{2\sqrt{2}}\langle X_n\mathbf{j}\otimes 1 - X_n\mathbf{k}\otimes \mathbf{i}|g\cdot X_l\mathbf{j}\otimes 1 - g\cdot X_l\mathbf{k}\otimes \mathbf{i}\rangle =$$
$$= \frac{\sqrt{d}}{2\sqrt{2}}\langle X_n\mathbf{j}|g\cdot X_l\mathbf{j}\rangle - \frac{\mathbf{i}\sqrt{d}}{2\sqrt{2}}\langle X_n\mathbf{j}|g\cdot X_l\mathbf{k}\rangle + \frac{\mathbf{i}\sqrt{d}}{2\sqrt{2}}\langle X_n\mathbf{k}|g\cdot X_l\mathbf{j}\rangle + \frac{\sqrt{d}}{2\sqrt{2}}\langle X_n\mathbf{k}|g\cdot X_l\mathbf{k}\rangle =$$
$$= \frac{1}{\sqrt{2}}{}^V\Phi_{n,l}^1(g) + \frac{\mathbf{i}}{\sqrt{2}}{}^V\Phi_{n,l}^{\mathbf{i}}(g),$$

$$\Psi_{n,l}^{+-}(g) = \frac{\sqrt{d}}{2\sqrt{2}}\langle X_n\otimes 1 + X_n\mathbf{i}\otimes \mathbf{i}|g\cdot X_l\mathbf{j}\otimes 1 - g\cdot X_l\mathbf{k}\otimes \mathbf{i}\rangle =$$
$$= \frac{\sqrt{d}}{2\sqrt{2}}\langle X_n|g\cdot X_l\mathbf{j}\rangle - \frac{\mathbf{i}\sqrt{d}}{2\sqrt{2}}\langle X_n|g\cdot X_l\mathbf{k}\rangle - \frac{\mathbf{i}\sqrt{d}}{2\sqrt{2}}\langle X_n\mathbf{i}|g\cdot X_l\mathbf{j}\rangle - \frac{\sqrt{d}}{2\sqrt{2}}\langle X_n\mathbf{i}|g\cdot X_l\mathbf{k}\rangle =$$
$$= \frac{-1}{\sqrt{2}}{}^V\Phi_{n,l}^{\mathbf{j}}(g) + \frac{\mathbf{i}}{\sqrt{2}}{}^V\Phi_{n,l}^{\mathbf{k}}(g),$$

$$\Psi_{n,l}^{-+}(g) = \frac{\sqrt{d}}{2\sqrt{2}}\langle X_n\mathbf{j}\otimes 1 - X_n\mathbf{k}\otimes \mathbf{i}|g\cdot X_l\otimes 1 + g\cdot X_l\mathbf{i}\otimes \mathbf{i}\rangle =$$
$$= \frac{\sqrt{d}}{2\sqrt{2}}\langle X_n\mathbf{j}|g\cdot X_l\rangle + \frac{\mathbf{i}\sqrt{d}}{2\sqrt{2}}\langle X_n\mathbf{j}|g\cdot X_l\mathbf{i}\rangle + \frac{\mathbf{i}\sqrt{d}}{2\sqrt{2}}\langle X_n\mathbf{k}|g\cdot X_l\rangle - \frac{\sqrt{d}}{2\sqrt{2}}\langle X_n\mathbf{k}|g\cdot X_l\mathbf{i}\rangle =$$
$$= \frac{1}{\sqrt{2}}{}^V\Phi_{n,l}^{\mathbf{j}}(g) + \frac{\mathbf{i}}{\sqrt{2}}{}^V\Phi_{n,l}^{\mathbf{k}}(g).$$

Orthonormality of ${}^V\Phi_{n,l}^\bullet$ follows from explicit formulas:

$${}^V\Phi_{n,l}^1 = \frac{1}{\sqrt{2}}\Psi_{n,l}^{++} + \frac{1}{\sqrt{2}}\Psi_{n,l}^{--}, \quad {}^V\Phi_{n,l}^{\mathbf{i}} = \frac{\mathbf{i}}{\sqrt{2}}\Psi_{n,l}^{++} - \frac{\mathbf{i}}{\sqrt{2}}\Psi_{n,l}^{--},$$
$${}^V\Phi_{n,l}^{\mathbf{j}} = \frac{1}{\sqrt{2}}\Psi_{n,l}^{-+} - \frac{1}{\sqrt{2}}\Psi_{n,l}^{+-}, \quad {}^V\Phi_{n,l}^{\mathbf{k}} = -\frac{\mathbf{i}}{\sqrt{2}}\Psi_{n,l}^{-+} - \frac{\mathbf{i}}{\sqrt{2}}\Psi_{n,l}^{+-}.$$

□

TABLE 1. Properties of $V \leftrightarrow \mathcal{V}$ depending on $\mathbb{F}$

| $\mathbb{F} = \mathrm{End}_G(V)$ | $\mathbb{R}$ | $\mathbb{C}$ | $\mathbb{H}$ |
|---|---|---|---|
| $V_{\mathbb{C}} \cong \mathcal{V}$ | Yes | No | No |
| $V_{\mathbb{C}}$ | $\mathcal{V}$ | $\mathcal{V} \oplus \mathcal{V}^*$ | $2\mathcal{V}$ |
| $\mathrm{Irr}_{\mathbb{R}} G$ to $\mathrm{Irr}_{\mathbb{C}} G$ correspondence is | 1:1 | 1:2 | 1:1 |
| $\mathcal{V}_{\mathbb{R}}$ | $2V$ | $V$ | $V$ |
| $\chi_{\mathcal{V}}$ takes values in | $\mathbb{R}$ | $\mathbb{C}$ | $\mathbb{R}$ |
| relation between characters | $\chi_V = \chi_{\mathcal{V}}$ | $\chi_V = 2 \cdot \mathrm{Re}(\chi_{\mathcal{V}})$ | $\chi_V = 2 \cdot \chi_{\mathcal{V}}$ |
| $\mathcal{V} \cong \mathcal{V}^*$ | Yes | No | Yes |
| $G$-invariant form on $\mathcal{V}$ | symmetric | only zero | skew-symmetric |
| $\mathrm{FS}(\mathcal{V})$ is equal to | 1 | 0 | -1 |
| $\mathrm{FS}(V)$ is equal to | 1 | 0 | -2 |

DEPARTMENT OF MATHEMATICS, UNIVERSITY OF WARWICK, COVENTRY, CV4 7AL, UK
*E-mail address*: `D.Rumynin@warwick.ac.uk`


\endappendix

\end{document}